\pdfoutput=1 
\documentclass[10pt]{article}

\usepackage{fullpage,amsfonts,amsmath,amssymb,amsopn,amsthm,mdwlist,wasysym,url,graphicx,verbatim,bussproofs,mathrsfs,leftidx,setspace}
\usepackage[cns]{llconnex}		

\theoremstyle{plain}
\newtheorem{theorem}{Theorem}
\newtheorem*{theoremnonumber}{Theorem}

\theoremstyle{definition}
\newtheorem{observation}{Observation}
\newtheorem{definition}{Definition}

\newtheorem{lemma}{Lemma}
\newtheorem{proposition}{Proposition}
\newtheorem{corollary}{Corollary}

\newcommand{\monoidal}{\varphi}
\newcommand{\comonoidal}{\psi}
\newcommand{\monbin}{\varphi}		
\newcommand{\monnul}{\varphi_0}		
\newcommand{\combin}{\psi}		
\newcommand{\comnul}{\psi_0}		

\newcommand{\boot}{\start}		

\newcommand{\unit}{\tau}
\newcommand{\counit}{\gamma}

\input diagxy

\let\mod\undefined
\DeclareMathOperator{\mod}{mod}

\newcommand{\V}{\mathcal V}
\newcommand{\W}{\mathcal W}

\newcommand{\KK}{\mathbb K}

\newcommand{\Vect}{{\bf Vec}}
\newcommand{\Hmod}{H\text-{\bf mod}}
\newcommand{\Jmod}{J\text-{\bf mod}}
\newcommand{\Mod}{{\bf Mod}}
\newcommand{\Cat}{{\bf Cat}}

\newcommand{\ba}{{\bf ba}\,}
\newcommand{\ha}{{\bf ha}\,}
\newcommand{\alg}{{\bf alg}\,}
\newcommand{\wba}{{\bf wba}\,}
\newcommand{\wha}{{\bf wha}\,}
\newcommand{\fm}{{\bf fmon}\,}
\newcommand{\sfm}{{\bf sfmon}}
\newcommand{\str}{{\bf strmon}}

\newcommand{\too}{\longrightarrow}

\newcommand{\Too}{\Longrightarrow}

\newcommand{\oto}[1]{\stackrel{\underrightarrow{\stackrel{#1}{\,\,\,\,\hphantom{\too}}}}{}}  

\newcommand{\loto}[1]{\stackrel{\underleftarrow{\stackrel{#1}{\,\,\,\,\hphantom{\too}}}}{}}  
\newcommand{\ladj}{\dashv}
\newcommand{\op}{{\rm op}}

\newcommand{\rulaw}{\simeq}
\newcommand{\lulaw}{\simeq}

\newcommand{\fullwidth}[1]{\widthimage{0.89}{#1}}
\newcommand{\widthimage}[2]{\begin{center}\includegraphics[width=#1\textwidth]{#2}\end{center}}

\newcommand{\widthfigure}[4]{
	\begin{figure}[htbp]
	\begin{center}\includegraphics[width=#1\textwidth]{#2}\end{center}
	\caption{#3}
	\label{#4}
	\end{figure}}

\newcommand{\fullfigure}[3]{
	\begin{figure}[htbp]
	\begin{center}\includegraphics[height=0.905\textheight]{#1}\end{center}
	\caption{#2}
	\label{#3}
	\end{figure}}

\newcommand{\fullwidthfigure}[3]{\widthfigure{0.95}{#1}{#2}{#3}}

\newcommand{\pic}[1]{\begin{array}{c} \includegraphics[scale=0.1]{#1} \end{array}}
\newcommand{\scaledpic}[2]{\begin{array}{c} \includegraphics[scale=#2]{#1} \end{array}}

\newcommand{\adjunction}[4]{

\node ONE(-500,0)[#1]
\node TWO(+500,0)[#2]

\arrow|a|/{@{>}@/^1em/}/[ONE`TWO;#3]
\arrow|b|/{@{>}@/^1em/}/[TWO`ONE;#4]
\place(0,0)[\bot]
}

\graphicspath{{tannaka/},{cyclicity/},{graphical/}}

\usepackage{fullpage,amsfonts,amsmath,amssymb,amsopn,amsthm,mdwlist,wasysym,url,graphicx,verbatim}
\usepackage[cns]{llconnex}	

\newcommand{\recon}[1]{\sfm\!\sswarrow\nolinebreak{#1}}
\newcommand{\reconstar}[1]{\sfm\!^*\!\!\sswarrow{#1}}
\newcommand{\strongrecon}[1]{\str\!\sswarrow{#1}}
\newcommand{\strongreconstar}[1]{\str\!^*\!\!\sswarrow{#1}}

\author{Micah Blake McCurdy\footnote{Saint Mary's University. Partially supported by a Macquarie University Research Excellence Scholarship
(MQRES). This work is substantially the same as parts of the third chapter of the author's Macquarie University Doctoral Thesis, ``Cyclic Star-autonomous Categories and
the Tannaka Construction via Graphical Methods'', completed July 2011}}
\date{12 October 2011}	

\title{Graphical Methods for Tannaka Duality of Weak Bialgebras and Weak Hopf Algebras in Arbitrary Braided Monoidal Categories}

\begin{document}

\maketitle

\abstract{\noindent Tannaka Duality describes the relationship between algebraic objects in a given category and their representations; an important case is that of
Hopf algebras and their categories of representations; these have strong monoidal forgetful ``fibre functors'' to the category of vector spaces.
We simultaneously generalize the theory of Tannaka duality in two ways: first, we replace Hopf algebras with \emph{weak Hopf algebras} and strong monoidal functors
with \emph{separable Frobenius monoidal functors}; second, we replace the category of vector spaces with an arbitrary braided monoidal category. To accomplish this
goal, we introduce a new graphical notation for functors between monoidal categories, using string diagrams with \emph{coloured regions}. Not only does this notation
extend our capacity to give simple proofs of complicated calculations, it makes plain some of the connections between Frobenius monoidal or separable Frobenius monoidal
functors and the topology of the axioms defining certain algebraic structures. Finally, having generalized Tannaka to an arbitrary base category, we briefly discuss
the functoriality of the construction as this base is varied.}

\section{Introduction} 
Broadly speaking, Tannaka duality describes the relationship between algebraic objects and their representations; for an excellent introduction,  
see the survey of Joyal and Street~\cite{JoyalStreet90}. On the one hand, given an algebraic object $H$ in a monoidal category $\V$ (for instance, a Hopf algebra
in the category 
$\Vect_k$ of vector spaces over a field $k$), one can consider the functor which takes the algebraic object to its category of 
representations, $\Hmod$, equipped with its canonical forgetful functor back to $\V$. This process is \emph{representation} and it can be defined in a 
great variety of situations, with very mild assumptions on $\V$. 


 
On the other hand, given a suitable functor $F \colon A \too \V$, we can try to use the properties of $F$ (which of course include those of $A$ and $\V$) 
to build an algebraic object in $\V$; this is what we call \emph{the Tannaka construction}. Historically, the algebraic objects have been considered  
primitive, and this process was called ``reconstruction''; but since it can be considered as an independent question, we discard the prefix ``re-''.

The classical paper of Tannaka~\cite{Tannaka38} describes the reconstruction of a compact group from its representations, and is the starting point for the theory
which bears his name. Crucially, for a given algebraic object, the forgetful functor from its category of representations to $\Vect_k$ is considered the starting
point for the project of reconstruction---such functors are known as ``fibre functors''.

Reconstruction requires more stringent assumptions on $F$---certainly $\V$ must be braided; objects in the image of $F$ must have duals; 
and $\V$ must admit certain ends or coends which cohere with the monoidal structure.

In this paper, we show that the theory of Tannaka duality can be extended to an adjunction between a suitable category of \emph{separable Frobenius monoidal functors} into
an arbitrary base category $\V$ and a suitable category of \emph{weak bialgebras} in $\V$. We describe the restriction of this adjunction to \emph{weak Hopf algebras}; and
we show that our constructions coincide with the existing theory of Tannaka duality where applicable. In a sequel~\cite{McCurdy11Tannaka2} to the present paper, we will 
show that this theory can
be refined to include various sorts of structured weak bialgebras and their correspondingly structured (generalized) fibre functors.

\subsection{Existing work}
Many people have devoted considerable effort to various versions of the Tannaka construction, at various levels of generality. Mostly, attention has been confined to
fibre functors which are \emph{strong monoidal} and which have codomain $\V = \Vect_k$. A landmark paper is that of Ulbrich~\cite{Ulbrich90}, who showed that one can
obtain a Hopf algebra from a strong monoidal functor $A \to \Vect_k$, where $A$ is an autonomous-but-not-necessarily-symmetric monoidal category. The case of 
not-necessarily-coherent strong monoidal functors into $\Vect_k$ has been shown by Majid~\cite{Majid92} to result in a quasi-Hopf algebra in the sense of 
Drinfeld~\cite{Drinfeld89} this was
extended by H\"aring~\cite{Haring-Oldenburg97} to cover the case of not-necessary-coherent weak monoidal functors into $\Vect_k$. The reader should note that the sense of ``weak'' Hopf
algebra in \cite{Haring-Oldenburg97} is slightly different from that of B\"ohm, Nill, and Szlach\'anyi~\cite{BohmNillSzlachanyi99} (whom we follow here); but the core 
idea is the same---namely, that ``weak'' Hopf algebras should be bialgebras in which the unit is not strictly grouplike.

The generalization of the Tannaka
construction to an arbitrary base category $\V$ (instead of merely $\Vect_k$) was done by Schauenburg~\cite{Schauenburg92}, followed slightly later by Majid~\cite{Majid93}.
A more abstract approach to the Tannaka construction, still using strong monoidal fibre functors, was initiated by Day~\cite{Day96}, who considered the case of 
$\V$ a suitable enriched category. This abstract line of thinking was 
extended by McCrudden in~\cite{McCrudden00}~and~\cite{McCrudden02} and more recently by Sch\"appi~\cite{Schaeppi09}.

However, for our purposes, the most closely related existing work is that of Szlach\'anyi~\cite{Szlachanyi05}, who discusses separable Frobenius monoidal functors into 
$\V = \Mod_R$, for $R$ a commutative ring. On the one hand, our work is slightly more general in certain aspects---for instance, we work with braided $\V$, whereas $\Mod_R$ is symmetric.
However, the treatment in~\cite{Szlachanyi05} is much more sophisticated than the approach of the present paper, taking in, as it does, the more general notion of algebroids as well as
tackling the Krein recognition problem, which we do not discuss. Finally, Pfeiffer~\cite{Pfeiffer09} has shown that every
\emph{modular} category admits a generalized fibre functor into the field of endomorphisms of its tensor unit; this functor is separable Frobenius monoidal and he shows that the Tannaka
construction makes it into a weak Hopf algebra of a particular type.

\subsection{Structure}

In Section~\ref{graphical-algebra}, we rehearse the basic algebraic notions of bialgebras, weak bialgebras, Hopf algebras, and weak Hopf algebras, together with the string diagrams which
will be used throughout. In Section~\ref{graphical-functors}, we introduce our new graphical language for functors between monoidal categories which will be the key technical tool
for all of our proofs. In Section~\ref{def-of-tan}, we define the Tannaka construction for separable Frobenius monoidal functors, obtaining weak bialgebras and weak Hopf algebras. 
In Section~\ref{def-of-mod}, we recall the representation theory of weak bialgebras and weak Hopf algebras. In Section~\ref{adjunction}, we show that these constructions 
form an adjunction where the Tannaka construction is left adjoint to representation. Finally, in Section~\ref{change-of-base}, we consider varying the base category through a suitable
2-category of braided monoidal categories.

\section{Graphical Notation for Algebraic Objects}\label{graphical-algebra}

We make extensive use of the now-standard string diagram calculus for depicting morphisms in monoidal categories. Our convention is to depict composition
from left-to-right and to depict the tensor product from top-to-bottom; so for instance we depict a composite $a \tens b \oto{f} c \oto{g} d \tens e$ as:
\[ \includegraphics[width=200pt]{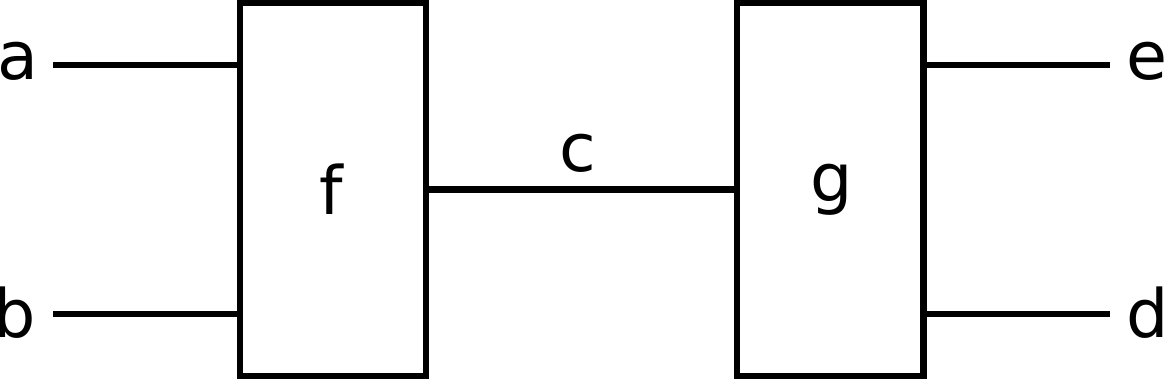} \]

\subsection{Basic Notions}
We recall the notions of \emph{weak bialgebra} and \emph{weak Hopf algebra}, to fix notation.

\begin{definition}[Algebras]
An \emph{algebra} or \emph{monoid} $H$ in a monoidal category $\V$ is an object $H$ equipped with a unit $\eta \colon \boot \too H$ and a multiplication 
$\mu \colon H \tens H \too H$, which must be associative and unital:
\[ \includegraphics[scale=0.2]{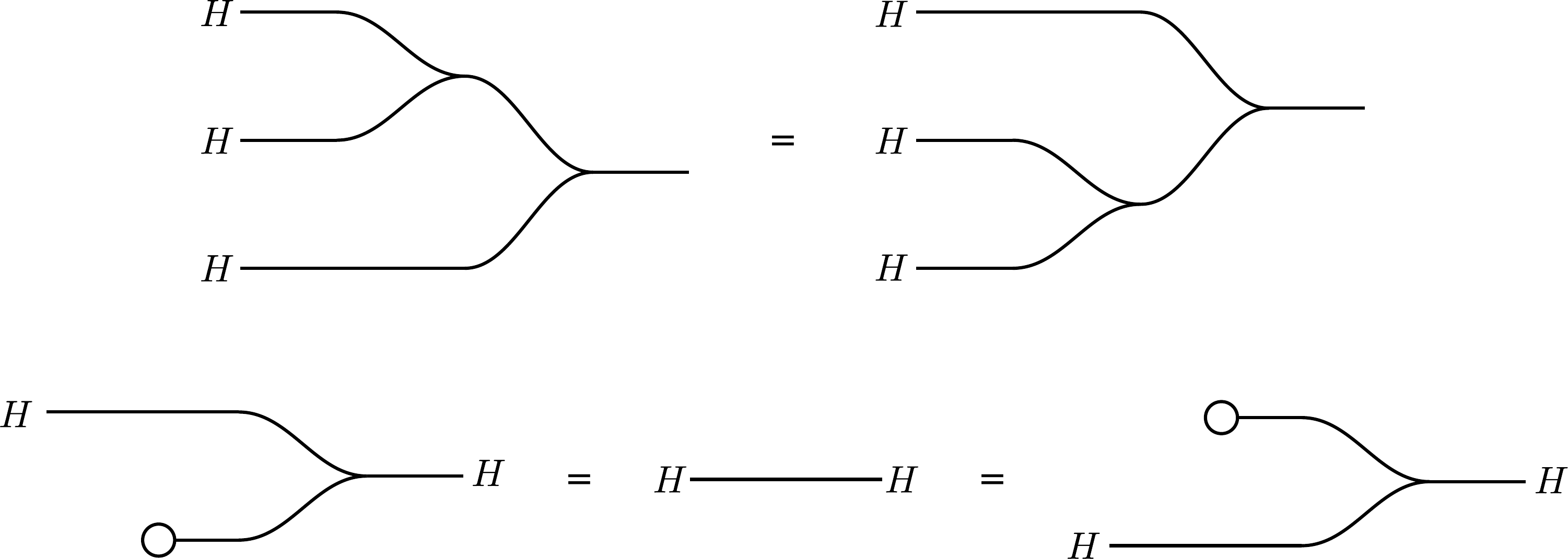} \]
\end{definition}

%

\begin{definition}[Coalgebras]
Dually, a \emph{coalgebra} or \emph{comonoid} $C$ is an object $C$ of $\V$ equipped with a counit $\epsilon \colon C \too \boot$ and a comultiplication 
$\Delta \colon C \too C \tens C$ and which must be coassociative and counital:
\[ \includegraphics[scale=0.2]{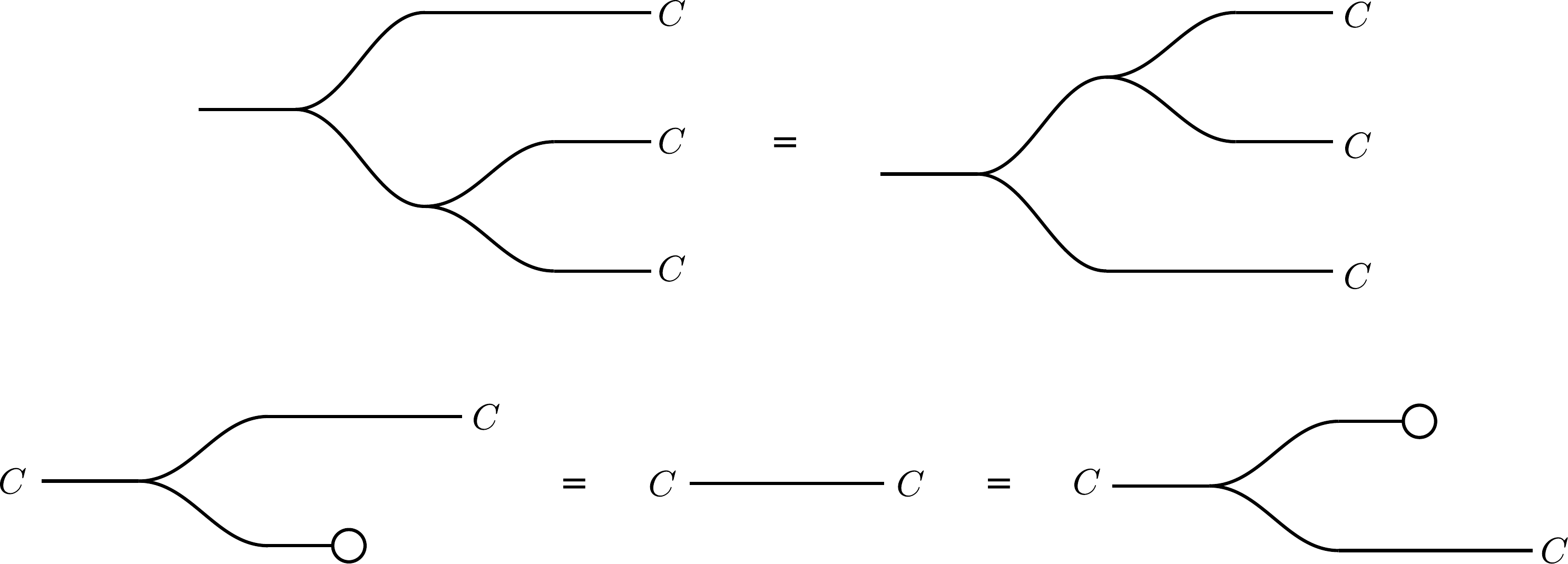} \]
\end{definition}

\begin{definition}[Convolution]
If $(A,\mu,\eta)$ is an algebra in a monoidal category $\V$, and $(C,\Delta,\epsilon)$ a coalgebra, then the set of arrows $\V(A,C)$ bears a canonical
monoid structure, known as \emph{convolution}, defined by: \[ f \star g = \mu(f \tens g)\Delta \] The neutral element for $\star$ is given by 
$\eta\epsilon$.
\end{definition} 

We can consider an object $H$ which is both an algebra and a coalgebra at once, and we can ask these two structures to cohere in various different
ways. For the moment we consider four such ways:
\subsubsection{Frobenius Algebras}

\begin{definition}[Frobenius Algebras]
An object $H$ equipped with both an algebra and a coalgebra structure is said to be a \emph{Frobenius algebra} if it satisfies:
\[ (H \tens \mu)(\Delta \tens H) = \Delta \mu = (\mu \tens H)(H \tens \Delta) \]
That is:
\[ \includegraphics[scale=0.2]{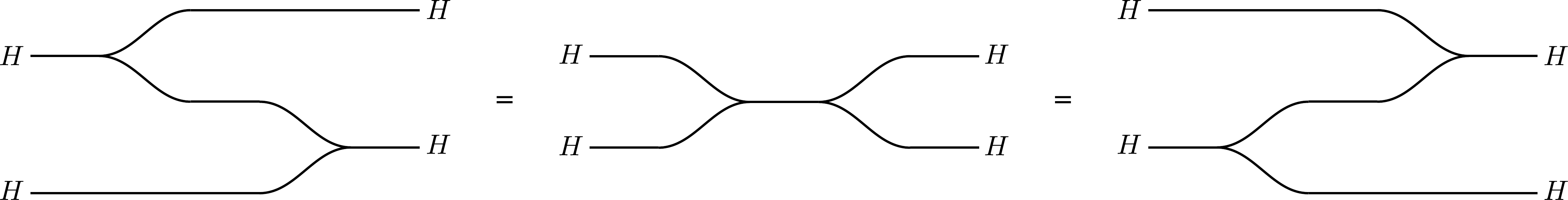} \]

\noindent A Frobenius algebra for which $\mu \Delta = H$ is said to be \emph{separable}:
\[ \includegraphics[scale=0.2]{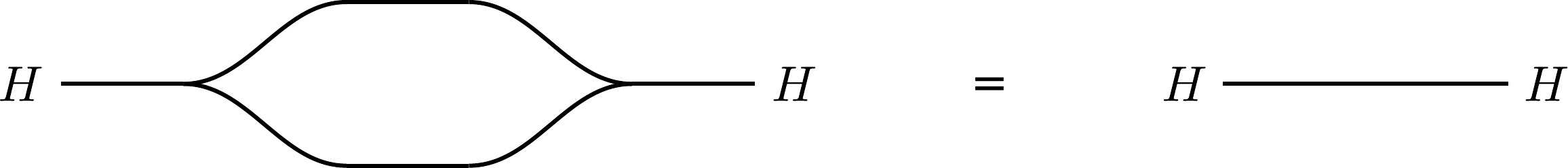} \]
\end{definition}
Note that the separability equation is precisely the assertion that the identity $H \colon H \too H$ is
a convolution idempotent $H \star H = H$.

\subsubsection{Bialgebras}

\begin{definition}[The Barbell] Suppose that $H$ is an object in a monoidal category, equipped with an algebra structure $(\mu,\eta)$ and a coalgebra structure $(\Delta,\epsilon)$.
We call the composite $\epsilon\eta$ the \emph{barbell}, because of its depiction: $\includegraphics[scale=0.2]{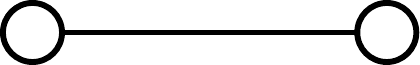}$
\end{definition}

\begin{definition}
An object in a braided category bearing an algebra and coalgebra struture is said to be a \emph{bialgebra} if it 
satisfies the following four axioms:

\begin{eqnarray}
\mbox{The Barbell Axiom:} &\qquad \begin{array}{c}\includegraphics[scale=0.1]{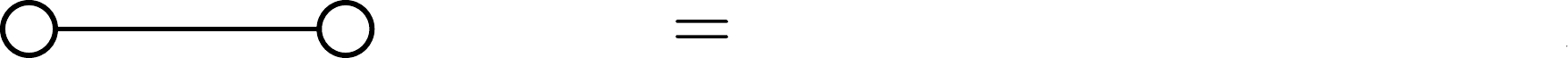}\end{array} \label{barbell-axiom} \\
\mbox{The (Strong) Unit Axiom:} &\qquad \begin{array}{c}\includegraphics[scale=0.1]{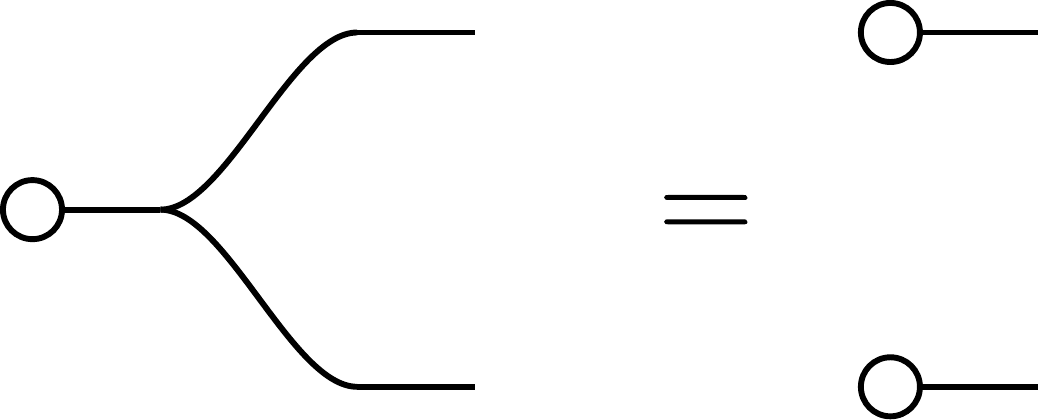}\end{array} \label{unit-axiom} \\
\mbox{The (Strong) Counit Axiom:} &\qquad \begin{array}{c}\includegraphics[scale=0.1]{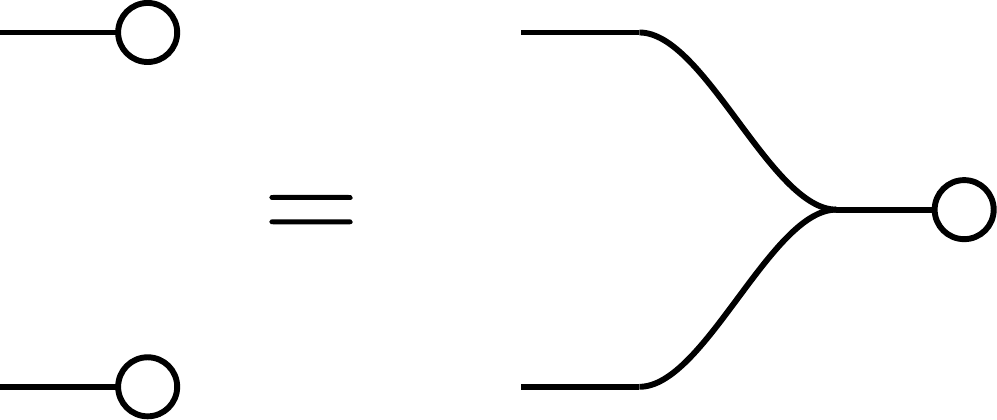}\end{array} \label{counit-axiom} \\
\mbox{The Bialgebra Axiom:} &\qquad \begin{array}{c}\includegraphics[scale=0.1]{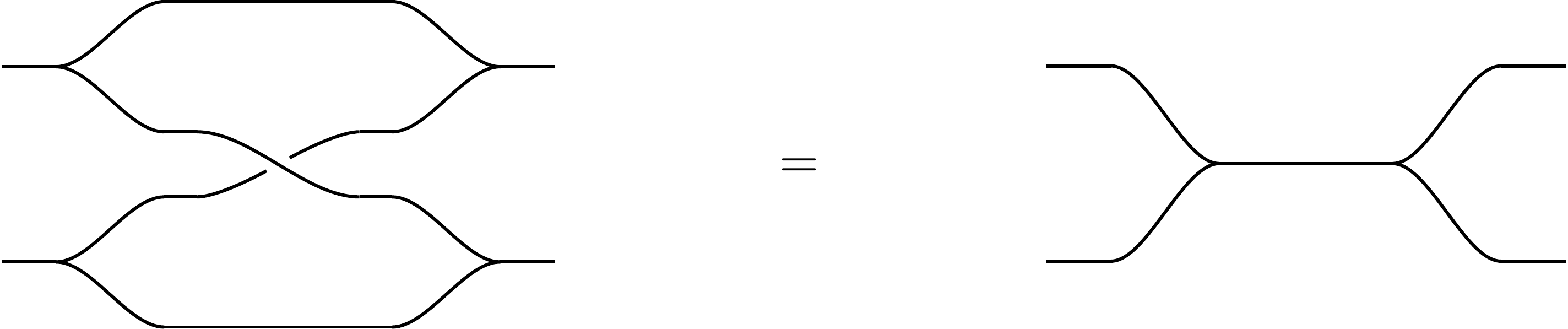}\end{array} \label{bialgebra-axiom}
\end{eqnarray} 

\noindent Note that the empty space on the right-hand side of the Barbell axiom depicts the identity on the tensor unit $\boot \colon \boot \too \boot$.
\end{definition}

\begin{definition} Let $H$ and $J$ be bialgebras in a braided monoidal category $\V$. Define a \emph{(non-weak) morphism of bialgebras} from $H$ to $J$
to be an arrow from $H$ to $J$ which commutes strictly with the multiplication, unit, comultiplication, and counit. In this way we obtain a category of
bialgebras in $\V$ which we denote $\ba \V$. \end{definition}

\subsubsection{Weak Bialgebras}

To move from a non-weak bialgebra to a weak bialgebra, we retain only the Bialgebra Axiom, replacing the other three axioms with weaker versions.
\begin{definition}[Weak Bialgebras]

An object in a braided category bearing an algebra and coalgebra structure is said to be a \emph{weak bialgebra} if it satisfies:
\begin{eqnarray}
\mbox{The Weak Unit Axioms:} &\qquad \begin{array}{c}\includegraphics[scale=0.1]{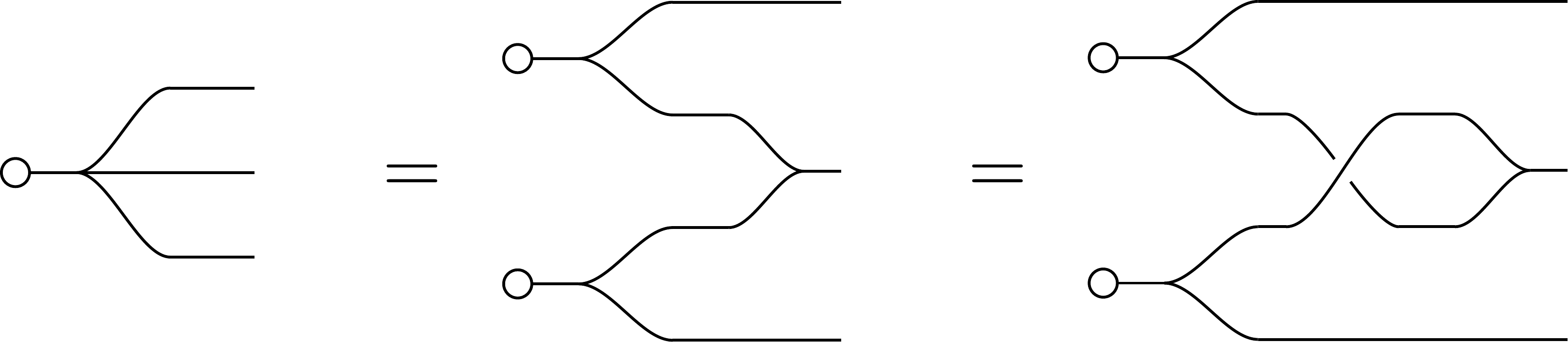}\end{array} \label{weak-unit-axioms} \\
\mbox{The Weak Counit Axioms:} &\qquad \begin{array}{c}\includegraphics[scale=0.1]{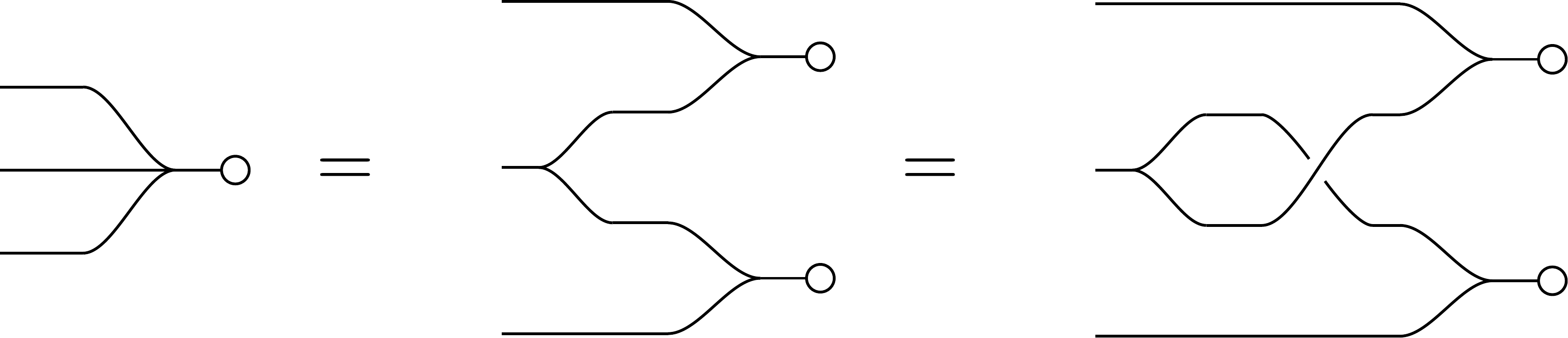}\end{array} \label{weak-counit-axioms} \\
\mbox{The Bialgebra Axiom:} &\qquad \begin{array}{c}\includegraphics[scale=0.1]{bialg.pdf}\end{array} \label{weak-bialgebra-axiom}
\end{eqnarray} Note that the braiding which occurs in the Weak Unit and Weak Counit Axioms is the inverse of the one which appears in the Bialgebra Axiom.
\end{definition} The notion of weak bialgebra was introduced by B\"ohm, Nill, and Szlach\a'nyi~\cite{BohmNillSzlachanyi99}, but see also the treatment of
Pastro and Street~\cite{PastroStreet09}

We defer discussion of morphisms of weak bialgebra until Section~\ref{extension-to-morphisms}.

\begin{definition} An element $c \colon \boot \too H$ of a bialgebra is said to be \emph{grouplike} if $\Delta c = c \tens c$.\end{definition}

\begin{definition}
An element $c \colon \boot \too H$ of a weak bialgebra $H$ is said to be \emph{almost grouplike} if $\Delta c = (\Delta \eta) \star (c \tens c) = (c \tens c) \star (\Delta \eta)$
\end{definition} Note that, if the weak bialgebra $H$ is in fact non-weak, then $(\Delta \eta) \star c \tens c = (\eta \tens \eta) \star (c \tens c)
= c \tens c$, and the almost grouplike elements of $H$ are, in fact, grouplike. In a bialgebra, the unit is grouplike by hypothesis, but it is
\emph{not} grouplike in a weak bialgebra, merely almost grouplike. Intuitively, we think of almost grouplike elements in a weak bialgebra as those
elements which are ``as grouplike as the unit is''.

\begin{observation}
The monoidal unit $\boot$ bears a canonical (trivial) algebra structure, as well as a trivial coalgebra structure. Since $\V$ is braided,
every tensor power of an algebra bears a canonical induced algebra structure; similarly, tensor powers of coalgebras are naturally also
coalgebras. Hence, we have another way of looking at the unit axioms for weak and non-weak bialgebras. As an algebra, $H \tens H$ has
two distinguished elements, namely, $\eta \tens \eta$ and $\Delta\eta$. In a non-weak bialgebra, we demand that these two be equal, but we
resist making this demand for a weak bialgebra. If $H$ is a weak bialgebra, then there are four distinguished elements of $H \tens H \tens H$,
namely: \[ \eta \tens \eta \tens \eta \qquad \qquad \Delta\eta \tens \eta \qquad \qquad \eta \tens \Delta\eta \qquad \qquad \Delta_3\eta \]
where $\Delta_3$ is the common value of $(\Delta \tens H)\Delta = (H \tens \Delta)\Delta$. Insisting that
these four distinguished elements should be equal is much too strong, instead, the weak unit axioms amount to the following:
\[ (\Delta\eta \tens \eta) \star (\eta \tens \Delta\eta) = \Delta_3\eta = (\eta \tens \Delta\eta) \star (\Delta\eta \tens \eta) \]
Similarly, the weak counit axioms can be given as:
\[ (\epsilon\mu \tens \epsilon) \star (\epsilon \tens \epsilon\mu) = \epsilon\mu_3 = (\epsilon \tens \epsilon\mu) \star (\epsilon\mu \tens \epsilon) \]
Written in this form, as in the graphical form, the duality between the weak unit and weak counit axioms is apparent. In the usual Sweedler notation
for weak bialgebras in $\Vect_k$ (where we adopt the conventional $\eta = 1$), these identities appear as 
$1_1 \tens 1_21_{1'} \tens 1_{2'} = 1_1 \tens 1_2 \tens 1_3 = 1_1 \tens 1_{1'}1_{2} \tens 1_{2'}$ and
$\epsilon(ab_1)\epsilon(b_2c) = \epsilon(abc) = \epsilon(ab_2)\epsilon(b_1c)$, and the duality is obfuscated.
\end{observation}

\subsection{The Canonical Idempotents on a Weak Bialgebra} \label{canonical-idempotents-section}

\begin{definition}
There are four canonical idempotents on a weak bialgebra, namely:
\fullwidth{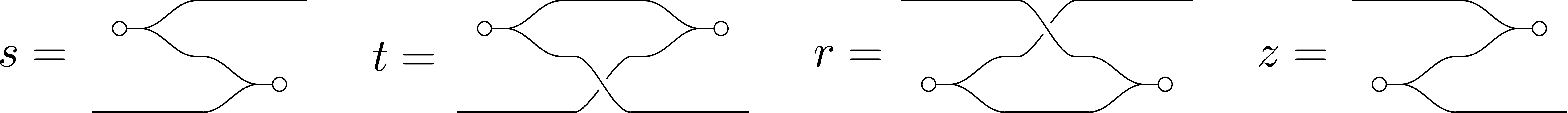}
Checking that they are idempotents is an exercise in applying the weak unit and weak counit axioms.
\end{definition}

\begin{definition}
	Let $C$ be a category. The \emph{idempotent-splitting completion} or \emph{Cauchy completion} or \emph{Karoubi envelope} of $C$ is written
	as $QC$. It is defined as having objects pairs $(A,a)$, where $a \colon A \too A$ is an idempotent in $C$, and morphisms 
	$f \colon (A,a) \too (B,b)$, where $f \colon A \too B$ is a morphism in $C$ such that $bfa = f$. Note that the identity on $(A,a)$ is
	the morphism $a \colon A \too A$, not the identity on $A$.
\end{definition}

\begin{proposition} \label{canonical-idempotents}
Let $H$ be a weak bialgebra in a monoidal category $\V$.
As objects in $Q\V$, all four canonical idempotents on $H$ are isomorphic.\end{proposition}
\begin{proof}
The four maps: \[ (H,s) \oto{t} (H,t) \oto{t} (H,z) \oto{r} (H,r) \oto{r} (H,s) \] are isomorphisms in $Q\V$ with inverses
\[ (H,s) \loto{s} (H,t) \loto{z} (H,z) \loto{z} (H,r) \loto{s} (H,s) \] which may be readily checked by the reader.\end{proof}

\subsubsection{The Internal Separable Frobenius Algebra in a Weak Bialgebra}\label{internal-sepfrob}
It is shown by Schauenburg (Proposition~4.2 of~\cite{Schauenburg03}, see also Pastro and Street~\cite{PastroStreet09}) that a splitting 
of the idempotent $t \colon H \too H$
on a weak Hopf algebra $H$ inherits a separable Frobenius structure from the weak bialgebra structure of $H$. 
\begin{definition}
Since all four canonical idempotents on $H$ are isomorphic as idempotents, we call this splitting 
\emph{the internal separable Frobenius algebra} associated to $H$.\end{definition}

\subsection{Hopf Notions}

\begin{definition}[Hopf Algebras]
A \emph{Hopf algebra} is a bialgebra equipped with an \emph{antipode} $S\colon H \too H$ which is a convolution inverse to the identity; 
that is, such that: \begin{equation} \begin{array}{c} \includegraphics[scale=0.2]{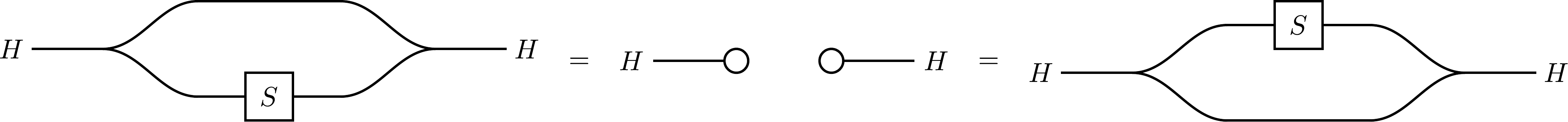} \end{array} \label{antipode-axioms} \end{equation}
\end{definition}

\begin{definition} Given two Hopf algebras $H$ and $J$ in a monoidal category $\V$, we define a \emph{morphism of Hopf algebras} from $H$ to $J$ 
to be merely a morphism of their underlying
bialgebras; it can be shown that such morphisms necessarily commute with the antipodes of $H$ and $J$. We obtain in this way a category $\ha \V$
of Hopf algebras in $\V$. \end{definition}

\begin{definition}[Weak Hopf Algebras]
A \emph{weak Hopf algebra} is a weak bialgebra with an \emph{antipode} $S\colon H \too H$, satisfying instead:
\begin{equation} S \star H = r \qquad \qquad S \star H \star S = S \qquad \qquad H \star S = t \label{weak-antipode-axioms} \end{equation}

where $r$ and $t$ are the canonical idempotents mentioned above; graphically:
\[ \begin{array}{c}\includegraphics[scale=0.16]{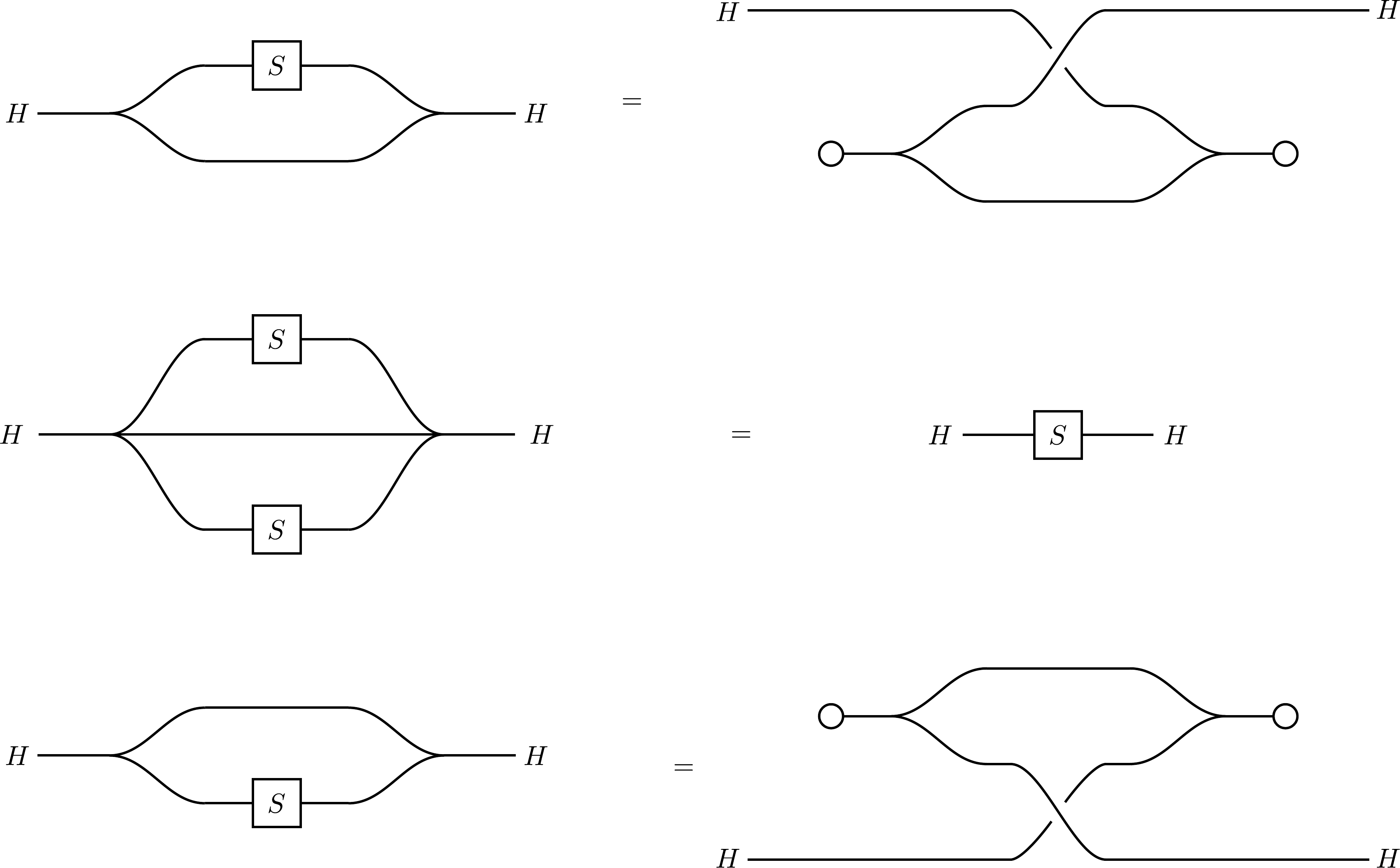} \end{array} \] 
Note that either of $S \star H = r$ or $H \star S = t$ can be combined with the Bialgebra Axiom (Equation~\ref{weak-bialgebra-axiom}) 
to give $H \star S \star H = H$, and so an antipode
on a weak Hopf algebra can be thought of as a well-behaved weak convolution inverse to the identity in the sense of semigroups.
\end{definition}

For emphasis, we will sometimes describe Hopf algebras as ``non-weak'' Hopf algebras. We defer discussion of morphisms between weak Hopf algebras until
Section~\ref{extension-to-morphisms}.

\section{Graphical Notation for Functors}\label{graphical-functors}

We introduce depictions for monoidal and comonoidal structures on functors between monoidal categories. The original notion for graphically 
depicting monoidal functors
as transparent boxes in string diagrams is due to Cockett and Seely~\cite{CockettSeely99}, and has recently been revived and popularized by 
Melli\`es~\cite{Mellies06} with prettier graphics and an
excellent pair of example calculations which nicely show the worth of the notation. However, a small alteration improves the notation considerably.
For a monoidal structure on a functor $f \colon A \too B$, we have a natural family of maps:
$\monbin \colon fx \tens fy \too f(x \tens y)$ and a map $\monnul \colon \boot \too f\boot$, which we notate as follows:
\begin{center}\resizebox{!}{50pt}{\includegraphics{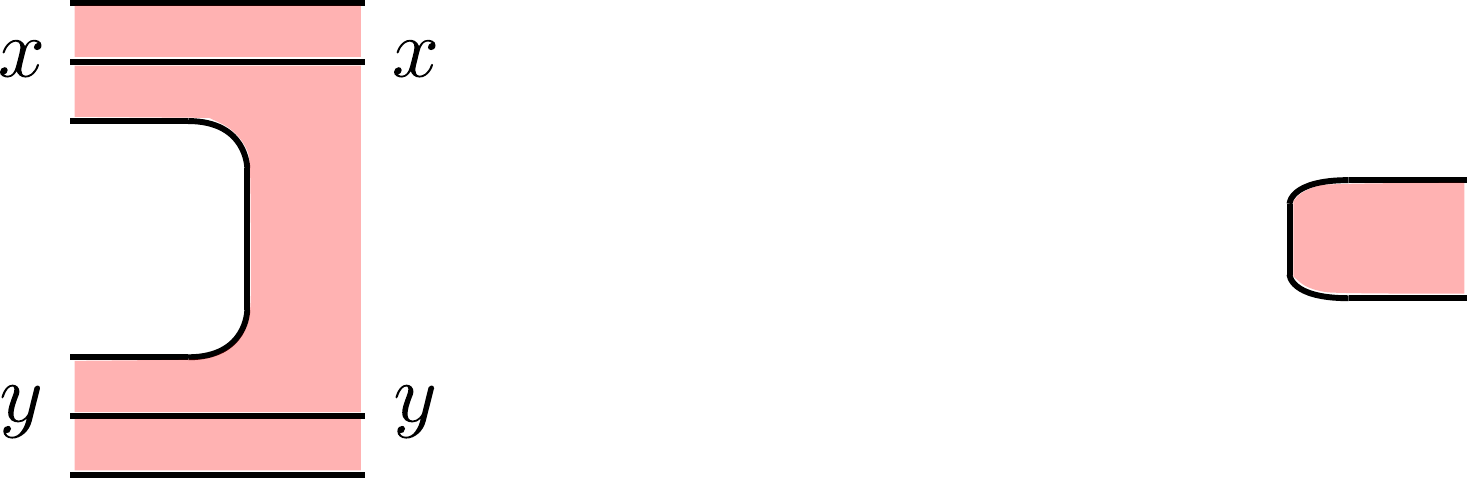}}\end{center} 
Similarly, for a comonoidal structure on $f$, we have maps $\combin \colon f(x \tens y) \too fx \tens fy$ and $\comnul \colon f\boot \too \boot$ which we notate in the 
obvious dual way, as follows:
\begin{center}\resizebox{!}{50pt}{\includegraphics{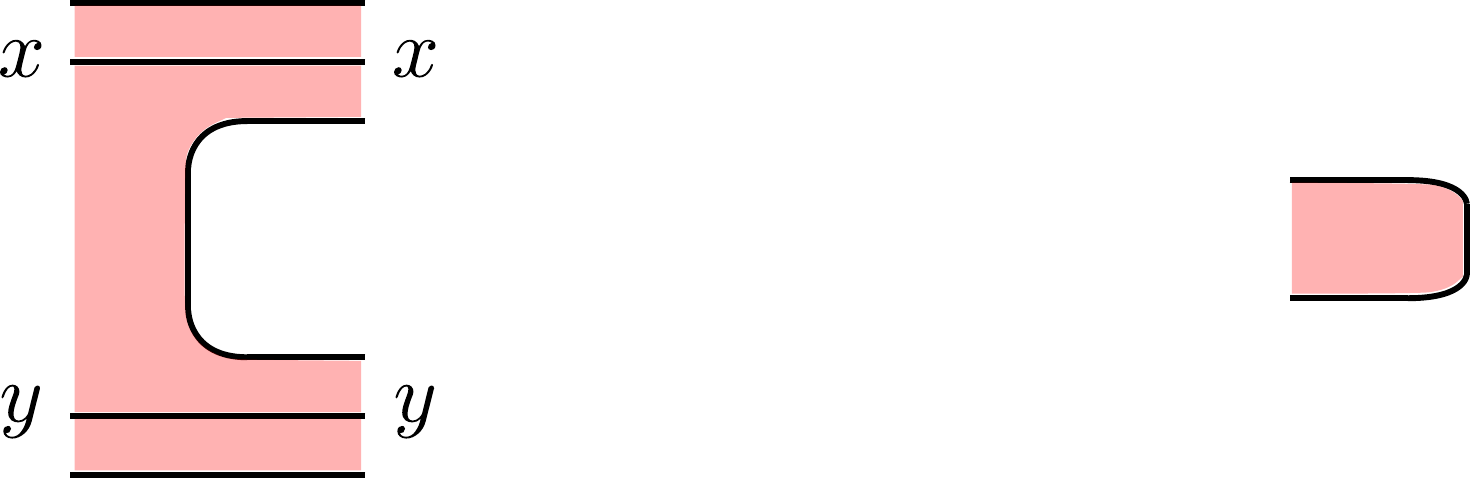}}\end{center}
Note that the functor symbol ``$f$'' does \emph{not} appear in the wire labels; after all, its red color identifies it. Furthermore, the
tensor unit $\boot$ is suppressed, as usual. Finally, notice that the naturality of the binary monoidal or comonoidal structure is made obvious by
the depiction of the wires labelled ``$x$'' or ``$y$'' passing unperturbed from left to right.

The structural maps for a monoidal functor are required to be associative:
\[ \includegraphics[scale=0.2]{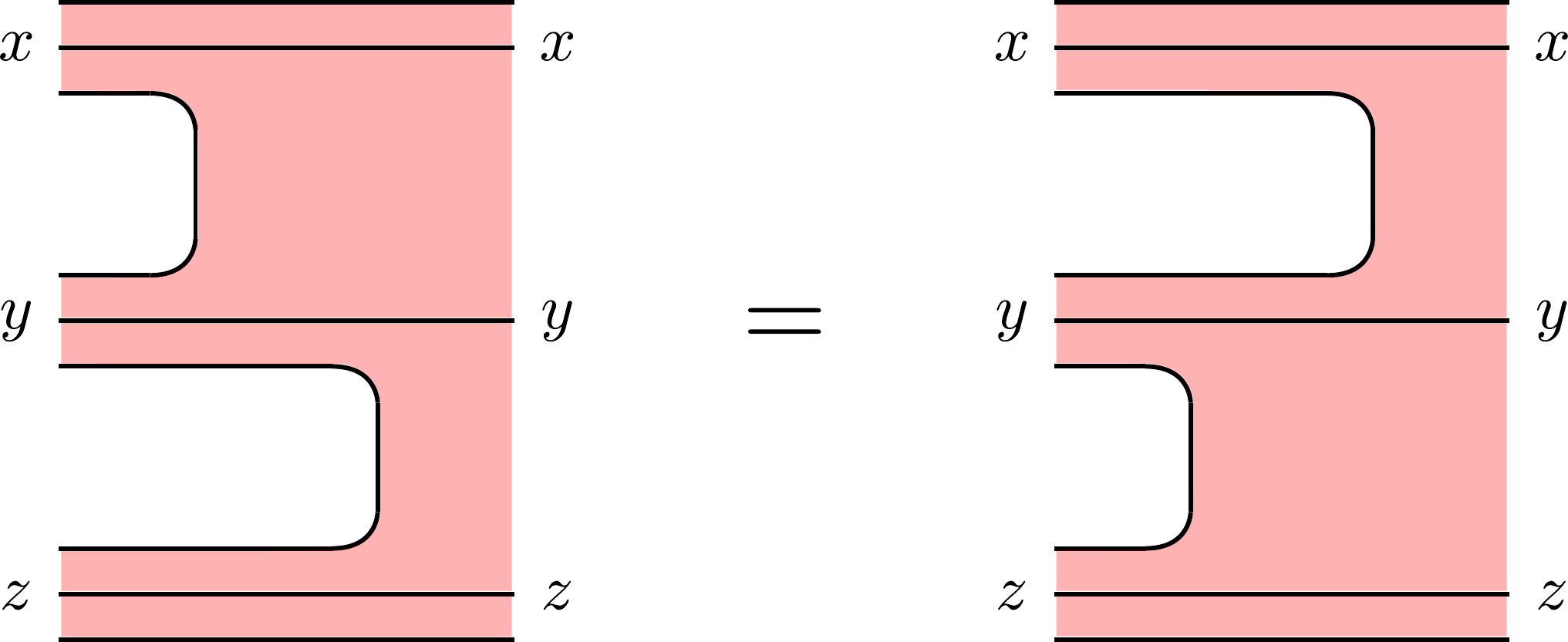} \]
and unital:
\[ \includegraphics[scale=0.2]{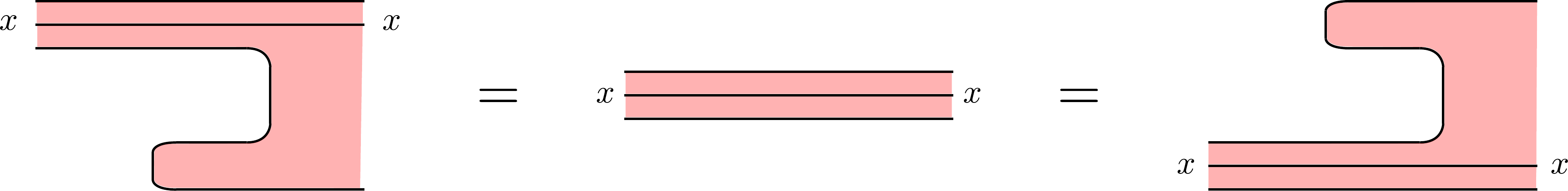} \]
where, once again, the corresponding constraints for a comonoidal functor are exactly the above with composition read right-to-left instead of left-to-right.
Note that flipping these axioms vertically leaves them unchanged.

The above axioms seem to indicate some sort of ``invariance under continuous deformation of functor-regions''. For a functor which is both 
monoidal and comonoidal, pursuing this line of thinking leads one to consider the following pair of axioms:
\fullwidth{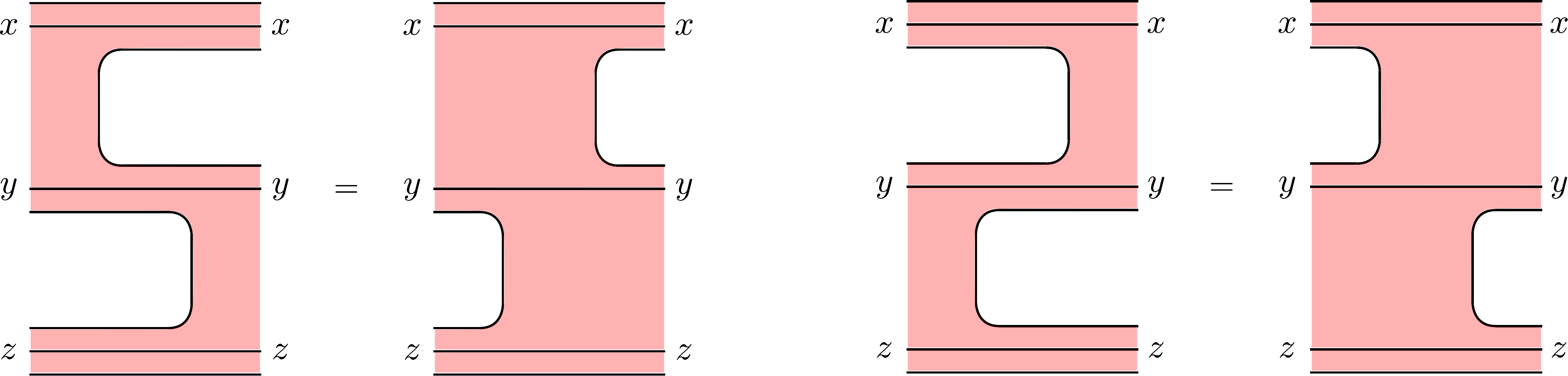}
Or, in pasting diagrams:
\begin{equation}\label{frobenius-axioms}
\begin{array}{c}
 \bfig

\vscalefactor{0.8}
\node tr(-500,+250)[fx \tens (fy \tens fz)]
\node br(-500,-250)[(fx \tens fy) \tens fz]
\node top(0,+750)[fx \tens f(y \tens z)]
\node bot(0,-750)[f(x \tens y) \tens fz]
\node tl(+500,+250)[f(x \tens (y \tens z))]
\node bl(+500,-250)[f((x\tens y) \tens z)]

\arrow|l|[tr`br;\delta]
\arrow|l|[top`tr;fx \tens \comonoidal]
\arrow|l|[br`bot;\monoidal \tens fz]
\arrow|r|[bl`bot;\comonoidal]
\arrow|r|[top`tl;\monoidal]
\arrow|r|[tl`bl;f\delta]

\node top(-2000,+750)[f(x \tens y) \tens fz]
\node tr(-1500,+250)[f((x\tens y) \tens z)]
\node br(-1500,-250)[f(x\tens (y\tens z))]
\node bot(-2000,-750)[fx \tens f(y \tens fz)]
\node tl(-2500,+250)[(fx \tens fy) \tens fz]
\node bl(-2500,-250)[fx \tens (fy \tens fz)]

\arrow|r|[tr`br;f\delta]
\arrow|r|[top`tr;\monoidal]
\arrow|r|[br`bot;\comonoidal]
\arrow|l|[bl`bot;fx \tens \monoidal]
\arrow|l|[top`tl;\comonoidal \tens fz]
\arrow|l|[tl`bl;\delta]

\efig \end{array}
\end{equation}

\begin{definition}[Definition~1 of Day and Pastro~\protect\cite{DayPastro08}; see also Definition~6.4 of Egger~\protect\cite{Egger08}]
A functor between monoidal categories bearing a monoidal structure and a comonoidal structure, satisfying Equations~\ref{frobenius-axioms}, is said to be
\emph{Frobenius monoidal}.
\end{definition}

Note that the unadorned ``Frobenius'' has already been used in~\cite{CaenepeelMilitaruZhu97} to mean a functor possessing coinciding left and 
right adjoints; we will have no use of this notion.

Frobenius monoidal functors are so-named because Frobenius monoidal functors from the terminal monoidal category into a category $C$ are in bijection
with Frobenius algebras in $C$. Furthermore, they sport two additional pleasant properties:
\begin{itemize}
\item Every strong monoidal functor is Frobenius monoidal (Proposition~3 of~\cite{DayPastro08});
\item Every Frobenius monoidal functor preserves duals (Theorem~2 of~\cite{DayPastro08}; 
this is a special case of Corollary~A.14 of~\cite{CockettSeely99}).
\end{itemize}

For the moment, let us examine the gap between Frobenius monoidal and strong monoidal functors. To demand that a Frobenius monoidal functor be strong is to
demand the following four conditions:
\begin{equation} \scaledpic{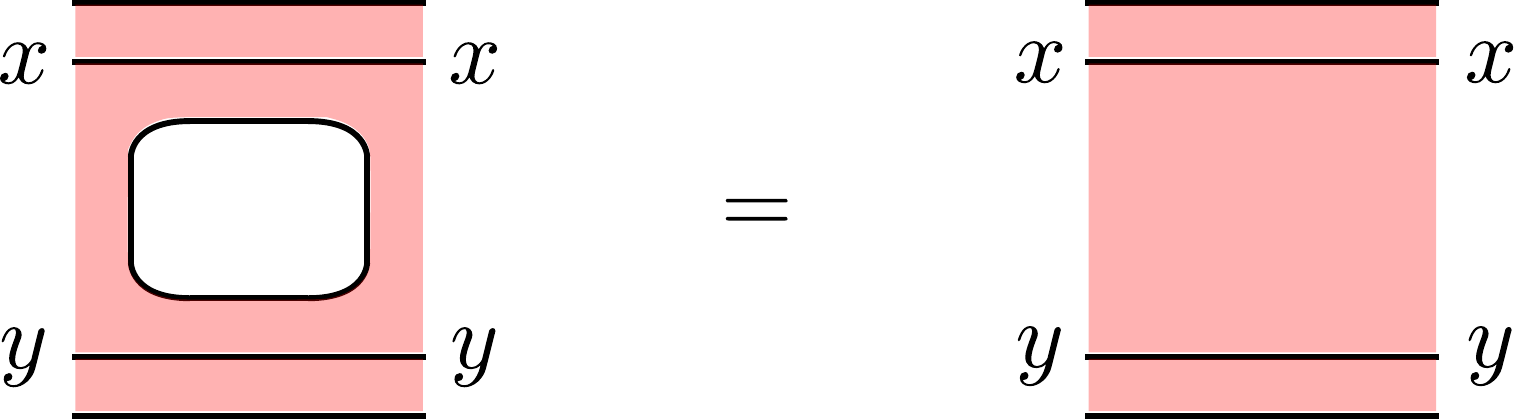}{0.4} \label{sep-condition-bin} \end{equation}
\begin{equation} \scaledpic{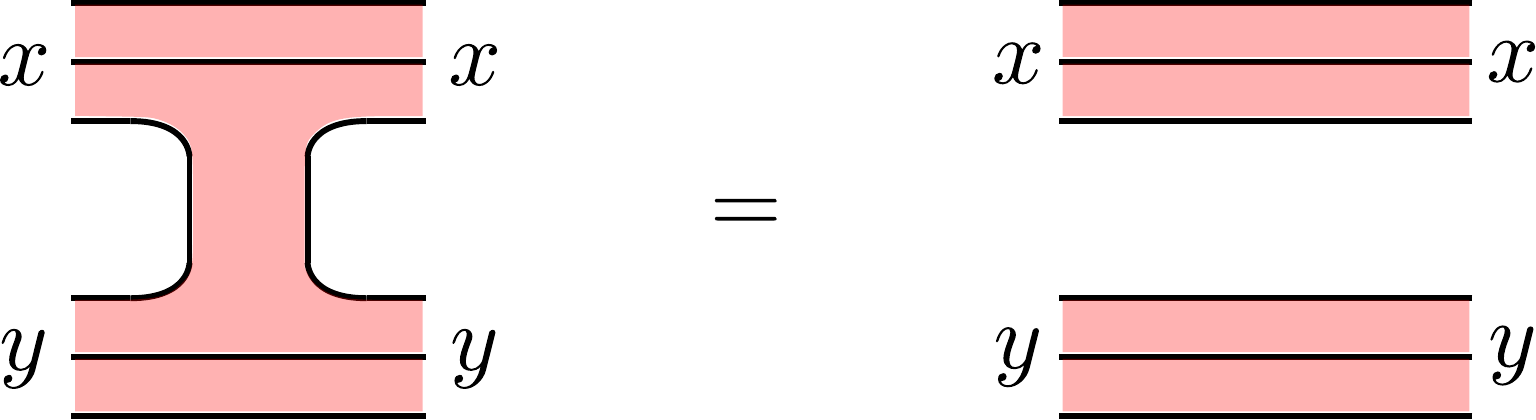}{0.4} \label{mix-condition-bin} \end{equation}
\begin{equation} \scaledpic{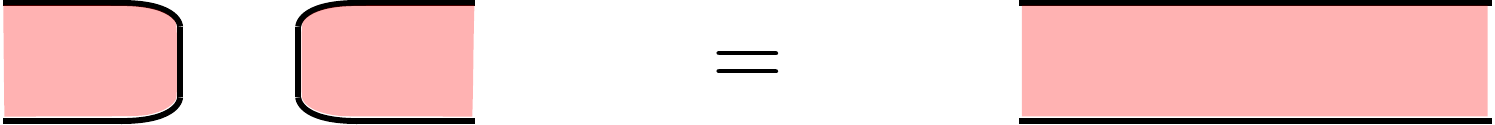}{0.4} \label{sep-condition-nul} \end{equation}		
\begin{equation} \scaledpic{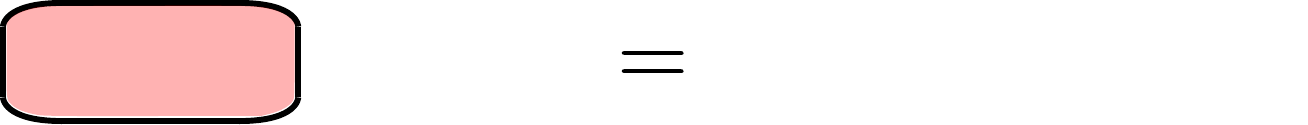}{0.4} \label{mix-condition-nul} \end{equation}		
where the blank right-hand-side of the bottom equation denotes the identity on the tensor unit. Following the above intuition of ``continuous
deformation of $f$-region'', we see that each condition here fails this intuition. 
Equations~\ref{mix-condition-bin}, \ref{sep-condition-nul}, and \ref{mix-condition-nul} each posit an equality between
two different numbers of ``connected components of $f$-regions''. Equation~\ref{sep-condition-bin} avoids this fault but instead posits 
an equality between a ``simply connected $f$-region'' and a non-simply connected such region---hence, even at this qualitative topological level, 
we see that this condition is unlike the others. Thus, we define:
\begin{definition}[Definition~6.1 of \cite{Szlachanyi05}] A Frobenius monoidal functor is \emph{separable} just when it satisfies Equation~\ref{sep-condition-bin}.\end{definition}

The original motivation for the word ``separable'' comes from the fact that separable Frobenius monoidal functors $1 \too \mathcal C$ correspond to 
separable Frobenius algebras in $\mathcal C$ in the classical sense.
The precise connection between the topology of the functor regions in our depictions and their algebraic properties is spelled out in~\cite{McCurdyStreet10}.

The category of monoidal categories and Frobenius monoidal functors between them we denote by $\fm$; the lluf subcategory of separable
Frobenius monoidal functors by $\sfm$, and the further lluf subcategory of \emph{strong} monoidal functors by $\str$. We shall have no need
of strict monoidal functors.

\section{The Tannaka Construction}\label{def-of-tan}

\begin{definition}
	Let $\V$ be a monoidal category. We say that $\V$ is \emph{monoidally complete} if it is complete as a category, 
	and, for every object $b \in \V$, the functors $b \tens -$ and $- \tens b$ are both continuous.
\end{definition}

\begin{definition}[Functors of Reconstruction Type]
Let $A$ be a small category and let $F \colon A \too \V$ be a functor. We say that $F$ is \emph{of covariant reconstruction type} if the following hold:
\begin{itemize}
	\item $\V$ is monoidally complete;
	\item $\V$ is braided monoidal;
	\item For every $a \in A$, there is a left dual $\prep{(Fa)}$ for $Fa$ in $\V$.
\end{itemize}
\end{definition}

\begin{definition} A functor $F \colon A \too \V$ is of \emph{contravariant} reconstruction type when $F^{\op}$ is of covariant reconstruction type.
\end{definition}

\begin{definition}[Tannaka Objects] Let $F \colon A \too \V$ be a functor of covariant reconstruction type. The \emph{covariant 
Tannaka object associated to $F$} is: \[ \tan F = \int_{a \in A} Fa \tens \prep{(Fa)} \]
\end{definition} 
\noindent Many treatments instead consider functors of contravariant reconstruction type, and form
\[ \cot F = \int^{a \in A} Fa \tens \prep{(Fa)}, \] this is the \emph{contravariant Tannaka object associated to $F$}.
Note that, since $\cot F = \tan F^\op$, we lose no generality by working always with functors of covariant reconstruction type; and we shall do so 
throughout the remainder of this paper.

In this section, we shall prove the following:
\begin{theoremnonumber} Let $F \colon A \too \V$ be a separable Frobenius monoidal functor of covariant reconstruction type. Then $\tan F$ bears the structure of
a weak bialgebra. Moreover, if $A$ is autonomous, then $\tan F$ bears the structure of a weak Hopf algebra.\end{theoremnonumber}

\noindent In a sequel~\cite{McCurdy11Tannaka2} to this paper, we shall give three refinements of this theorem; namely: \begin{itemize}
	\item If $A$ is braided, then $\tan F$ is a braided or quasitriangular weak bialgebra in $\V$, generalizing the notion of quasitriangular bialgebra~\cite{Drinfeld86}.
	\item If $A$ and $\V$ are both tortile categories, then $\tan F$ is a ribbon weak bialgebra in $\V$, generalizing the notion of ribbon bialgebra~\cite{ReshetikhinTuraev90}.
	\item If $A$ is a cyclic category in the sense of~\cite{EggerMcCurdy10} (that is, having isomorphic left and right duals), then $\tan F$ is a
		cyclic weak bialgebra. This last generalizes the existing notion of sovereign bialgebra introduced in~\cite{Bichon01}.
\end{itemize}

\begin{observation}
The object $\tan F$ acts universally on the functor $F$, with action $\alpha \colon \tan F \tens F \too F$ is defined to have components:
\[ \tan F \tens Fx = \left( \int_{a \in A} Fa \tens \prep{(Fa)} \right) \tens Fx \oto{\pi_x \tens Fx} Fx \tens \prep{(Fx)} \tens Fx 
\oto{Fx \tens \epsilon_x} Fx \tens \boot \oto{\rulaw} Fx \]
using the $x$'th projection from the end followed by the counit of the $\prep{(Fx)} \ladj Fx$ adjunction. By ``universality'' here, we mean
that composition with $\alpha$ mediates a bijection between maps $X \too \tan F$ in $\V$ and natural transformations $X \tens F \too F$, which
may be readily verified.
\end{observation}

Dually, there is a canonical coaction $\alpha' \colon F \too F \tens \cot F$; see page~254~of~Ulbrich~\cite{Ulbrich90}.

The dinaturality of the end in $a$ gives rise to the naturality of the above defined action, which we notate as:
\begin{center}\resizebox{300pt}{!}{\includegraphics{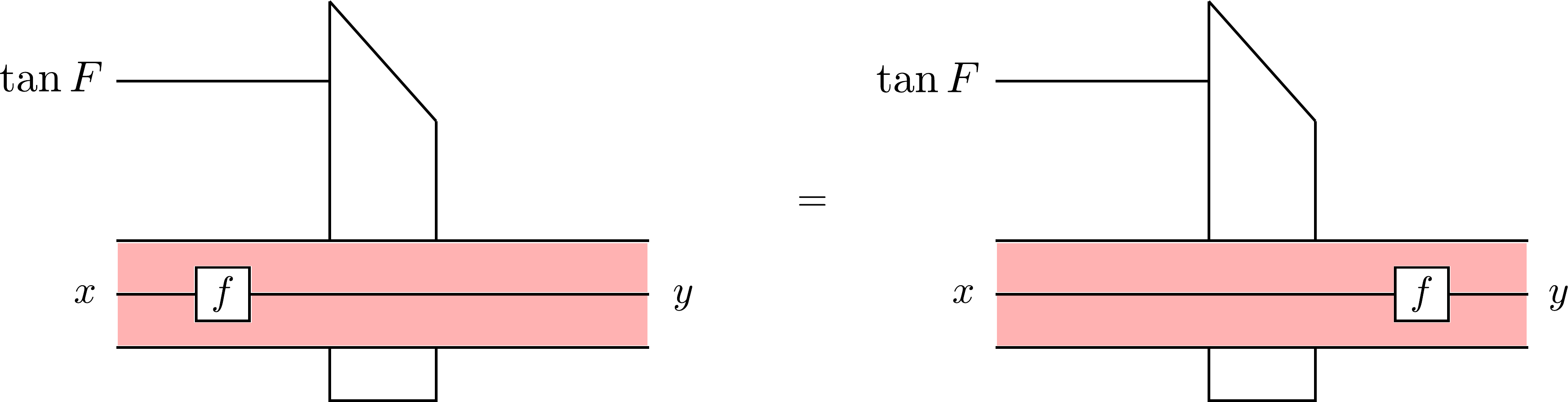}}\end{center}

If $F \colon A \too \V$ is a functor of representation type, then so too is $F^n$, by which we mean the functor $A^n \too \V$ defined by
$(a_1,a_2,\ldots,a_n) \mapsto Fa_1 \tens Fa_2 \tens \cdots \tens Fa_n$. From the action $\alpha \colon \tan F \tens F \too F$, we can obtain
actions of $(\tan F)^{\tens n}$ on $F^n$, written $\alpha^n$. Taking $\alpha^1 = \alpha$, we define $\alpha^n$ recursively as follows:
\[ \bfig


\node aa(-1000,+500)[(\tan F)^{\tens n} \tens F^n]
\node a(-1000,0)[(\tan F)^{\tens(n-1)} \tens \tan F \tens F^{n-1} \tens F]
\node b(-1000,-500)[(\tan F)^{\tens(n-1)} \tens F^{n-1} \tens \tan F \tens F]
\node c(+1000,-500)[F^{n-1} \tens F]
\node cc(+1000,+500)[F^n]

\arrow/=/[a`aa;]
\arrow[a`b;(\tan F)^{\tens(n-1)} \tens {\rm braid} \tens F]
\arrow[b`c;\alpha^{n-1} \tens \alpha^1]
\arrow/=/[c`cc;]
\arrow[aa`cc;\alpha^n]

\efig \]

\begin{proposition} For each $n \in \mathbb N$, the map $\alpha^n \colon (\tan F)^{\tens n} \tens F^n \too F^n$ exhibits $(\tan F)^{\tens n}$
as $\tan F^n$.
\end{proposition}
\begin{proof} Since $\V$ is monoidally complete, tensoring with $\tan F$ preserves ends.
The proposition then follows easily from the case $n=1$ above.
\end{proof}

\begin{definition}[Discharged forms] 
For any map $f \colon X \too (\tan F)^{\tens n}$ in $\V$, we call the map
\[ X \tens F^n \oto{f \tens F^n} (\tan F)^{\tens n} \tens F^n \oto{\alpha^n} F^n \] the \emph{discharged form} of $f$. From the above proposition,
two maps are equal if and only if they have the same discharged form.
\end{definition}
We will use this property to define algebraic structures on $\tan F$, as well as to verify all of the axioms of those algebraic structures.


\subsection{Definition of the Structure}
\subsubsection{Algebra Structure}

\begin{proposition} Let $F\colon A \too \V$ be a functor of reconstruction type. Then $\tan F$ is an algebra,
with multiplication defined as having discharged form:
\begin{equation} \begin{array}{c}\includegraphics[scale=0.28]{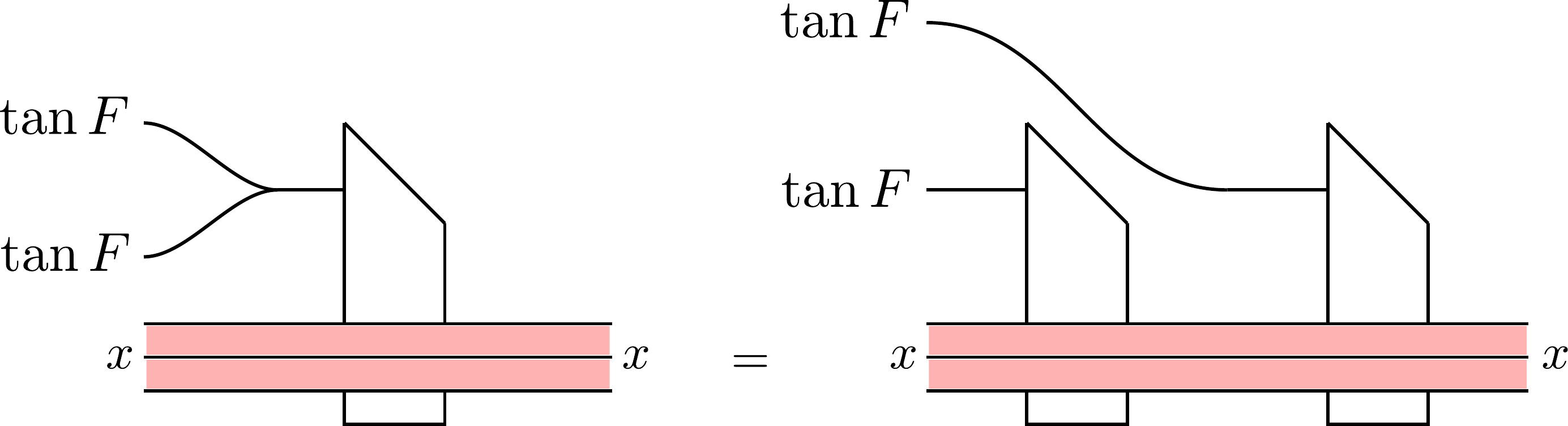} \end{array} \label{def-of-multiplication-eq} \end{equation}
and unit having discharged form:
\begin{equation} \begin{array}{c}\includegraphics[scale=0.28]{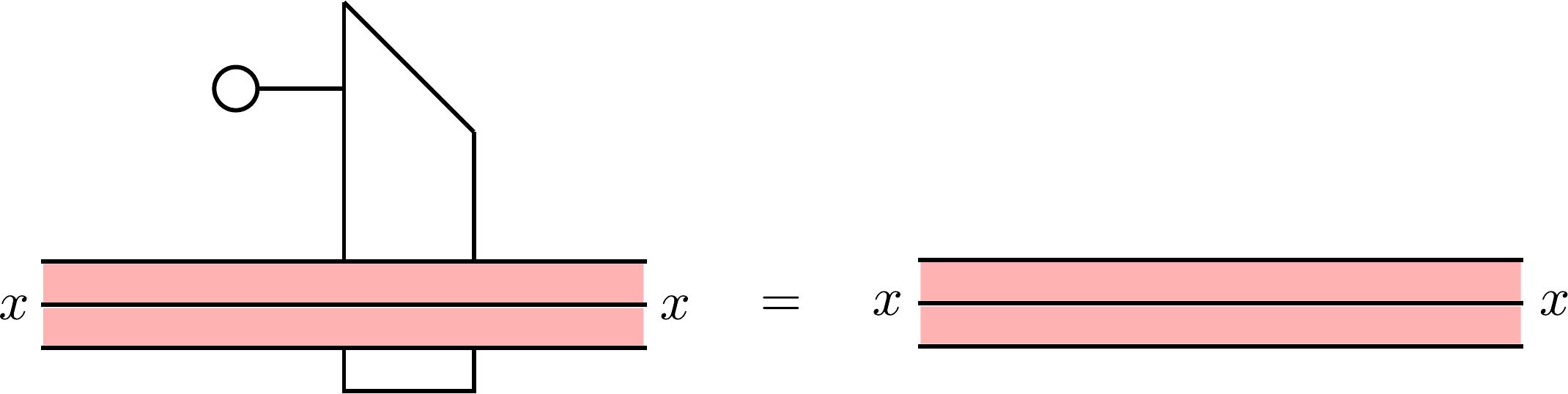} \end{array}\label{def-of-unit-eq} \end{equation}
\end{proposition}

\noindent Note that this monoidal structure is associative and unital, without assuming that $A$ is monoidal.

\subsubsection{Coalgebra Structure}

\begin{proposition} Suppose that $F\colon A \too \V$ is a monoidal and comonoidal functor of reconstruction type.
Then, without assuming any coherence between these structures, we can define a coassociative comultiplication on $\tan F$ as having
discharged form:
\begin{equation} \begin{array}{c} \includegraphics[scale=0.25]{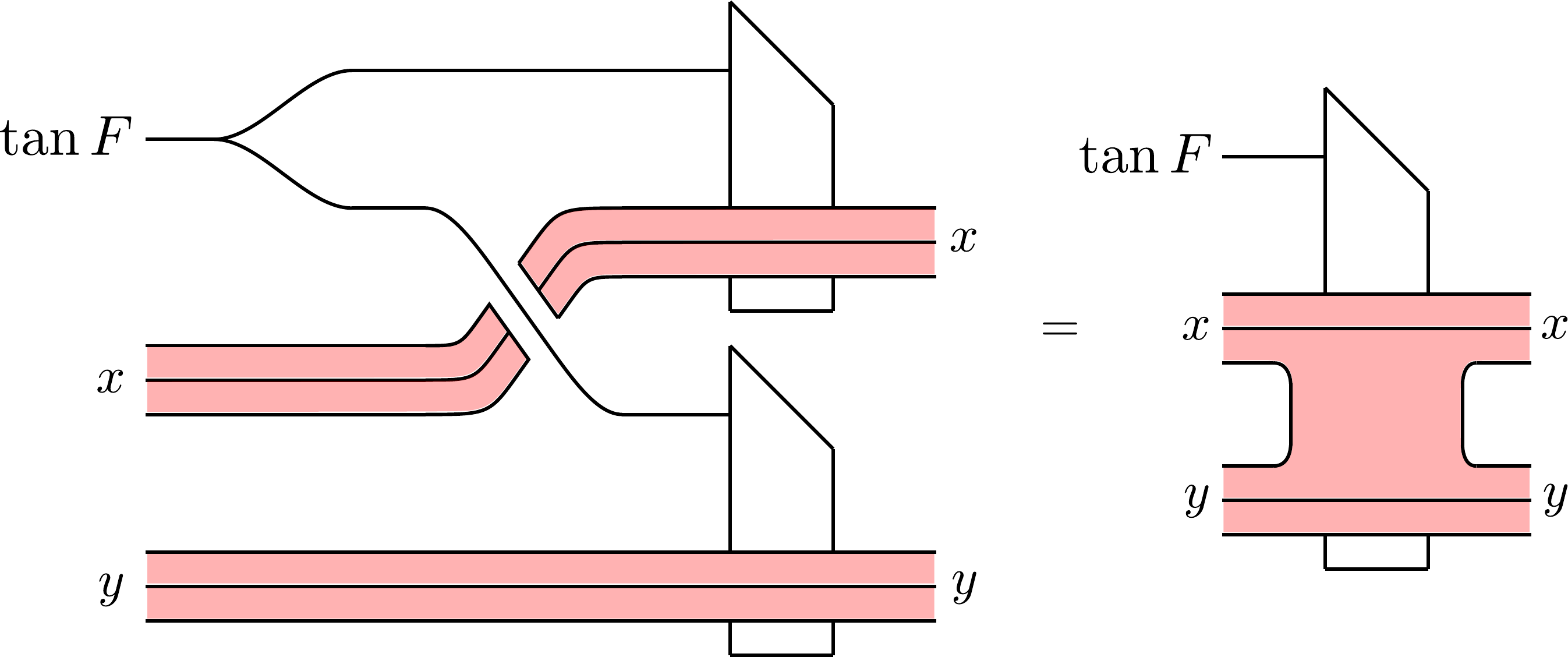} \end{array} \label{def-of-comultiplication-eq} \end{equation}
As well as a counit for $\tan F$:
\begin{equation} \begin{array}{c} \includegraphics[scale=0.25]{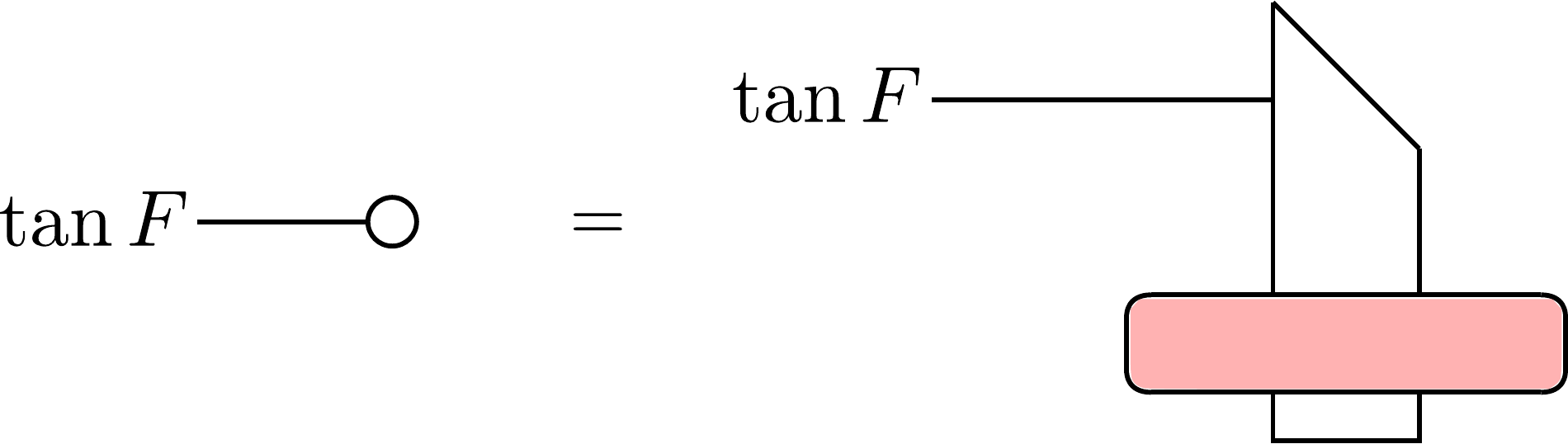} \end{array} \label{def-of-counit-eq} \end{equation}
\end{proposition}

\begin{observation} These definitions imply that the discharged form of the iterated comultiplication $\tan F \too (\tan F)^{\tens n}$
is obtained as: \[ \tan F \tens Fx_1 \tens \cdots \tens Fx_n \oto{\tan F \tens \monoidal} \tan F \tens F(x_1 \tens \cdots \tens x_n) 
\oto{\alpha} F(x_1 \tens \cdots \tens x_n) \oto{\comonoidal} Fx_1 \tens \cdots \tens Fx_n \] 
\end{observation}

\subsubsection{Hopf Algebra Structure}

\begin{proposition} Let $F \colon A \too \V$ be a separable Frobenius monoidal functor of reconstruction type, and suppose that $A$ has left duals.
Then there is a map $S \colon \tan F \too \tan F$ which we think of as a candidate for an antipode, defined with discharged form:
\begin{equation}
\begin{array}{c} \includegraphics[width=350pt]{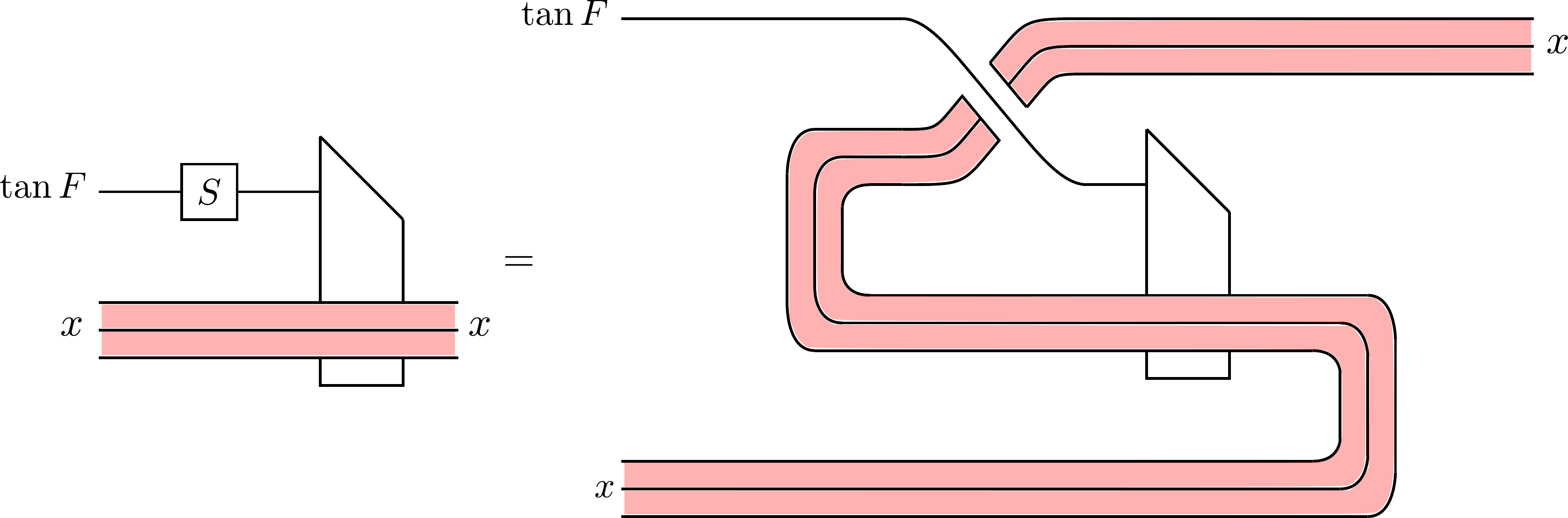} \end{array} \label{def-of-antipode-eq}
\end{equation}
\end{proposition}

Notice in particular how the monoidal and comonoidal structures on $F$ permit one to consider the application of $F$ as not merely ``boxes'' but 
more like a flexible sheath.

As motivation for this graphical notation, compare a more traditionally rendered definition of $S$; as the unique map satisfying:
\[ \bfig

\hscalefactor{1.5}

\node a(-750,+1000)[\tan F \tens Fx]
\node b(0,+1000)[\tan F \tens Fx]
\node c(+750,+1000)[Fx]
\node d(-1000,+500)[\tan F \tens \boot \tens Fx]
\node e(-1000,0)[\tan F \tens F\boot \tens Fx]
\node f(-1000,-500)[\tan F \tens F(x \tens \prep{x}) \tens Fx]
\node g(-1000,-1000)[\tan F \tens Fx \tens F\prep{x} \tens Fx]
\node h(0,-1000)[Fx \tens \tan F \tens F\prep{x} \tens Fx]
\node i(+1000,-1000)[Fx \tens F\prep{x} \tens Fx]
\node j(+1000,-500)[Fx \tens F(\perp{x} \tens x)]
\node k(+1000,0)[Fx \tens F\boot]
\node l(+1000,+500)[Fx \tens \boot]

\arrow[a`b;S \tens Fx]
\arrow[b`c;\alpha x]
\arrow[a`d;\simeq]
\arrow[d`e;\tan F \tens \monnul \tens Fx]
\arrow[e`f;\tan F \tens F\unit \tens Fx]
\arrow[f`g;\tan F \tens \comonoidal \tens Fx]
\arrow|b|[g`h;b \tens F\prep{x} \tens Fx]
\arrow|b|[h`i;Fx \tens \alpha \prep{x} \tens Fx]
\arrow[i`j;Fx \tens \monoidal]
\arrow[j`k;Fx \tens F\counit]
\arrow[k`l;Fx \tens \comnul]
\arrow[l`c;\rulaw^{-1}]

\efig \]

Among other things, for $S$ to be well-defined in this way we must show that the long lower composite is natural in $x$; when rendered graphically, this is immediate, even though
a careful proof of this fact requires the naturality of $\alpha$, the naturality of the binary monoidal and comonoidal structure maps, the dinaturality of the unit and counit
maps in $A$, and the naturality of the braid.

Different treatments disagree about whether or not is necessary for the antipode $S \colon H \too H$ of a Hopf or weak Hopf algebra to be composition invertible.
The above definition seems \emph{not} to be invertible, in general. However, if, in addition to left duals, the category $A$ also has \emph{right}
duals, then one can define an analogous map $S^{-1} \colon H \too H$, using a ``Z-bend'' instead of an ``S-bend'' in the functor region; which the reader may
verify is an inverse to this map.

\subsection{Verification of Axioms}

Having defined all the various structural maps, we now see how they fit together to make bialgebras, weak bialgebras, Hopf algebras, and weak Hopf algebras.

\begin{theorem} \label{weak-bialgebra-thm}
Let $F\colon A \too \V$ be a separable Frobenius monoidal functor of reconstruction type. Then, with algebra structure defined by 
Equations~\ref{def-of-multiplication-eq}~and~\ref{def-of-unit-eq} and coalgebra structure defined by Equations~\ref{def-of-comultiplication-eq}~and~\ref{def-of-counit-eq},
$\tan F$ is a weak bialgebra. \end{theorem}

\begin{proof}
 
 
First, we verify the Bialgebra Axiom (Equation~\ref{weak-bialgebra-axiom}) by the following computations:
\fullwidth{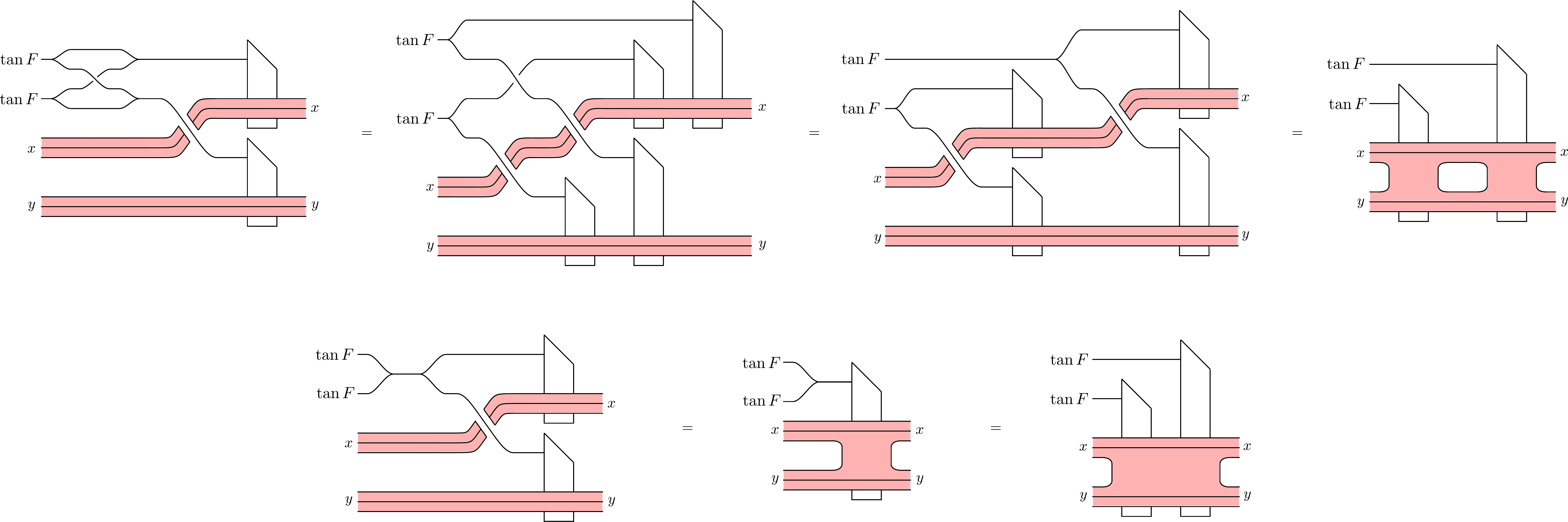} 
Comparing these shows that it suffices to know $F(x \tens y) \oto{\comonoidal} Fx \tens Fy \oto{\monoidal} F(x \tens y)$ should be 
the identity; this is separability of $F$.
 

Second, we verify the Weak Unit Axioms (Equations~\ref{weak-unit-axioms}). In discharged form, the first unit expression is calculated as:
\fullwidth{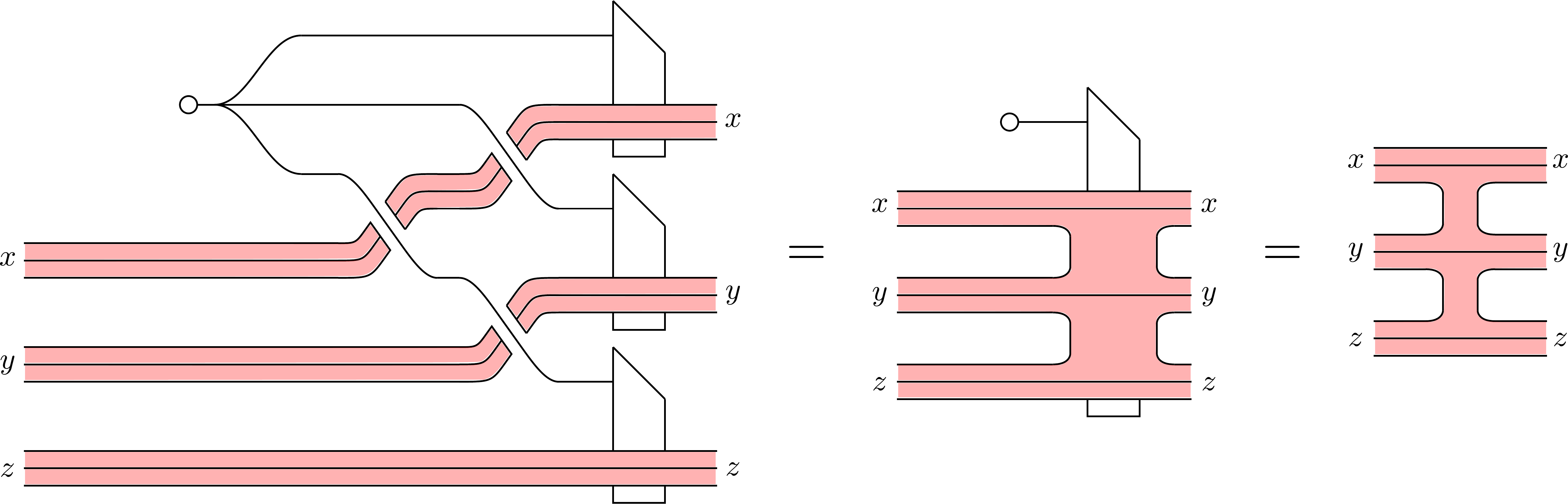}

The calculations in Figure~\ref{weakunit23} show that the second and third unit expressions have the following discharged forms:
\fullwidth{weakunit23abbrev.pdf}
For these unit axioms, we see that it suffices to assume that $F$ is Frobenius monoidal.
\fullwidthfigure{weakunit23.pdf}{Weak unit calculations. In both calculations, the equalities hold by: definition of the multiplication
of $\tan F$; braid axioms; the definition of the comultipliation of $\tan F$; and, finally, the definition of the unit of $\tan F$.}{weakunit23}

Finally, we verify the Weak Counit Axioms (Equations~\ref{weak-counit-axioms}). The discharged form of the first of these is easily calculated:
\fullwidth{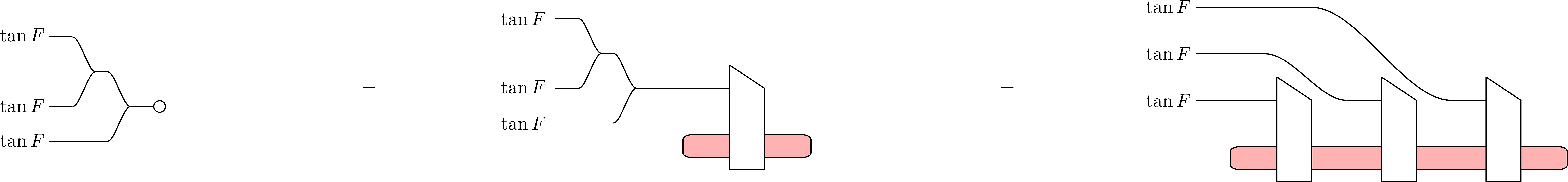}
The discharged forms of the second and third counit expression are computed
in Figure~\ref{weakcounit23}; they are equal, as desired. Examining this figure shows that the counit axioms follow merely from $F$ being both 
monoidal and comonoidal, without requiring $F$ to be Frobenius monoidal or separable.
\fullfigure{weakcounit23.pdf}{Weak counit calculations}{weakcounit23} 
This completes the proof. \end{proof}
This asymmetry between the verifications of the Weak Unit and the Weak Counit Axioms results from working with the \emph{covariant} Tannaka object $\tan F$;
had we instead used the contravariant Tannaka object, $\cot F$, the situation would be reversed.

\begin{corollary} \label{bialgebra-corollary} Let $F \colon A \too \V$ be a separable Frobenius monoidal functor of reconstruction type. If $F$ is moreover strong monoidal, then
the weak bialgebra $\tan F$ constructed in Theorem~\ref{weak-bialgebra-thm} is, in fact, a (non-weak) bialgebra.\end{corollary}

\begin{proof} As shown by B\"ohm, Nill, and Szlach\'anyi~(\cite{BohmNillSzlachanyi99}, page~5), to show that a weak bialgebra is a bialgebra, 
it suffices to show that the Barbell is trivial (Equation~\ref{barbell-axiom}) and \emph{either} the Strong
Unit Axiom (Equation~\ref{unit-axiom}) or the Strong Counit Axiom (Equation~\ref{counit-axiom}) holds. 

We compute that the barbell of $\tan F$ is:
\begin{center}\includegraphics[width=350pt]{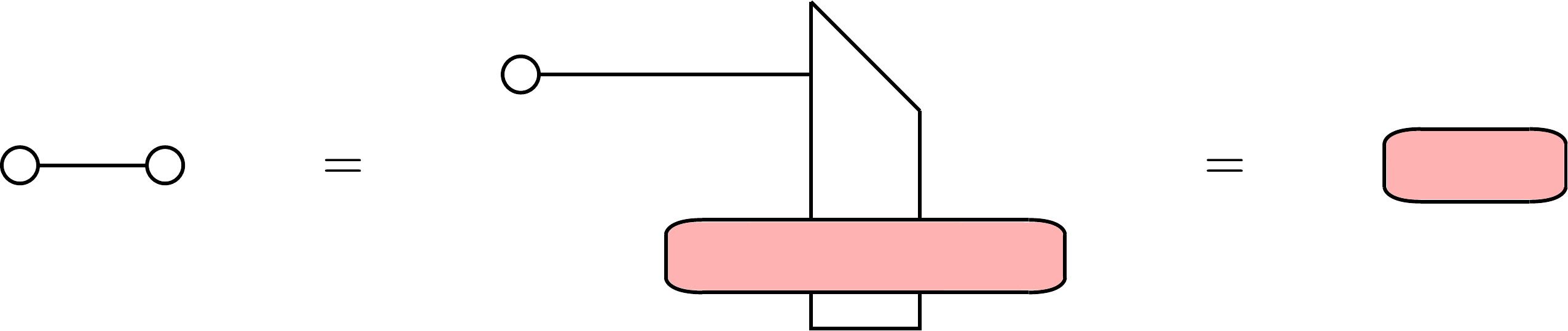}\end{center}
That is, the barbell is the composite $\boot \oto{\monnul} F\boot \oto{\comnul} \boot$, which is the identity when $F$ is strong.

We choose to establish the Strong Counit Axiom (Equation~\ref{counit-axiom}), using the following two calculations: 
\fullwidth{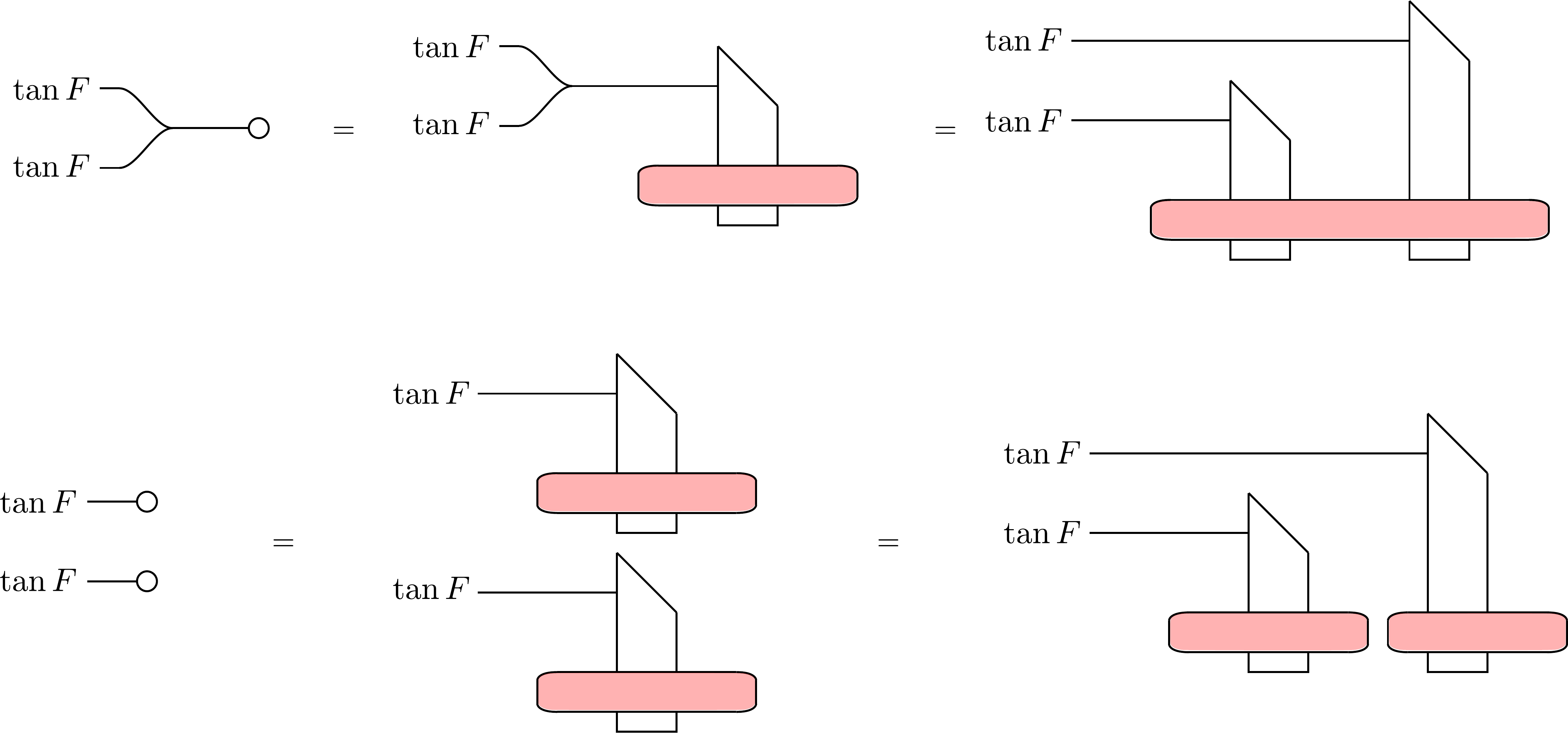} 
and we see that for these two to be equal, it suffices to have $F\boot \oto{\comnul} \boot \oto{\monnul} F\boot$ be the identity; which is the case if $F$ is strong.
\end{proof}

\noindent It is equally easy (albeit longer) to verify the bialgebra axioms (Equations~\ref{barbell-axiom},~\ref{unit-axiom},~\ref{counit-axiom},~and~\ref{bialgebra-axiom})
directly.

\subsubsection{Hopf Algebras and Weak Hopf Algebras}

\begin{theorem} \label{weak-hopf-algebra-thm} Let $F\colon A \too \V$ be a separable Frobenius monoidal functor of reconstruction type, and let $\tan F$ be the weak bialgebra constructed as
in Theorem~\ref{weak-bialgebra-thm}. If $A$ has left duals, 
then the definition of $S$ in Equation~\ref{def-of-antipode-eq} equips the weak bialgebra $\tan F$ with a weak Hopf algebra structure. 
\end{theorem}
\begin{proof}
From Theorem~\ref{weak-bialgebra-thm}, we know that $\tan F$ is a weak bialgebra; we must simply verify the three Weak Antipode Axioms (Equations~\ref{weak-antipode-axioms}).
The pair of calculations in Figure~\ref{antipodeconvolutions} compute the discharged forms of $S \star \tan F$ and $\tan F \star S$; and the
discharged forms of the idempotents $r$ and $t$ are computed in Figure~\ref{ants}. Comparing the two figures shows $S \star \tan F = r$ and $\tan F \star S = t$
as desired.
\fullfigure{antipodeconvolutions.pdf}{Calculations of $S \star \tan F$ and $\tan F \star S$}{antipodeconvolutions}	
\fullfigure{ants.pdf}{``Source'' and ``Target'' maps}{ants}		
Finally, we must show that $S \star \tan F \star S = S$; this is shown in Figures~\ref{triples1}~and~\ref{triples2}.
\fullfigure{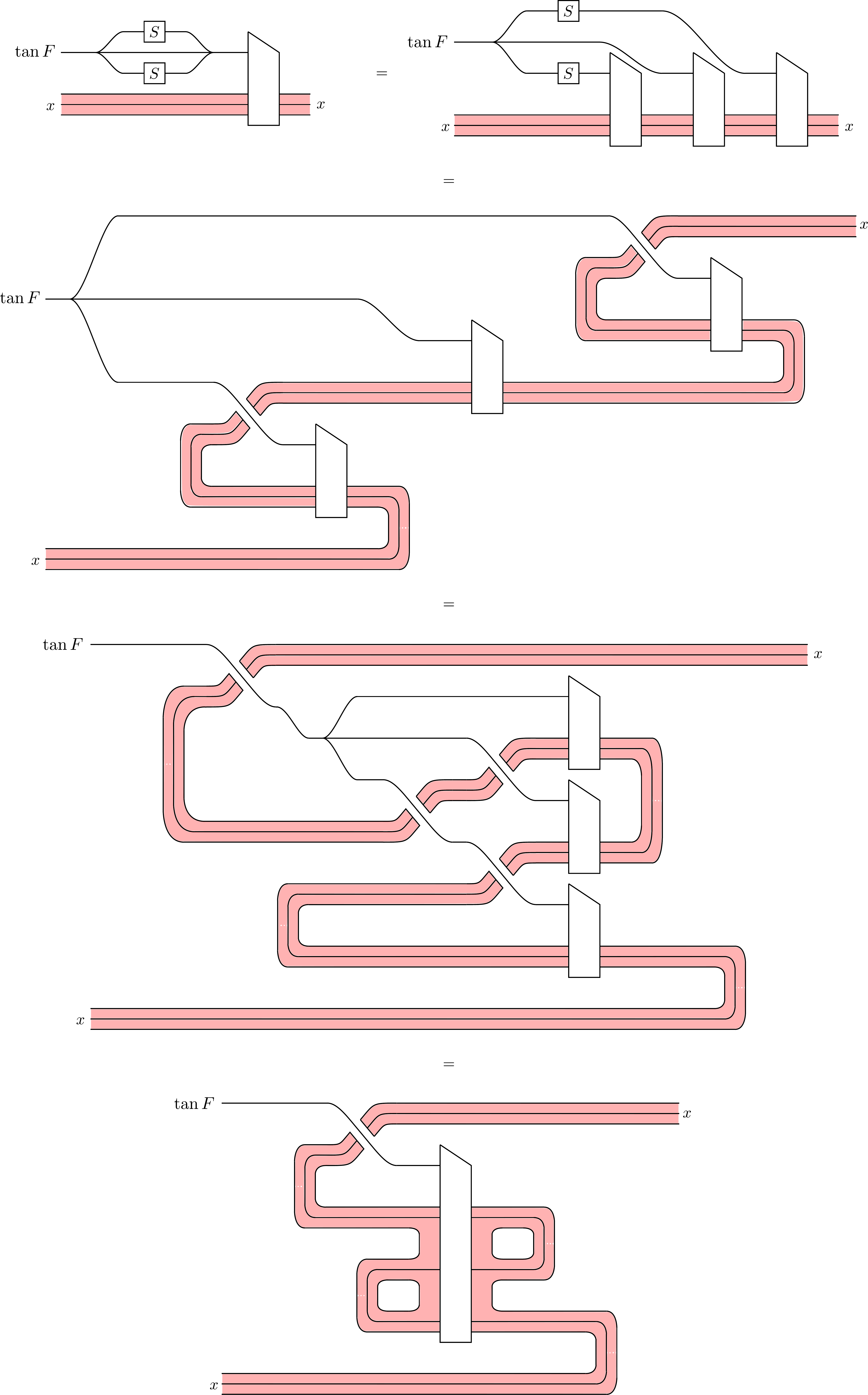}{The calculation showing $S \star \tan F \star S = S$ (part 1 of 2). The equalities hold by: definition of the	
multiplication on $\tan F$; the definition of the antipode on $\tan F$; a slew of naturalities and braid axioms; and, finally,
the definition of the comultiplication.}{triples1}
\fullfigure{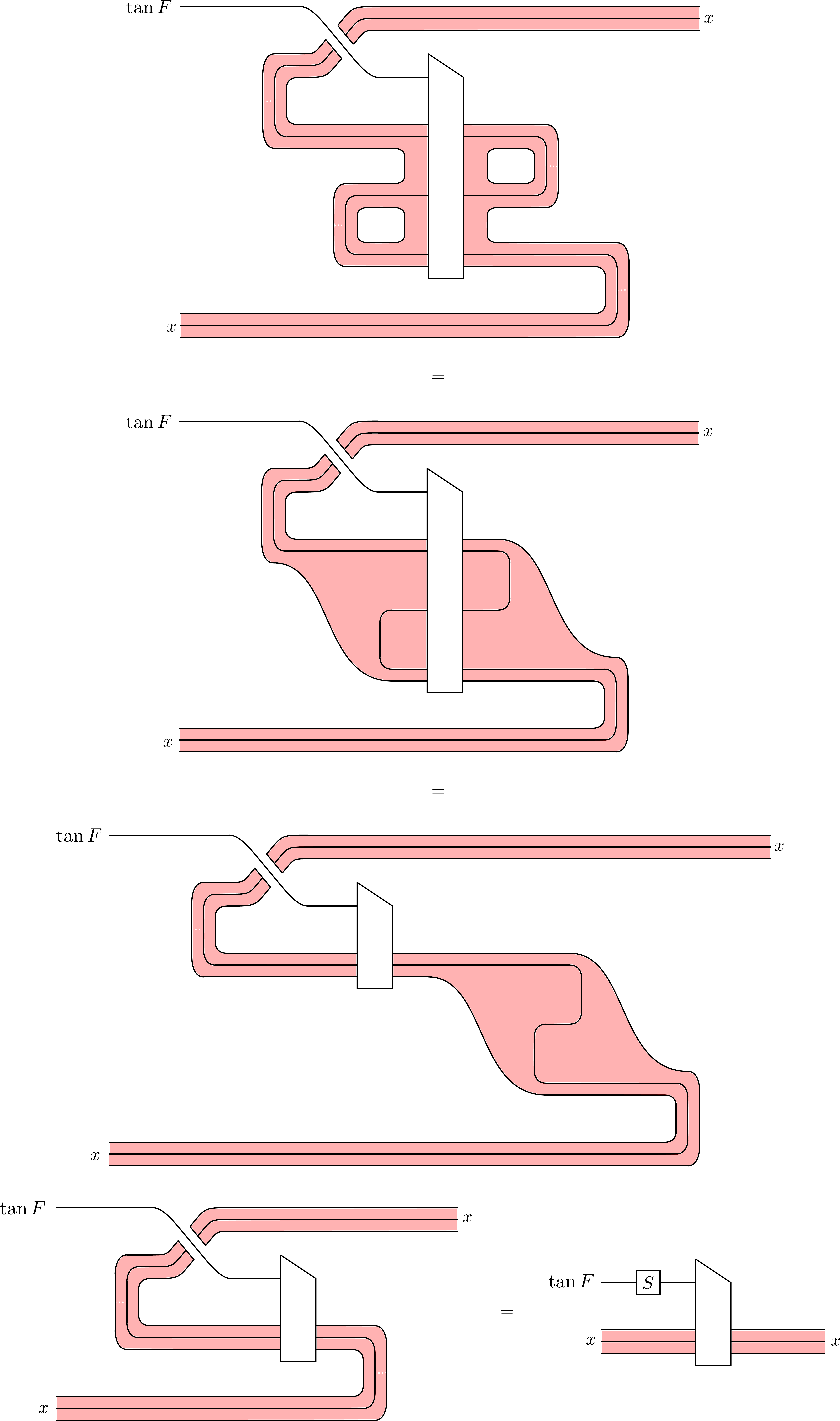}{The calculation showing $S \star \tan F \star S = S$ (part 2 of 2). The equalities hold by: two instances of 	
separability of $F$ and one each of $F$ being monoidal and comonoidal; naturality of $\alpha$; a triangle identity in $A$; and, finally, the definition
of the antipode of $\tan F$.}{triples2}
\end{proof}

\begin{corollary} Let $F \colon A \too \V$ be a separable Frobenius monoidal functor of reconstruction type, and suppose that $A$ has left duals. If $F$ is moreover strong
monoidal, then the weak Hopf algebra $\tan F$ constructed in Theorem~\ref{weak-hopf-algebra-thm} is a (non-weak) Hopf algebra.\end{corollary}

\begin{proof} From Corollary~\ref{bialgebra-axiom}, we know that $\tan F$ is a bialgebra when $F$ is strong monoidal. Therefore, the canonical idempotents $r$ and $t$ which
appear in the weak antipode axioms are both equal to the convolution identity, $\eta\epsilon$, and thus the weak antipode axioms (Equations~\ref{weak-antipode-axioms}) 
degenerate into the non-weak antipode axioms (Equations~\ref{antipode-axioms}).
\end{proof}

\section{Representations of Weak Bialgebras and Weak Hopf Algebras}\label{def-of-mod}

Here we recall the theory of the representations of a weak bialgebra, adapted slightly to our purposes from Nill~\cite{Nill99}, B\"ohm and Szlachanyi~\cite{BohmSzlachanyi00}, and
Pastro and Street~\cite{PastroStreet09}.

Let us now suppose that our base category $\V$ has given splittings for idempotents; that is, an equivalence $Q\V \simeq \V$.
Let a weak bialgebra $H$ in $\V$ be given. We consider the category of left $H$-modules, which we 
write as $\Hmod$; its objects are pairs $(a,\alpha)$, where $a$ is an object of $\V$ and $\alpha \colon H \tens a \too a$ is a unital, associative action 
of $H$ on $a$. Its morphisms $f:(a,\alpha) \too (b,\beta)$ are merely morphisms $f \colon a \too b$ in $\V$ which respect $\alpha$ and $\beta$ in the obvious 
way. Certainly this is a perfectly good category and the obvious mapping $(a,\alpha) \mapsto a$ describes (the object-part of) a perfectly good functor  
$U_H \colon \Hmod \too \V$. 
It is an obvious idea to give $\Hmod$ a monoidal product by defining: 
\[ (a,\alpha) \tens_H (b,\beta) = \left( a \tens b, \begin{array}{c} \includegraphics[scale=0.1]{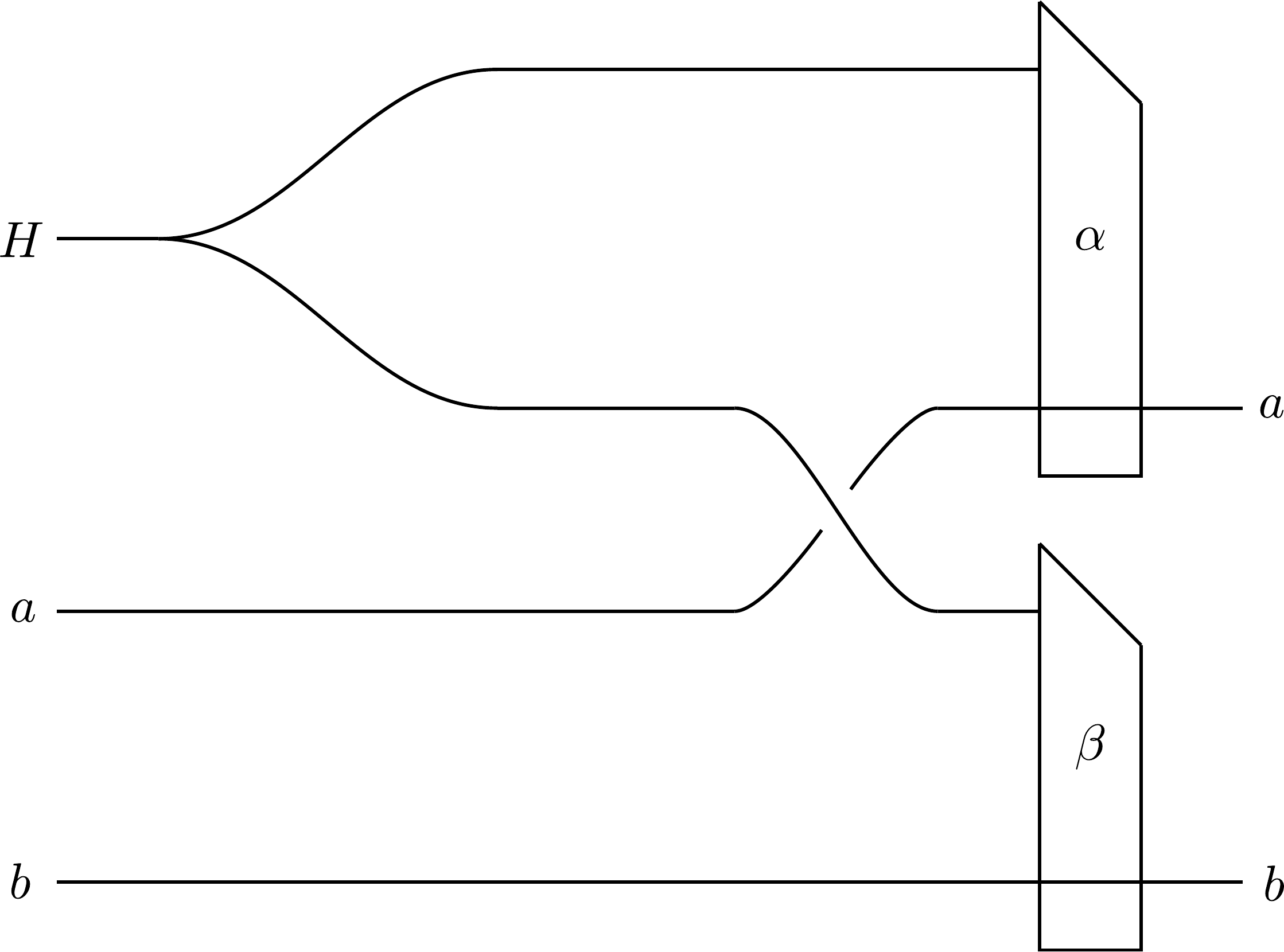} \end{array} \right)  \] 
This action is associative but fails to be unital. To prove that it unital, we would have to show that
\[ \begin{centering} \includegraphics[scale=0.2]{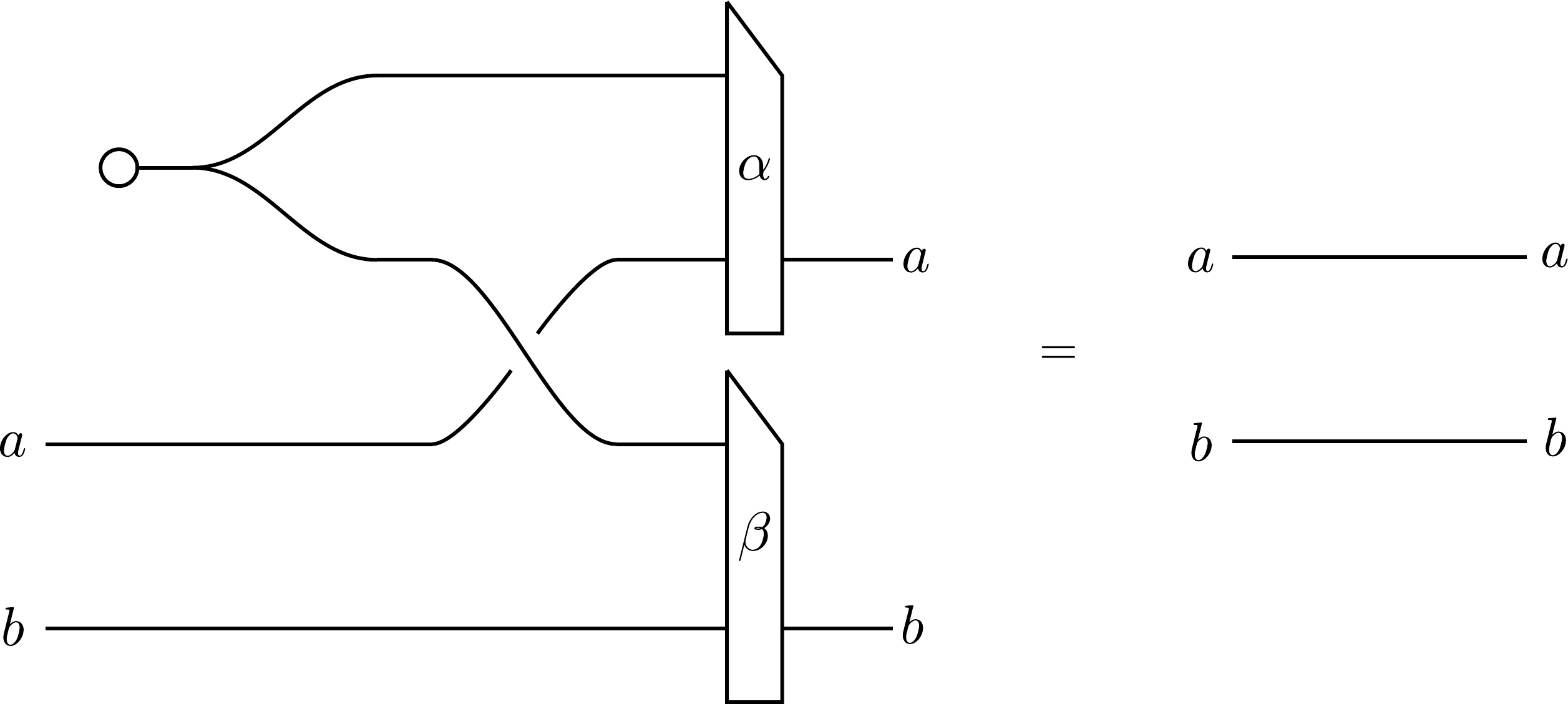} \end{centering} \] 
Since $\begin{array}{c} \includegraphics[scale=0.1]{bifurc-unit.pdf} \end{array}$  
does \emph{not} necessarily hold in a weak bialgebra, this last equality generally does not hold. However, the 
left-hand-side of the above is nevertheless an idempotent on $a \tens b$, as an easy calculation shows. We write this idempotent as $\nabla_{a,b}$, abbreviating 
it to $\nabla$ when context permits. 
 
We define a new category of modules for $H$, which we write as $\Hmod_Q$. The objects of $\Hmod_Q$ are triples 
$(a, \alpha \colon H \tens a \too a, a' \colon a \too a)$, where $a$ is an object of $\V$, where $a'$ is an idempotent on $a$, and where $\alpha$ is
an action which is associative and ``unital-up-to-$a'$''; that is, we insist on $\alpha(\eta \tens a)= a'$. This of course means that $a'$
is redundant; it can be obtained from $\alpha$ and the unit of $H$. Moreover, it can be readily deduced that $a'$ obtained in this way must necessarily
be idempotent and satisfy $\alpha(H \tens a')= \alpha = a'\alpha$.

Now, we can define a monoidal product on $\Hmod_Q$ by: 
\[ \left( a, \pic{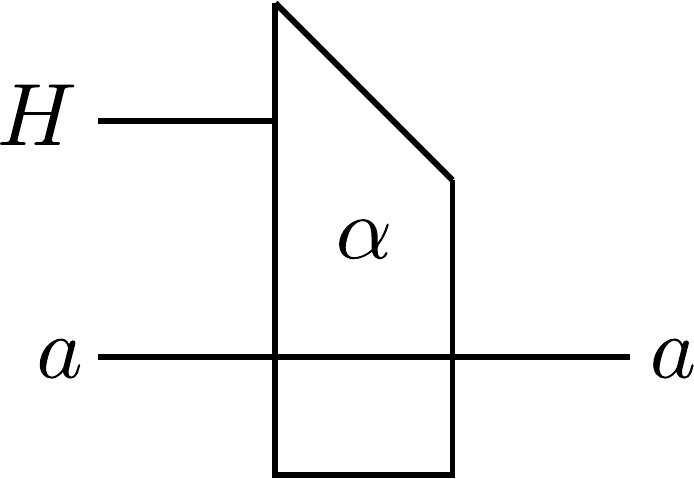},a' \right) \tens_H \left( b, \pic{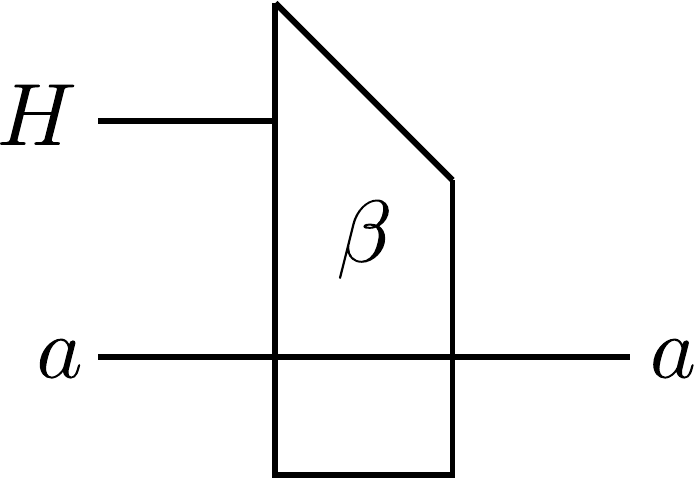}, b' \right)  
= \left( a \tens b, \pic{action-two.pdf}, \nabla_{a,b} \right) \] It may seem surprising to note that $a'$ and $b'$ do not feature on the right-hand
side of this definition; however, since $a'$ satsfies $\alpha(H \tens a')= \alpha = a'\alpha$ (and similarly for $b'$), this is not so strange.

It is routine to verify that the equivalence $Q\V \simeq \V$ lifts to an equivalence
$\Hmod_Q \simeq \Hmod$, but we shall nevertheless continue to work in $Q\V$ and $\Hmod_Q$ for clarity.
 
The unit $\boot_H$ for the above monoidal structure is obtained using the canonical idempotent $t$ defined in 
Section~\ref{canonical-idempotents-section}, namely: 
\[ \boot_H = \left( H, \pic{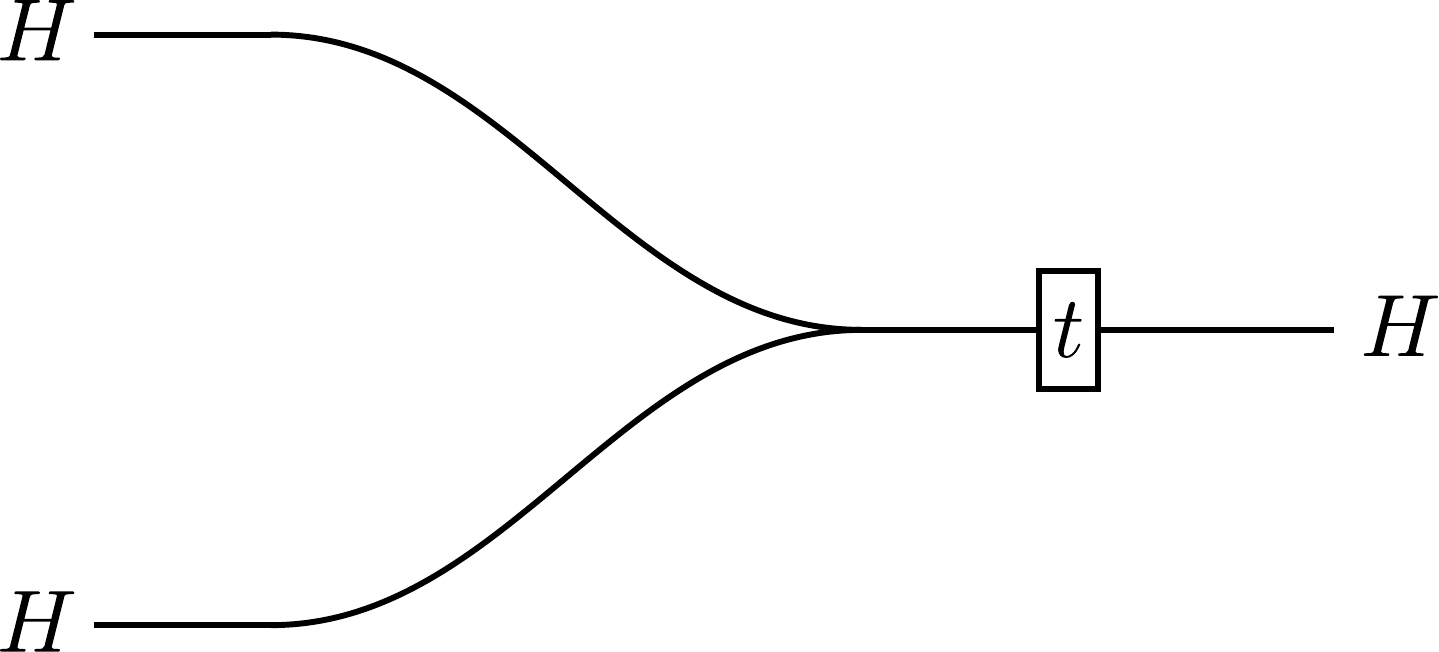}, t \right) \] 
This choice is arbitrary and unimportant, since, as we have remarked above in Proposition~\ref{canonical-idempotents}, all four idempotents are 
isomorphic. However, the precise form of the nullary monoidal constraint isomorphisms will depend on this choice; here, they are:
\[ \includegraphics[scale=0.15]{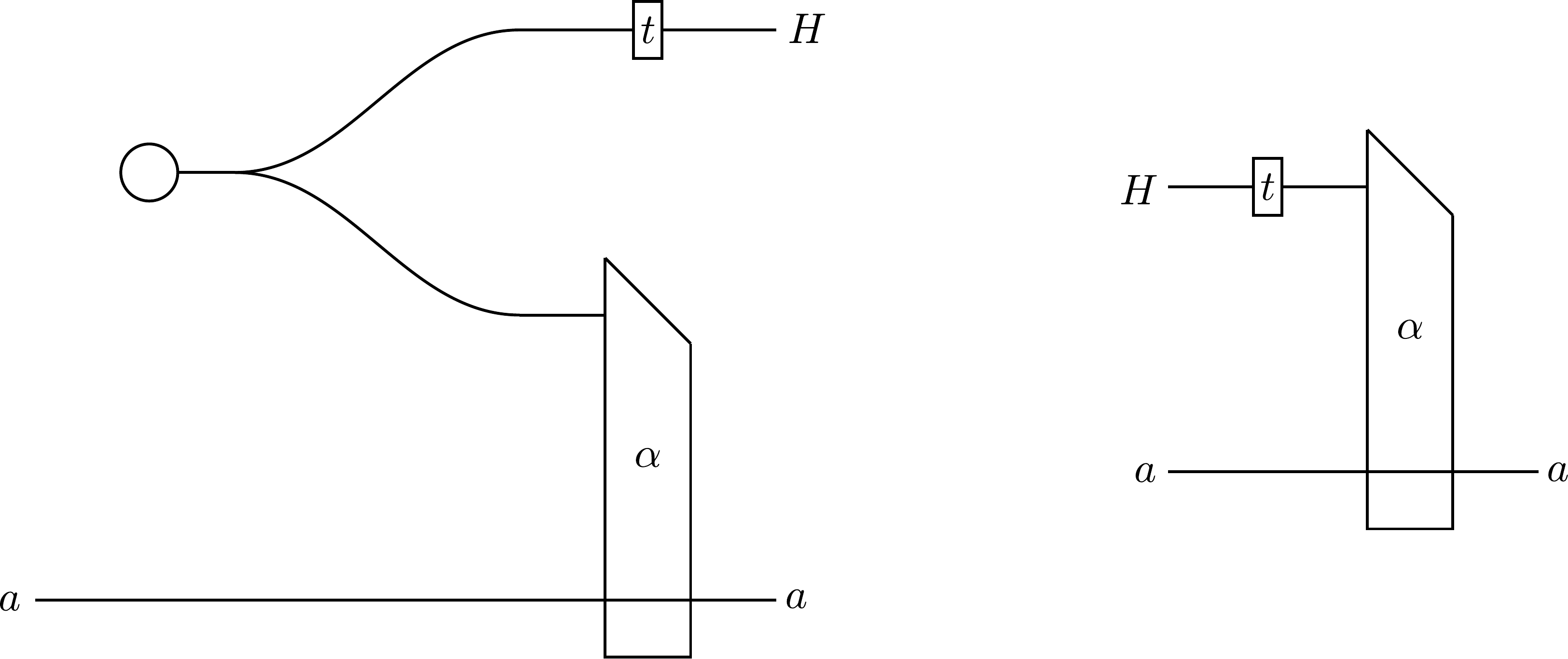} \] 
\[ \includegraphics[scale=0.15]{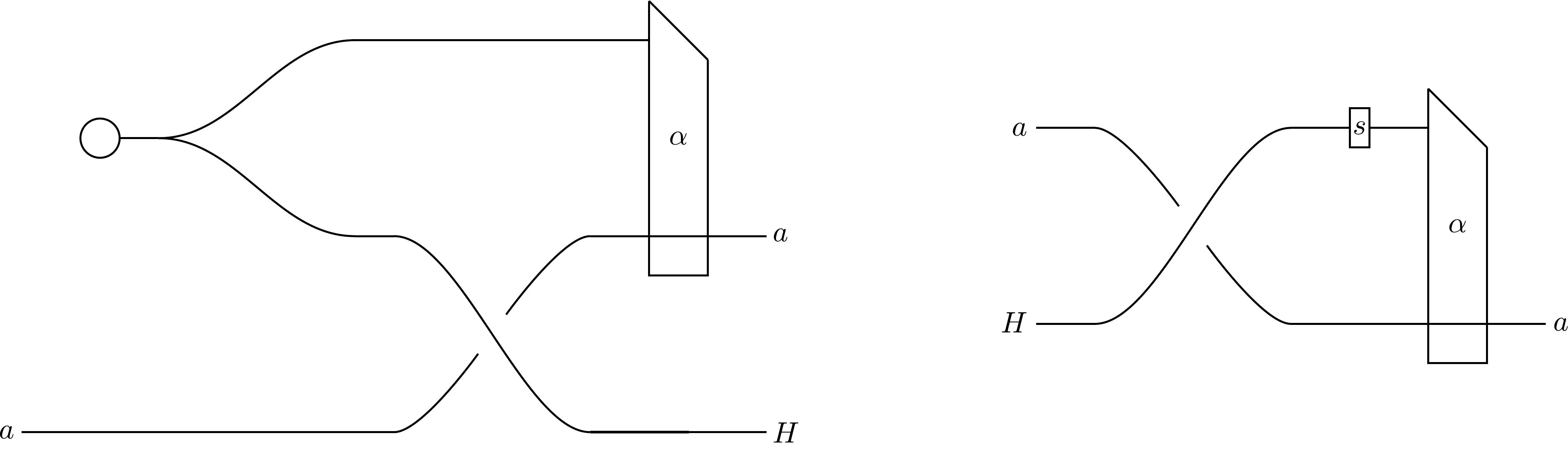} \]
\noindent We omit the (routine) verifications that these are well-defined as maps of actions and maps of idempotents.
 
With these definitions, $U_H \colon \Hmod_Q \too Q\V$ inherits a separable Frobenius monoidal structure, with both binary structure maps given by 
$\nabla$ and nullary structure maps given by:
\[ (\boot,\boot) \oto{\eta} (H,t) = U_H\boot_H \qquad \qquad U_H\boot_H = (H,t) \oto{\epsilon} (\boot,\boot) \]
Verifying the various axioms is routine.


\subsubsection{Representations of Weak Hopf Algebras}\label{rep-weak-hopf}

If our weak bialgebra $H \in \V$ is known to be a weak \emph{Hopf} algebra, then its category of representations $\Hmod$ is ``as autonomous as $\V$ is'';
that is, if an object $a$ has a dual in $\V$, every representation $(a,\alpha \colon H \tens a \too a)$ of $H$ has a dual in $\Hmod$. For details,
see Section~4 of Pastro~and~Street~\cite{PastroStreet09}, although note that the treatment there uses corepresentations instead of representations.
In particular, if $\V$ is autonomous, then $\Hmod_Q$ is also autonomous.

\subsection{Extension of the Tannaka Construction and Representation to Morphisms}\label{extension-to-morphisms}

Given a separable Frobenius monoidal functor $F \colon A \too \V$ of reconstruction type, we have described in Section~\ref{def-of-tan} a method for
obtaining a weak bialgebra $\tan F$ in $\V$. Similarly, given a weak bialgebra $H$ in a braided, monoidally complete category $\V$, the construction in Section~\ref{def-of-mod}
produces a separable Frobenius monoidal functor $U \colon \Hmod \too \V$ of reconstruction type. Of course, we would like to construe these constructions as the object parts
of functors; this will require defining a suitable category of functors into $\V$ and a suitable category of weak bialgebras in $\V$.

\begin{definition} Fix a braided, monoidally complete category $\V$.
Denote by $\recon{\V}$ the category whose objects are separable Frobenius monoidal functors of reconstruction type into $\V$. 
If $F \colon A \too \V$ to $G \colon C \too \V$ are two such functors, then a morphism $H \colon F \too G$ in $\recon{\V}$ is a separable Frobenius monoidal functor (not necessarily of
reconstruction type) $H \colon A \too C$ for which $GH = F$.
\end{definition} Another way to view this category is as the full subcategory of the slice category $\sfm/\V$ determined by the morphisms of reconstruction type; we use the 
``modified slash'' notation to emphasize that $\recon{\V}$ is \emph{not} itself a slice category, since the objects in $\recon{\V}$ are required to be of reconstruction type but
the morphisms are not.

\begin{definition} Fix $\V$ as in the above definition. We denote by $\reconstar{\V}$ the subcategory of $\recon{\V}$ determined by the functors of reconstruction type whose
domains have left duals.
\end{definition}

However, for the morphisms between weak bialgebras, we will need a not-so-well-known notion.
\begin{definition}
Let $H$ and $J$ be weak bialgebras in $\V$, and let $f \colon H \too J$ be an arrow in $\V$. We say that $f$ is a 
\emph{weak morphism of weak bialgebras}~(compare~\cite{Szlachanyi03}, Proposition~1.4; the notion here is the union of the notions there of ``weak left morphism''
and ``weak right morphism'')
if it: \begin{enumerate}
\item Commutes with the four canonical idempotents on $H$ and $J$,
\item Strictly preserves the multiplications and units of $H$ and $J$, and
\item Weakly preserves the comultiplications of $H$ and $J$ in the sense that:
\[ \includegraphics[scale=0.2]{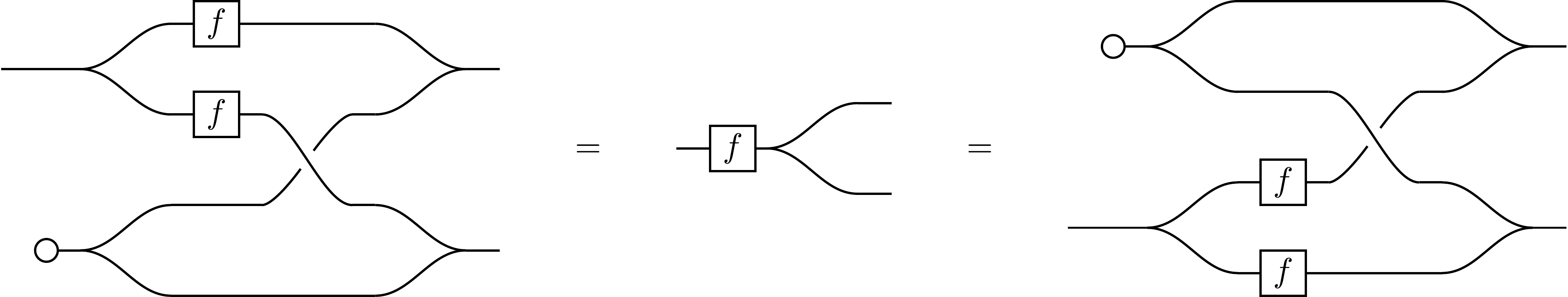} \]
\end{enumerate}
\end{definition}

The asymmetry between the preservation of multiplication and preservation of comultiplication corresponds to the choice of modules instead of
comodules in the representation theory earlier. Had we chosen to work with comodules, we would instead consider the dual notion of morphisms which
strictly preserve the comultiplication and counit but only weakly preserve the multiplication.

It is not too difficult to prove that the composite of two weak morphisms is a weak morphism. The first two conditions pose no difficulty; as for
the third condition, we prove the second equality by the following:

\fullwidth{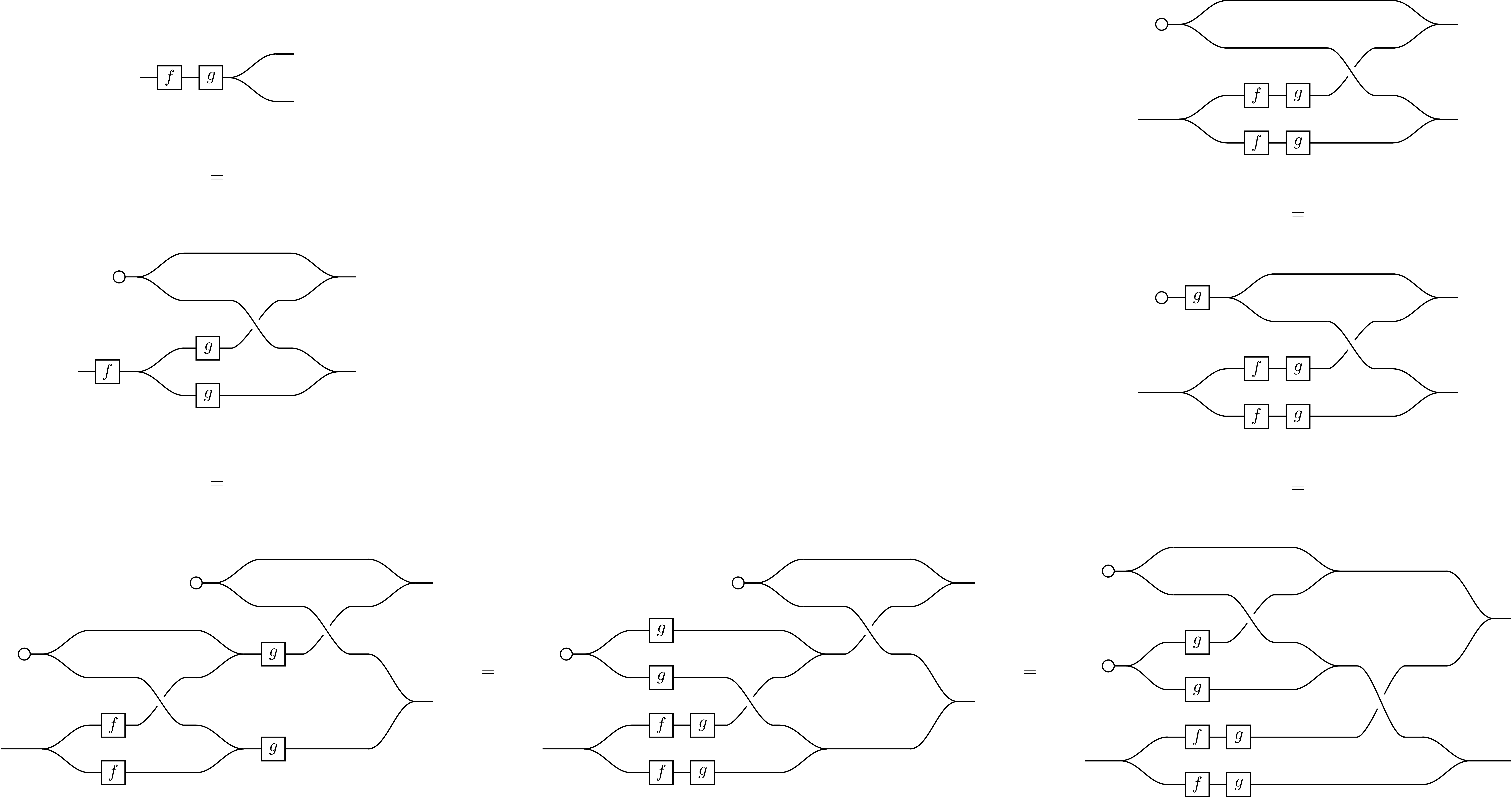}

In counter-clockwise order from top-left, the equalities hold since: $g$ weakly preserves comultiplication; $f$ weakly preserves comultiplication; 
$g$ strictly preserves multiplication; associativity of multiplication and some braid axioms; $g$ weakly preserves comultiplication; $g$ strictly 
preserves units.

The first equality in condition~3 is proved similarly. In sum, weak morphisms between weak bialgebras in a braided monoidal category $\V$ form a category which we
write as $\wba \V$. We define a weak morphism of weak Hopf algebras to be a weak morphism between underlying weak bialgebras, and we denote this category by
$\wha \V$.

\begin{observation} 
Every strong morphism of weak bialgebras (that is, one strictly preserving the units, counits, multiplications and comultiplications)
is a weak morphism of weak bialgebras, and, moreover, if the weak bialgebra is in fact a usual bialgebra, then the notions of weak and strong
morphism coincide. In particular, this means that we have inclusions $\ba \V \too \wba \V$ and $\ha \V \too \wha \V$.
\end{observation}

\subsection{Extension of Tannaka Construction to morphisms}
In this section we extend the Tannaka construction described in Section~\ref{def-of-tan} to a functor \[ \tan \colon \recon{\V} \too \wba \V \]
Suppose that \[ \bfig \Vtriangle[A`C`\V;H`F`G] \efig \] is a morphism $H \colon F \too G$ in $\recon{\V}$.
We must obtain from such a commuting triangle a weak morphism of weak Hopf algebras $\tan H \colon \tan G \too \tan F$. A morphism from $\tan G$ into $\tan F$ is the same thing 
as an action of $\tan G$ on $F$; we take here the canonical action \[ \tan G \tens F = \tan G \tens GH \oto{\alpha H} GH = F \]  
Graphically, we write this as: \[ \includegraphics[scale=0.2]{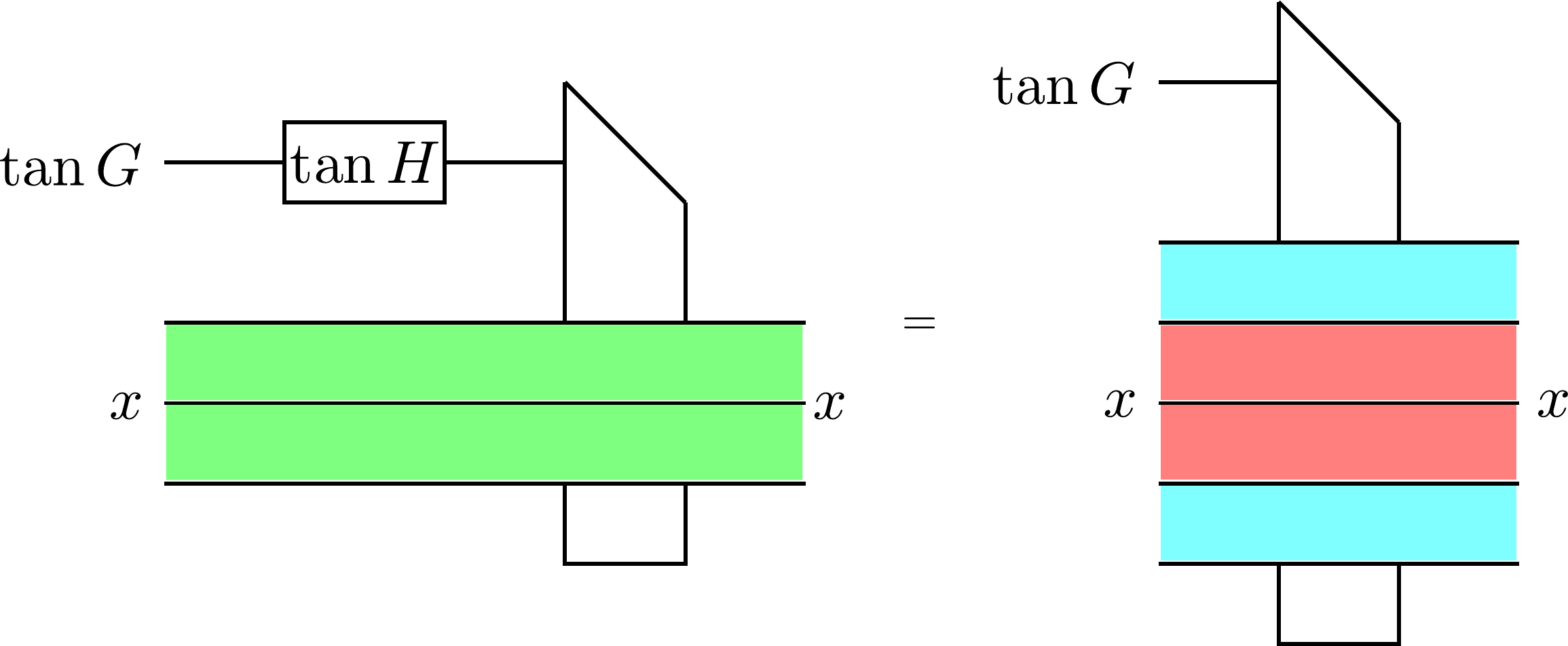} \] where we have written $F$ as green, $H$ as red, and $G$ as blue. 
Note that the boundaries of this definition are equal precisely because $F = GH$. 
 
We must verify that $\tan H$ strictly preserves the monoidal structures of $\tan G$ and $\tan F$ and weakly preserves their comultiplication. 
As for the unit, it is immediate: \[ \includegraphics[scale=0.11]{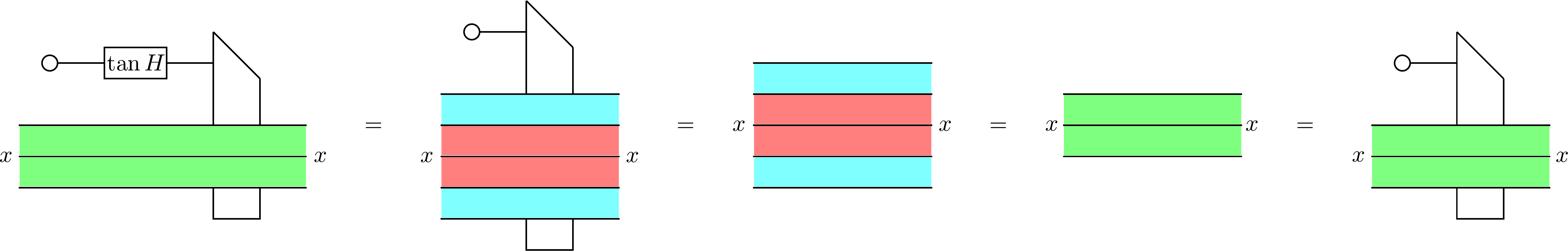} \] 
And the multiplication is similarly easy: 
\fullwidth{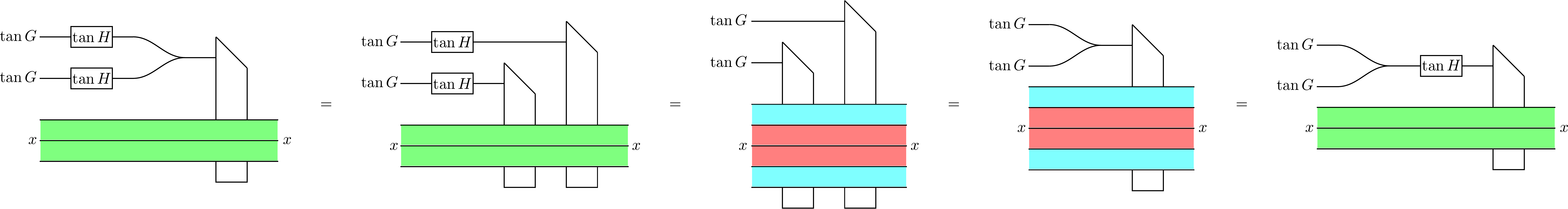}
 
However, as expected for a weak morphism of weak bialgebras, 
$\tan H$ need not strictly preserve the comultiplications. On the one hand, we compute the discharged form 
of $\tan G \oto{\Delta} \tan G \tens \tan G \oto{\tan H \tens \tan H} \tan F \tens \tan F$: 
\fullwidth{tanH-comultiplication-1.pdf}
Whereas, on the other hand, we compute the discharged form of $\tan G \oto{\tan H} \tan F \oto{\Delta} \tan F \tens \tan F$: 
\fullwidth{tanH-comultiplication-2.pdf}
Certainly the above shows that, if $H$ is strong monoidal, $\tan H$ will preserve the comultiplications strictly. 
 
As an aside, we investigate whether $\tan H$ preserves the counits. On the one hand, we compute: 
\[ \includegraphics[scale=0.1]{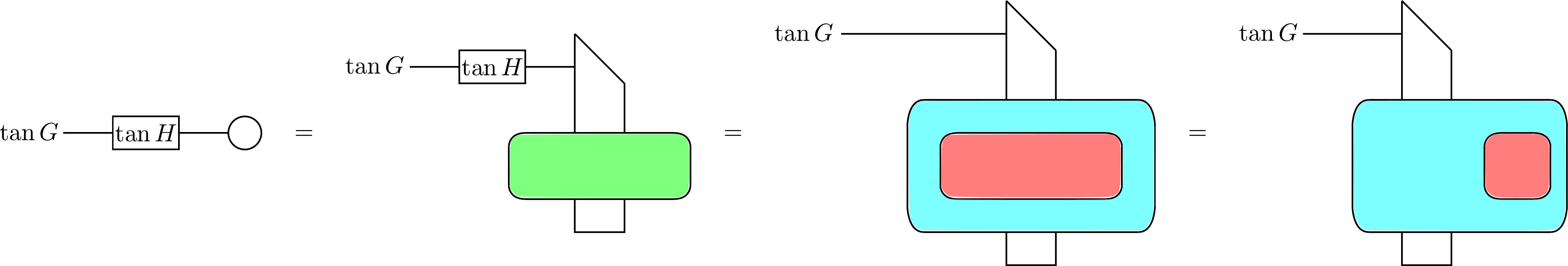} \] 
And on the other hand, we compute: 
\[ \includegraphics[scale=0.1]{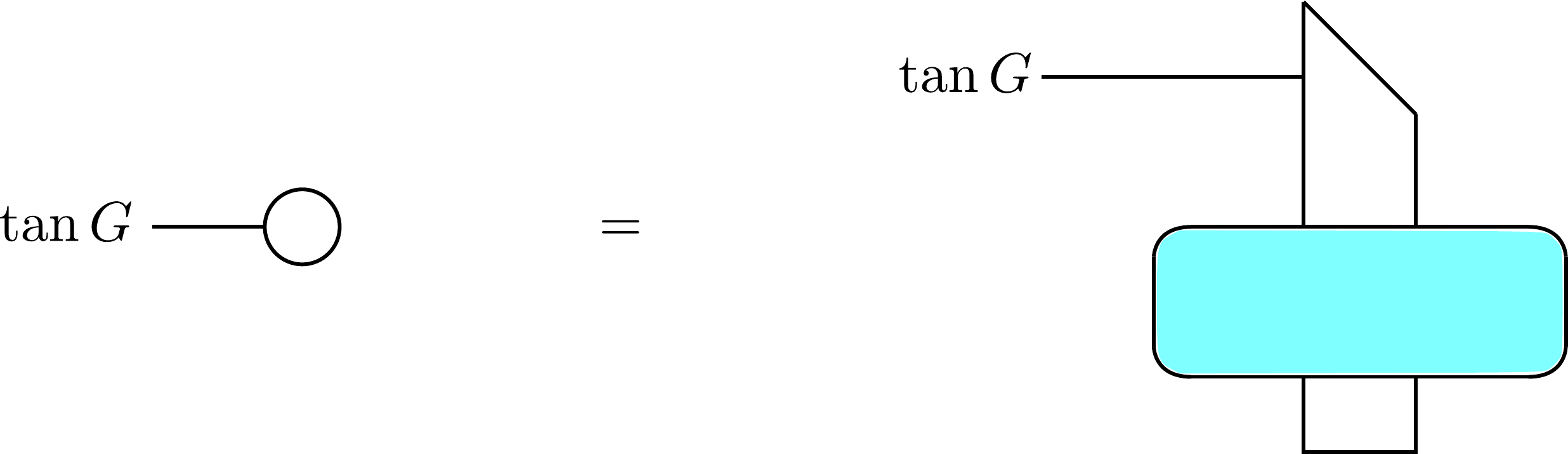} \] 
So we see that, for $\tan H$ to preserve the counits, it suffices for $H$ to be strong; specifically, for the composite 
$\boot \oto{\monnul} H\boot \oto{\comnul} \boot$ to 
be the identity. 
 
We proceed to show that $\tan H$ weakly respects the comultiplications of $\tan G$ and $\tan H$. We show the second equality of Condition~3
in the definition of weak morphism, the first equality is proved similarly. First, we compute the discharged form of 
$\boot \oto{\eta} \tan G \oto{\delta} \tan G \tens \tan G$ as: 
\[ \includegraphics[scale=0.1]{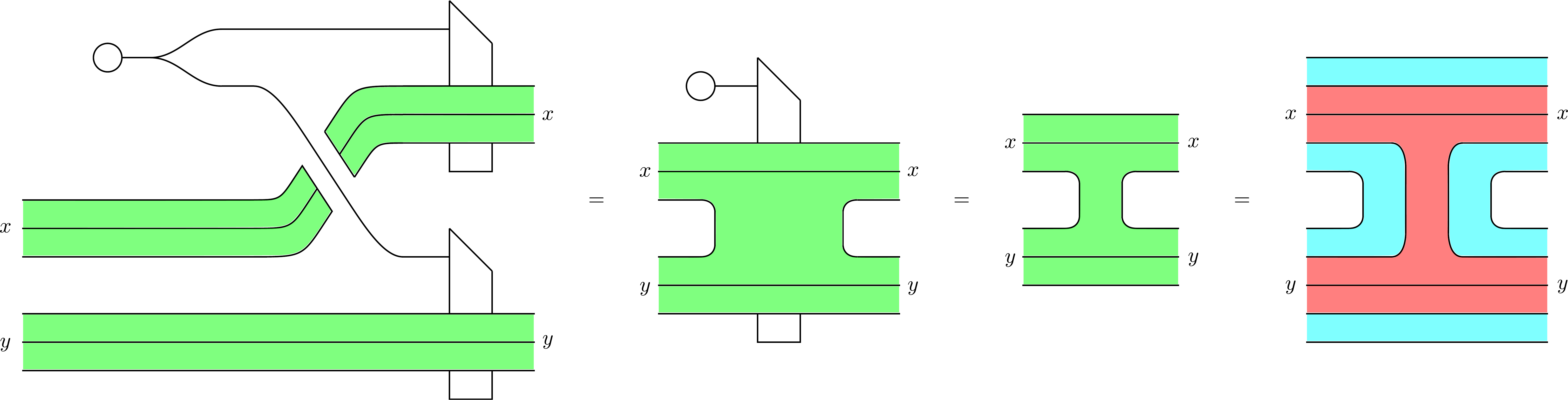} \] 
Second, exploiting the basic fact that the discharged form of a product is the composite of discharged forms, we see that the discharged 
form of \[ \includegraphics[scale=0.1]{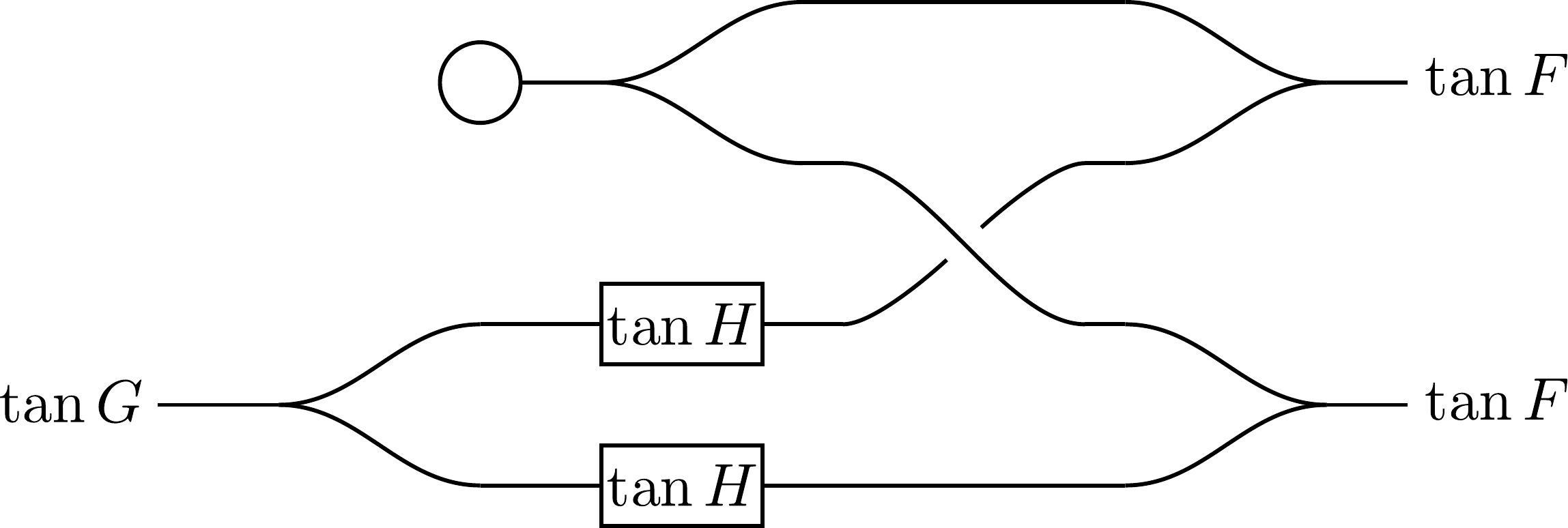} \] is: \fullwidth{tanH-comultiplication-weak.pdf}
where we have used the fact that $G$ is separable followed by the naturality of the canonical action of $\tan G$ on $G$. Thus, $\tan H$
respects the comultiplications of $\tan H$ and $\tan G$ in the sense required of a weak morphism of weak bialgebras.
 

Finally, we must check that $\tan H$ commutes with the four canonical idempotents. 
We show that $(\tan H)r = r(\tan H)$ by the following chain of calculation:
\[ \includegraphics[scale=0.11]{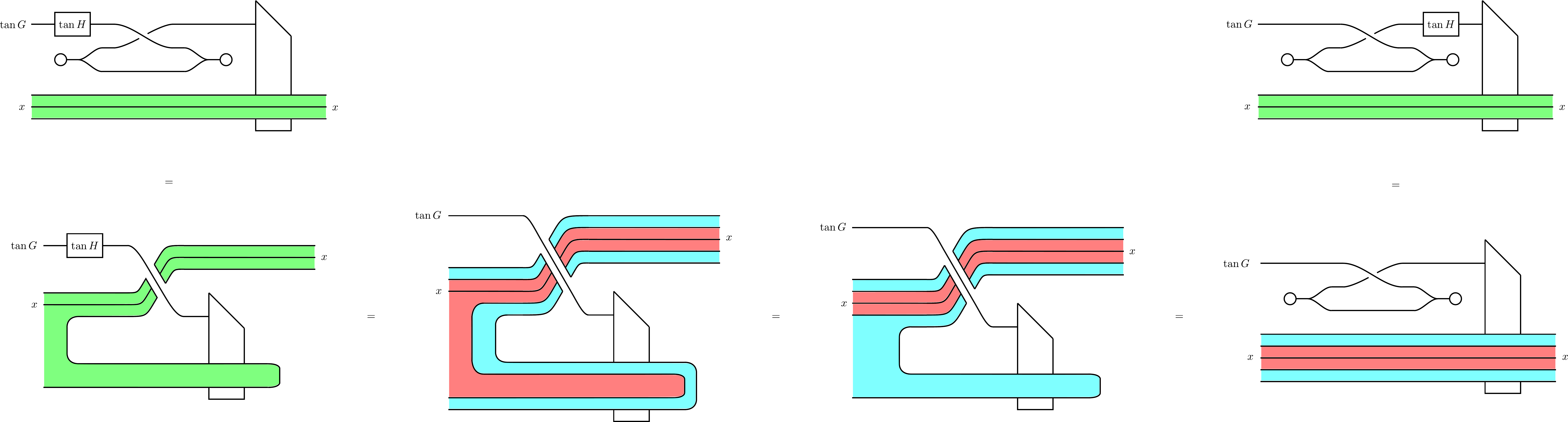} \] Counter-clockwise from top-left, the equalities hold
by: the discharged form of $r$ from the left-hand column of Figure~\ref{ants}; the definition of $\tan H$; naturality of
action and the monoidality of $F$; the discharged form of $r$ once again; and finally the definition of $\tan H$ again.
The proofs that $\tan H$ respects the other three idempotents are similar.

Thus, we have that, for $H$ an arrow in $\recon{\V}$, the arrow $\tan H$ is a weak morphism of weak bialgebras. It is routine to verify that
$\tan$ defined on morphisms in this way preserves composition and identities; hence, we have a functor: 
\[ \displaystyle \tan \colon \recon{\V} \too \left( \wba \V \right)^\op \] And, if we restrict to the full subcategory of $\recon{\V}$ 
consisting of functors with autonomous domain, we have a functor: \[ \displaystyle \tan \colon \reconstar{\V} \too \left( \wha \V \right)^\op \]

\subsection{Extension of the Representation Theory to Morphisms}\label{mod-on-morphisms}

Let $f \colon H \too J$ be a weak morphism of weak bialgebras. We define $f^* = Q(f\text-\mod) \colon \Jmod_Q \too \Hmod_Q$ to have action on objects:
\[ f^*\left(a, \pic{action-alpha.pdf},a' \right) = \left(a,\pic{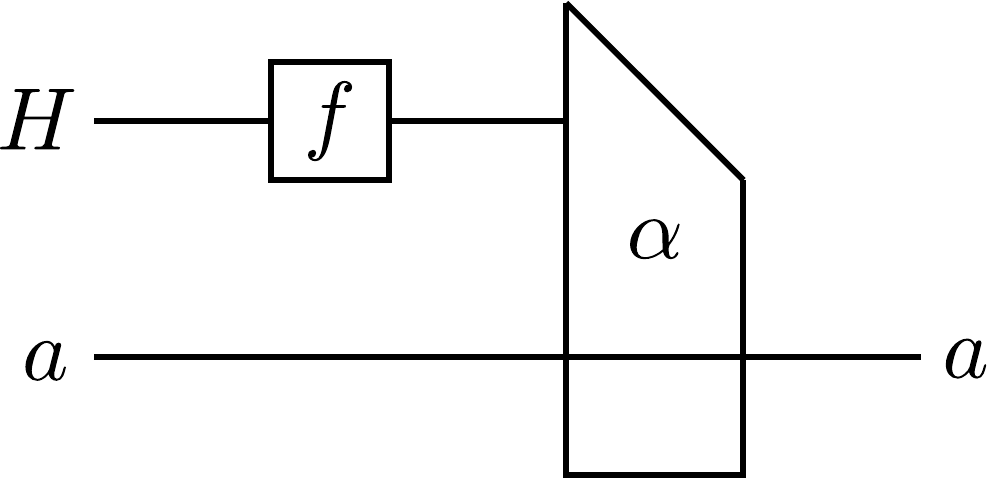},a' \right) \] and to be the identity on morphisms.

Since $f$ strictly preserves the unit and the multiplication, $f^*$ takes associative and unital $J$-modules to associative and unital
$H$-modules, as required. It is clear that, 
as mere functors, $U_Hf^* = U_J$. What is considerably more complicated is the separable Frobenius monoidal structure on $f^*$. Let us agree to 
abbreviate the right-hand side of the above definition as $f^*a$, to simplify notation.

We compute
\begin{align*}
        f^*a \tens_H f^*b &= \left( a, \pic{action-alpha-f.pdf}, a' \right) \tens_H \left( b, \pic{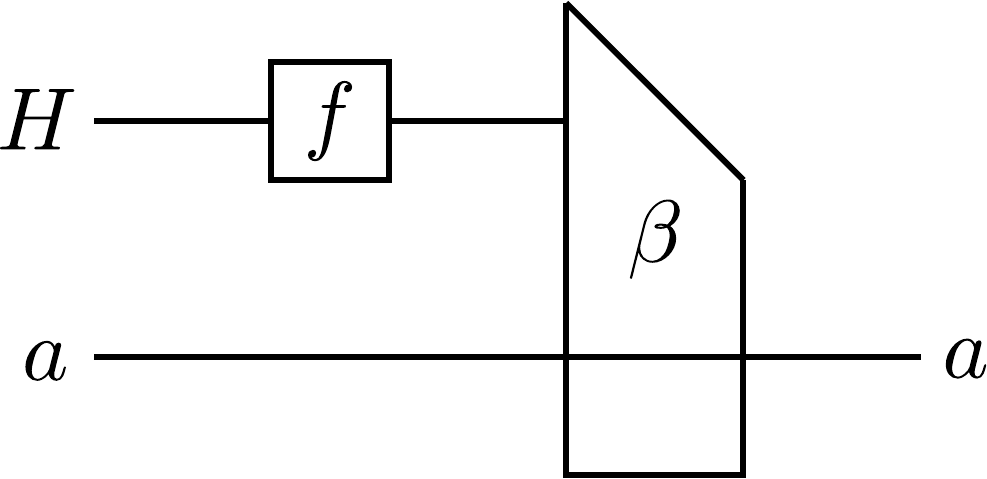}, b' \right) \\
                        &= \left( a \tens b, \pic{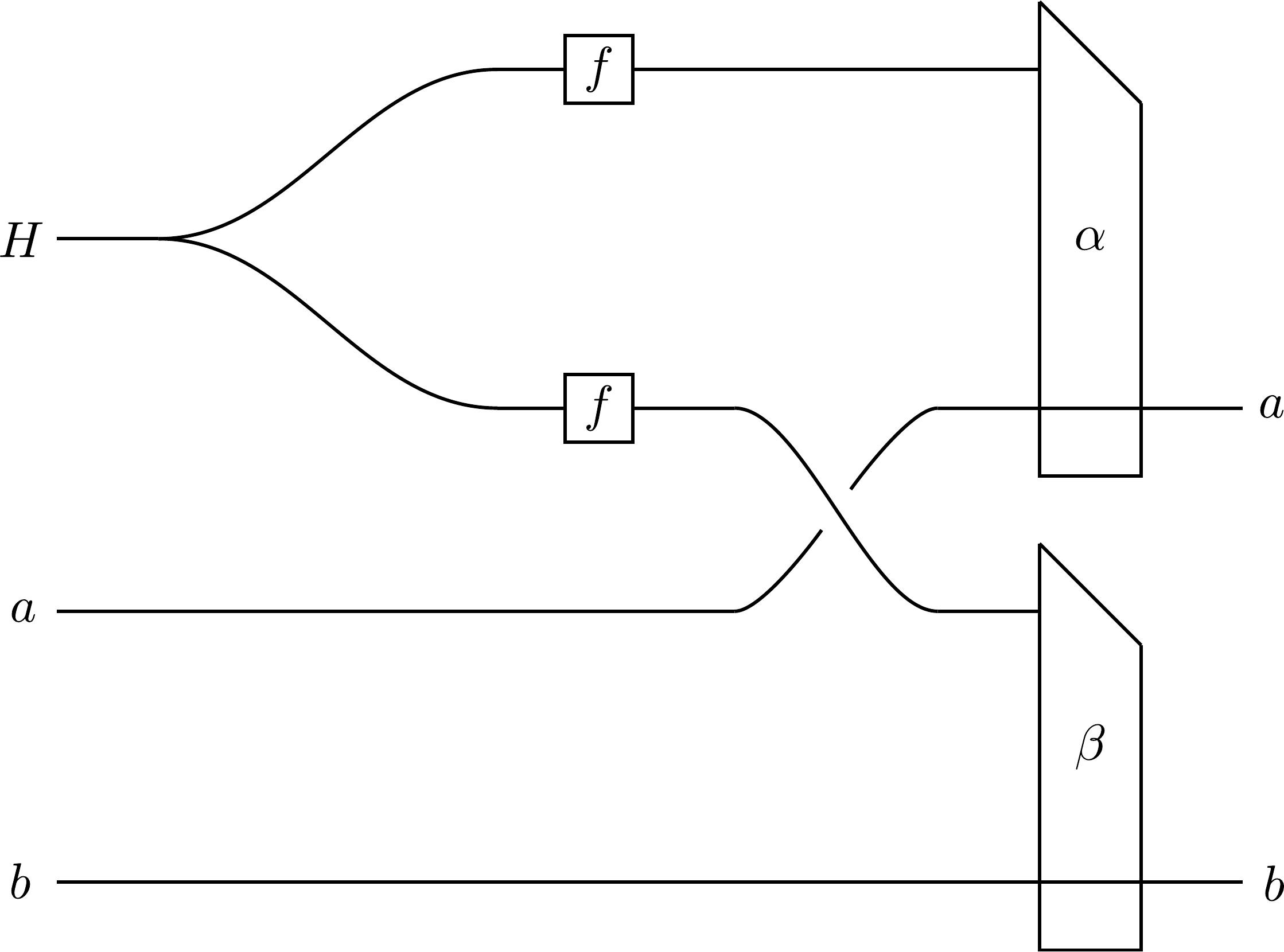}, \pic{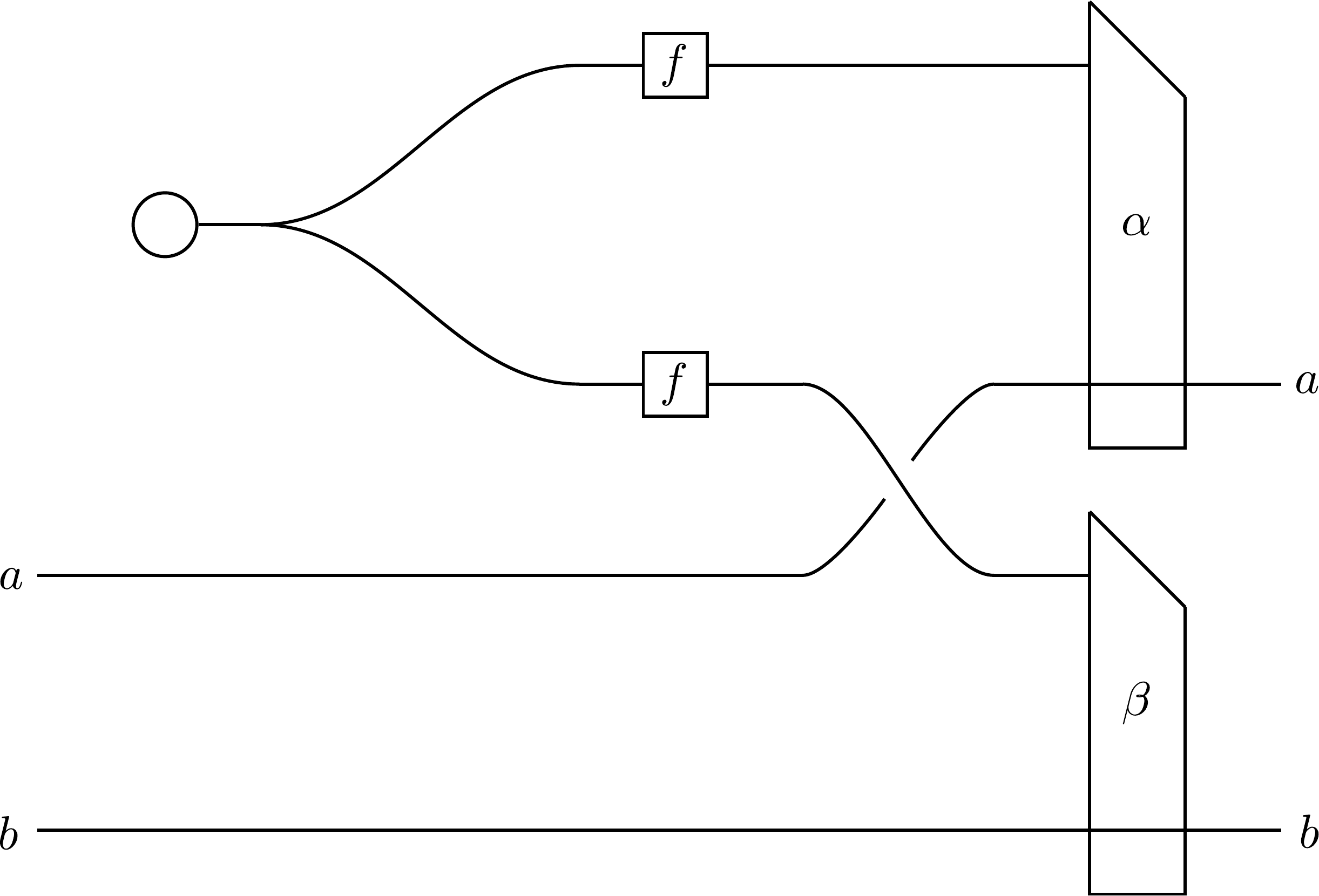} \right) \\
        f^*(a \tens_J b)  &= f^*\left( \left(a,\pic{action-alpha.pdf},a'\right) \tens_J \left(b,\pic{action-beta.pdf},b'\right) \right) \\
                        &= f^*\left(a \tens b \pic{action-two.pdf}, \nabla_{a,b} \right) 
			= \left(a \tens b \pic{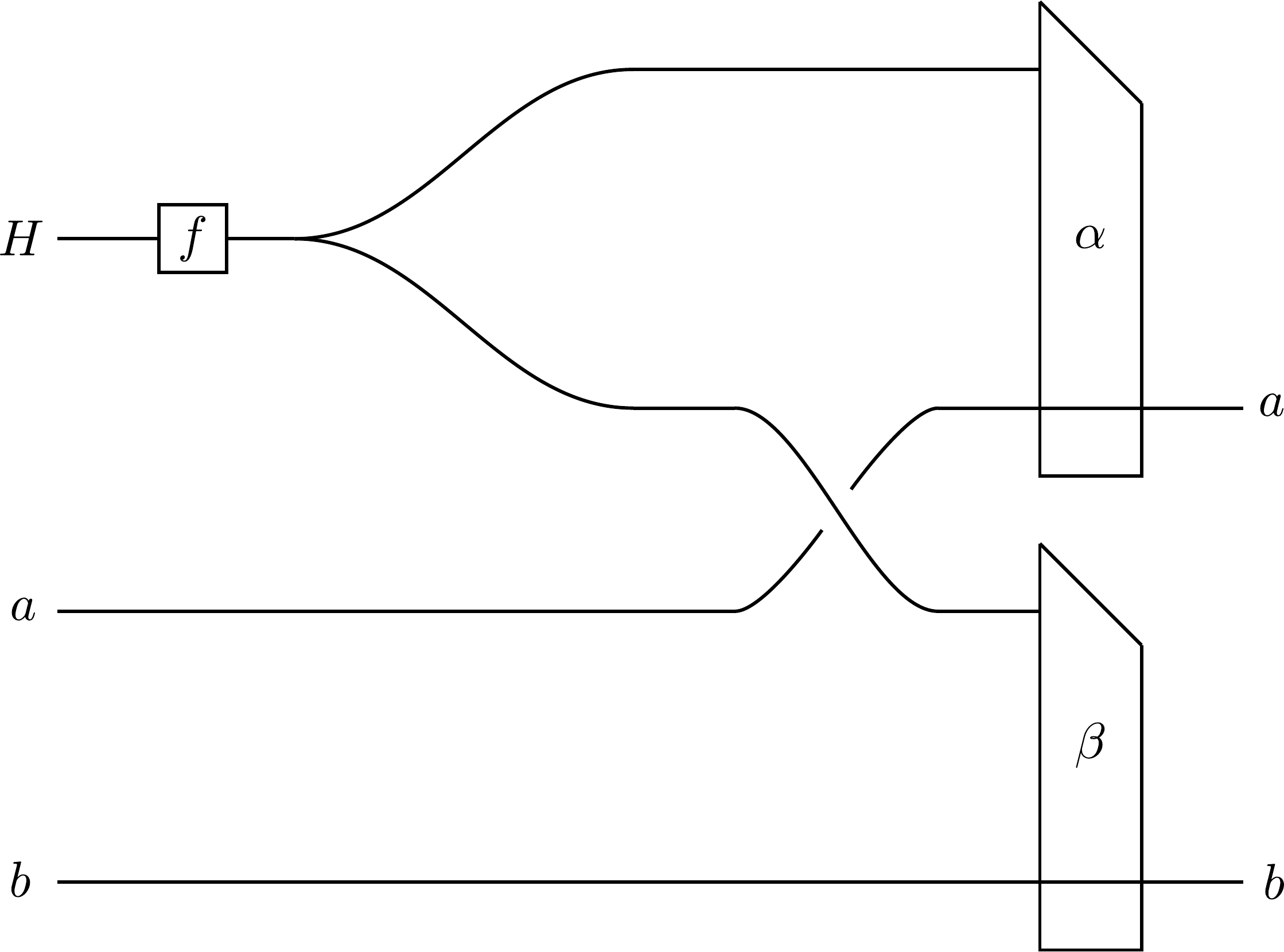}, \nabla_{a,b} \right)\\
\end{align*}

By condition~3 of $f$ being a weak morphism of weak Hopf algebras, we can view $\nabla_{a,b}$ as a monoidal structure 
$f^*a \tens_H f^*b \too f^*(a \tens_J b)$
as well as a comonoidal structure $f^*(a \tens_J b) \too f^*a \tens_H f^*b$. Moreover, this is clearly separable, since the idempotent on
$f^*(a \tens_J b)$ is $\nabla_{a,b}$. However, since the idempotent on $f^*a \tens_H f^*b$ is \emph{not} equal to $\nabla_{a,b}$, the composite
\[ f^*a \tens_H f^*b \too f^*(a \tens_J b) \too f^*a \tens_H f^*b \] is not necessarily the identity.

Furthermore, for the nullary structure, we compute:
\begin{align*}
        \boot_H &= \left( H, \pic{H-as-unit.pdf}, t \right) \\
        f^*\boot_J &= f^*\left( J, \pic{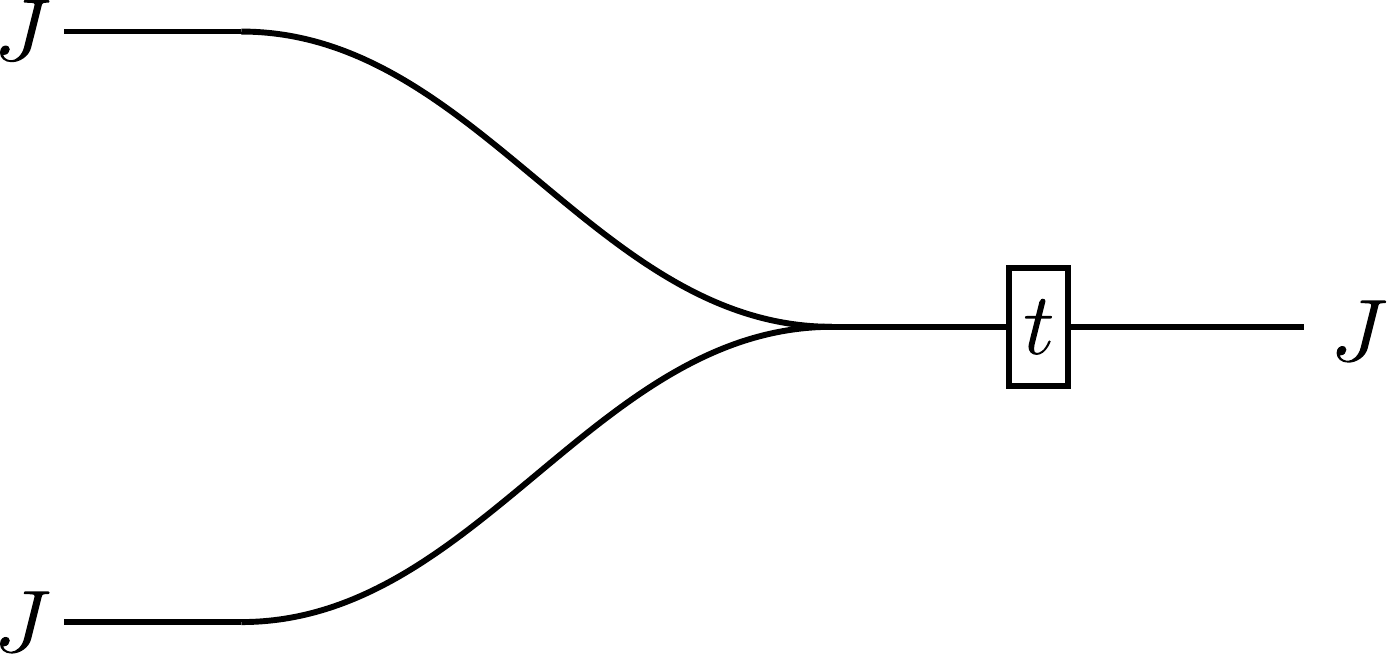}, t \right) \\
                &= \left( J, \pic{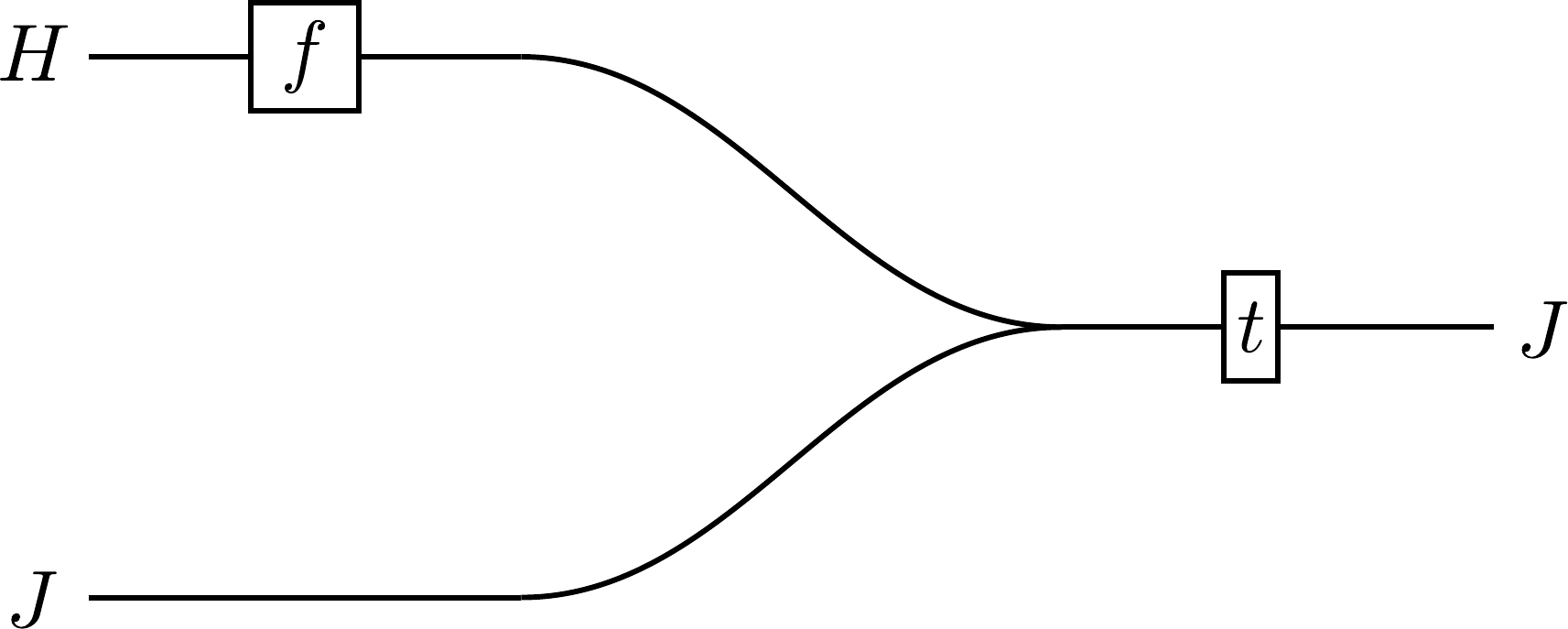}, t \right) \\
\end{align*}

We define $\boot_H \too f^*\boot_J$ to be $ft$ and $f^*\boot_J \too \boot_H$ to be $\pic{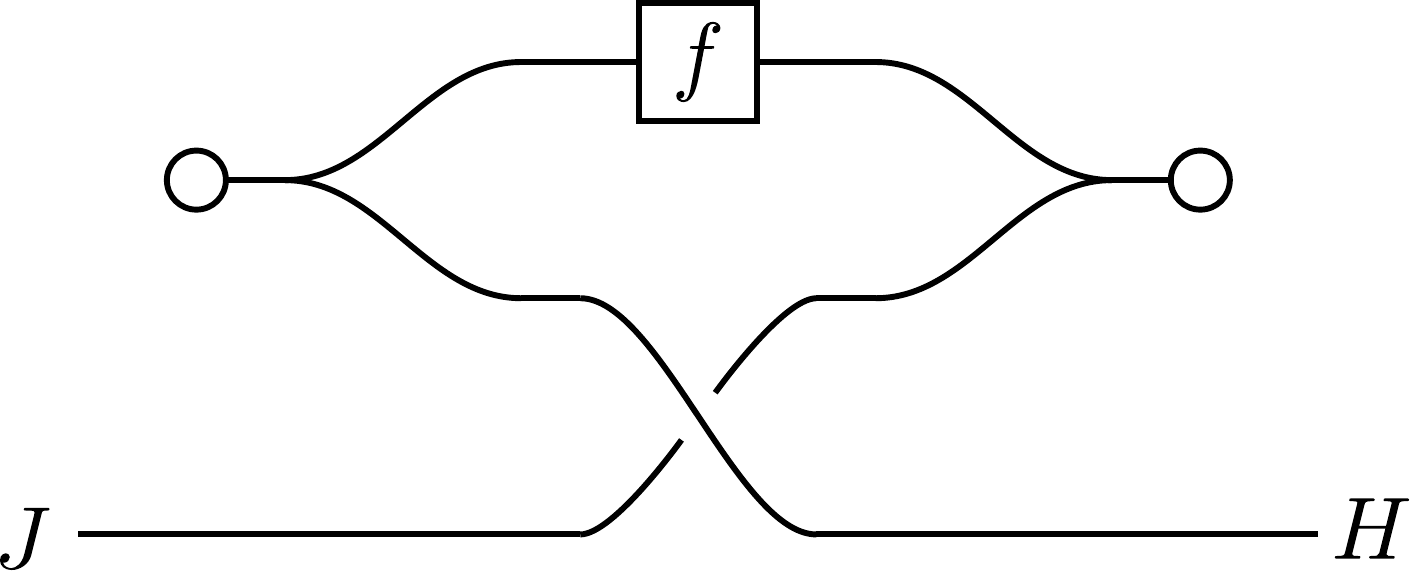}$. Notice that, 
when $f$ is the identity, both the monoidal and comonoidal structure are $t$; which is to say that $(-)^*$ preserves identities.

It is a somewhat lengthy verification to show that all of of the above maps are well-defined and constitute a separable Frobenius monoidal
structure on $f^*$; we consider the Frobenius axioms themselves (Equations~\ref{frobenius-axioms}), leaving the other details to the reader. To save space, we label each of the
morphisms in the diagrams below with the element of $H \tens H \tens H$ which acts on $a \tens b \tens c$, according to the definition
of~$\nabla$ and the tensor products $\tens_H$ and $\tens_J$. 
From the above definition: \[ \bfig

\square(0,0)<1200,500>[f^*a \tens_H f^*(b \tens_J c)`f^*a \tens_H f^*b \tens_H f^*c`f^*(a \tens_J b \tens_J c)`f^*(a \tens_J b) \tens_H f^*c;
\pic{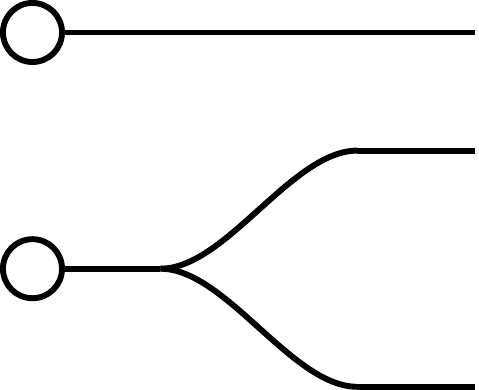}`\pic{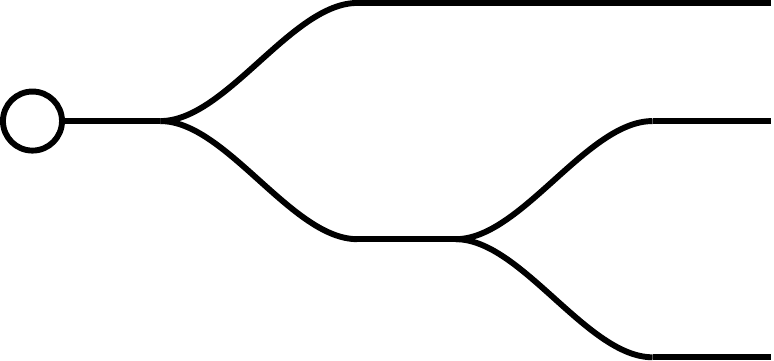}`\pic{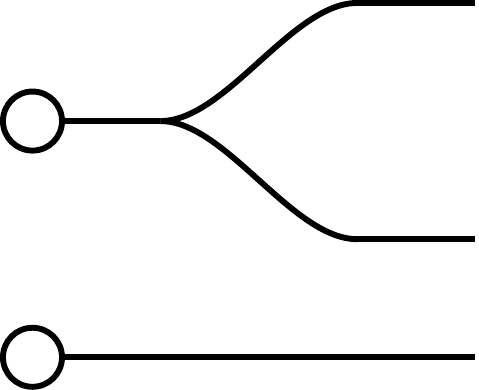}`\pic{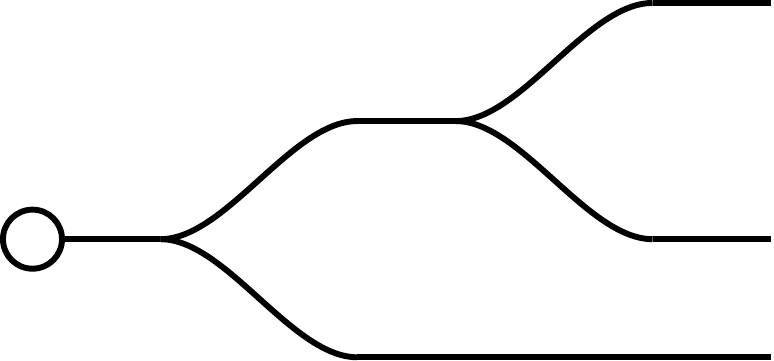}]

\square(0,+1000)<1200,500>[f^*(a \tens_J b) \tens_H f^*c`f^*a \tens_H f^*b \tens_H f^*c`f^*(a \tens_J b \tens_J c)`f^*a \tens_H f^*(b \tens_J c);
\pic{three-split-up.pdf}`\pic{three-skew-up.pdf}`\pic{three-split-down.pdf}`\pic{three-skew-down.pdf}]
\efig \]

\noindent Easy calculations show that the bottom-left composites of the above are: \[ \pic{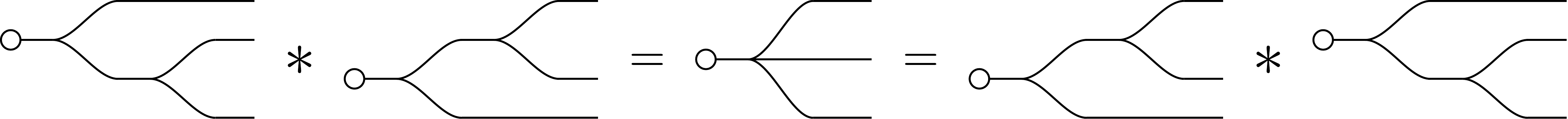} \]
Furthermore, the top-right composites of the above squares are calculated as: \[ \pic{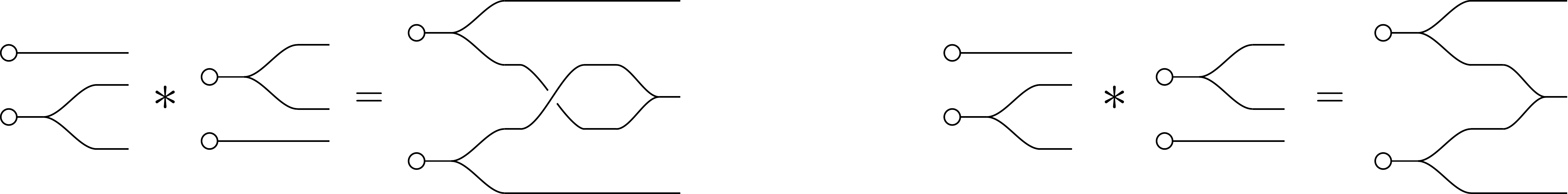} \]
Therefore, we see that these squares commute precisely because of the Weak Unit Axioms (Equations~\ref{weak-unit-axioms}) for $J$.

Further calculations show that $(gf)^* = g^*f^*$ as Frobenius monoidal functors; consequently, we obtain a functor:
\[ \displaystyle \mod \colon \left( \wba \V \right)^\op \too \recon{\V} \]

Since weak morphisms between weak Hopf algebras are simply weak morphisms between their underlying weak bialgebras, and strong monoidal
functors between autonomous categories are simply strong monoidal functors between their underlying monoidal categories, this
$\mod$ restricts to a functor:
\[ \displaystyle \mod \colon \left( \wha \V \right)^\op \too \reconstar{\V} \]

\section{The Tannaka Adjunction}\label{adjunction}

In this section, we will show that the functors defined in the previous two sections form an adjunction, specifically:
\[ \bfig

\node a(-500,0)[\displaystyle \recon{\V}]
\node b(+500,0)[\left( \wba \V \right)^\op]

\arrow|a|/{@{>}@/^1em/}/[a`b;\tan]
\arrow|b|/{@{>}@/^1em/}/[b`a;\mod]
\place(0,0)[\bot]

\efig \]

\noindent Furthermore, there is a restricted adjunction:
\[ \bfig

\node a(-500,0)[\displaystyle\reconstar{\V}]
\node b(+500,0)[\left( \wha \V \right)^\op]

\arrow|a|/{@{>}@/^1em/}/[a`b;\tan]
\arrow|b|/{@{>}@/^1em/}/[b`a;\mod]
\place(0,0)[\bot]

\efig \]

\subsubsection{Units and Counits}

Let $H$ be a weak Hopf algebra in $\V$. We define a unit $\eta \colon H \too \tan U_H$, where $U_H \colon \Hmod \too \V$ is the forgetful functor.
Specifically, we define $\eta$ to correspond to the obvious action $\tilde \alpha \colon H \tens U_H \too U_H$ whose component at an $H$-module
$(A,\alpha)$ is $\alpha$. This is readily checked to be natural in $H$, and a \emph{strong} morphism of weak bialgebras; for instance, the 
following diagram shows that $\eta$ respects the counits:
\[ \bfig

\node a(-1500,+1000)[H]
\node b(+500,+1000)[\tan U_H]
\node c(-1500,+500)[H \tens \boot]
\node d(+500,+500)[\tan U_H \tens \boot]
\node e(-500,-500)[H \tens U_H \boot_H]
\node f(+500,0)[\tan U_H \tens U_H \boot_H]
\node g(-1500,-500)[H \tens H]
\node h(+500,-500)[U_H\boot_H]
\node i(-1500,-1000)[H]
\node j(-500,-1000)[H]
\node k(+500,-1000)[\boot]

\place(-500,+750)[\rulaw]
\place(-1125,-250)[\phi_0]
\place(-1000,-750)[\tilde \alpha]
\place(-500,-1250)[]
\place(+175,-850)[\psi_0]
\place(+175,-350)[\eta]

\arrow|a|[a`b;\eta]
\arrow|m|[a`c;\rulaw^{-1}]
\arrow|m|[c`d;\eta \tens \boot]
\arrow|m|[b`d;\rulaw^{-1}]
\arrow|m|[c`e;H \tens \phi_0]
\arrow|m|[d`f;\tan U_H \tens \phi_0]
\arrow|m|[e`f;\eta \tens U_H \boot_H]
\arrow|m|[f`h;\alpha]
\arrow|m|[e`h;\tilde \alpha]
\arrow|m|[c`g;H \tens \eta]
\arrow/=/[e`g;]
\arrow/=/[h`j;]
\arrow|m|[g`i;\mu]
\arrow|m|[i`j;t]
\arrow|m|[h`k;\psi_0]
\arrow|m|[j`k;\epsilon]
\arrow|r|/{@{>}@/^5em/}/[b`k;\epsilon]
\arrow/{@{=}@/_5em/}/[a`i;]
\arrow|b|/{@{>}@/_5em/}/[i`k;\epsilon]

\efig \]

The irregular central cell commutes since $\tens$ is functorial; the cell marked $\rulaw$ commutes by naturality of $\rulaw$; the left-hand bubble
commutes since $H$ is a unital algebra; the right-hand bubble commutes by definition of $\epsilon$; the cell marked $\phi_0$ commutes by definition
of $\phi_0$; the cell marked $\eta$ commutes by definition of $\eta$; the cell marked $\tilde \alpha$ commutes by definition of $\tilde \alpha$, since
the tensor unit $\boot_H$ in $\Hmod$ is $(H,t\mu,t)$; the lower bubble is an easy calculation; and the cell labelled $\psi_0$ commutes by the definition
of $\psi_0$ given in Section~\ref{def-of-mod}.

Let $F \colon A \too \V$ be a separable Frobenius functor of reconstruction type. We define a (contravariant) counit 
$\epsilon_F \colon A \too (\tan F)\text-\mod$ by taking
every object $x$ of $A$ to $Fx$ equipped with the canonical $\tan F$ action. Specifically:
\[ \epsilon x = \left( Fx, \tan F \tens Fx \oto{\alpha} Fx, Fx \right) \]
Given this, we compute: \begin{align*}
        \epsilon(x \tens y) 
                &= \left( F(x \tens y), \tan F \tens F(x \tens y) \oto{\alpha} F(x \tens y), F(x \tens y) \right) \\
                &= \left( F(x \tens y), \pic{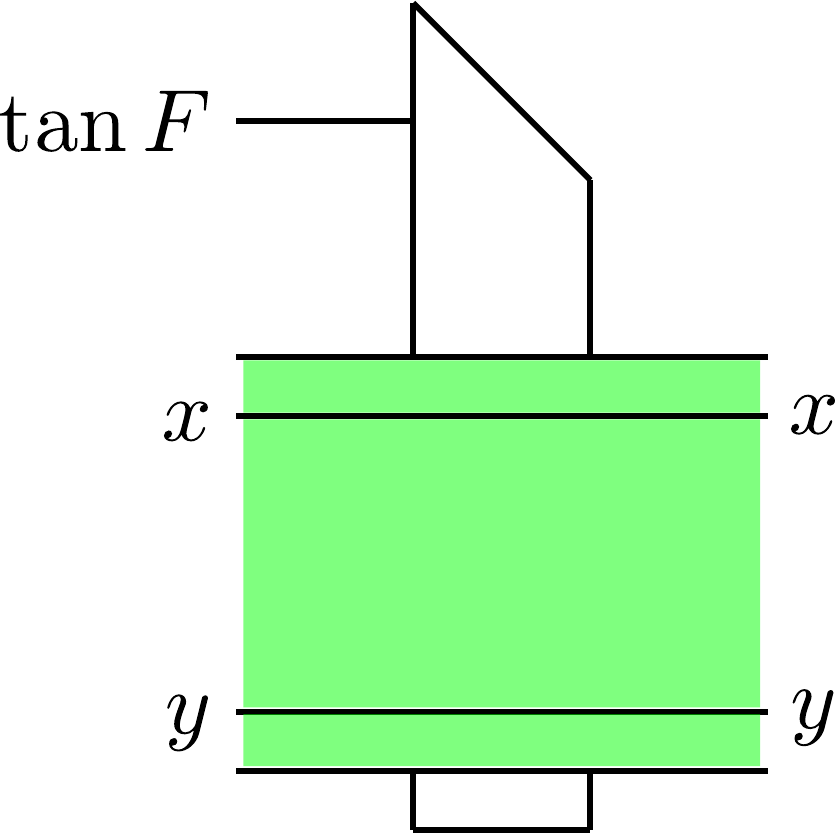}, \pic{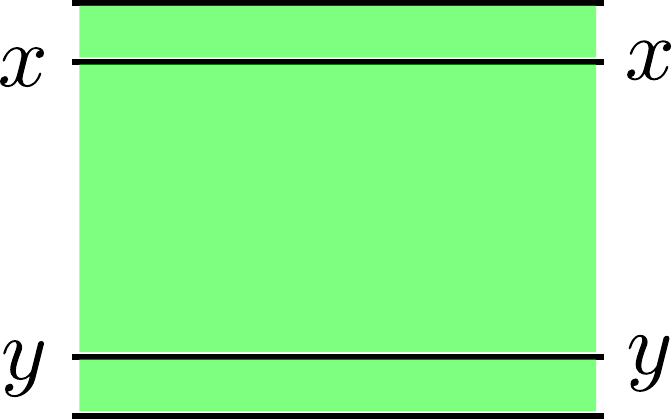} \right) \\
        \epsilon x \tens \epsilon y
                &= \left( Fx, \tan F \tens Fx \oto{\alpha} Fx, Fx \right) \tens \left( Fy, \tan F \tens Fy \oto{\alpha} Fy, Fy \right) \\
                &= \left( Fx \tens Fy, \pic{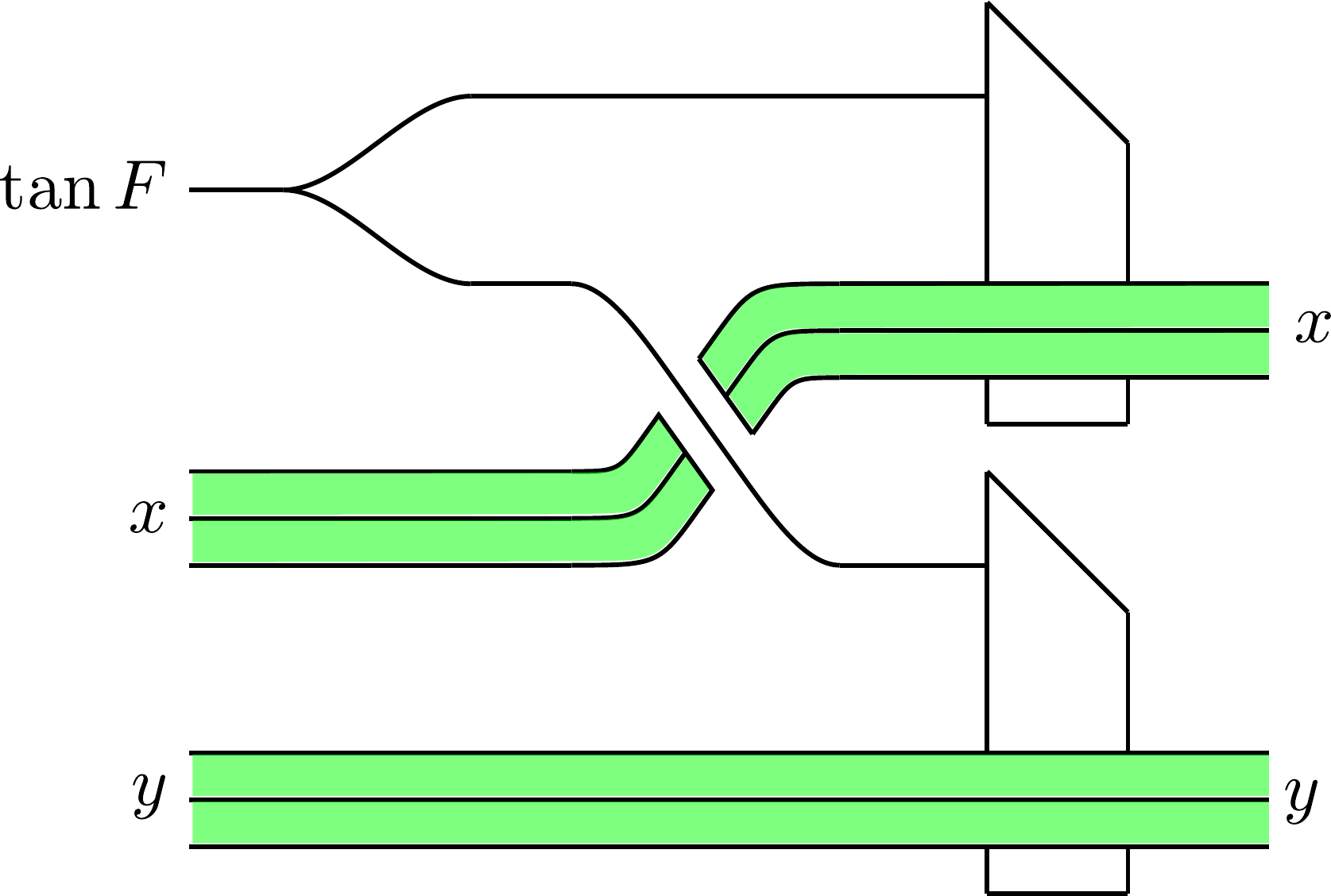}, \pic{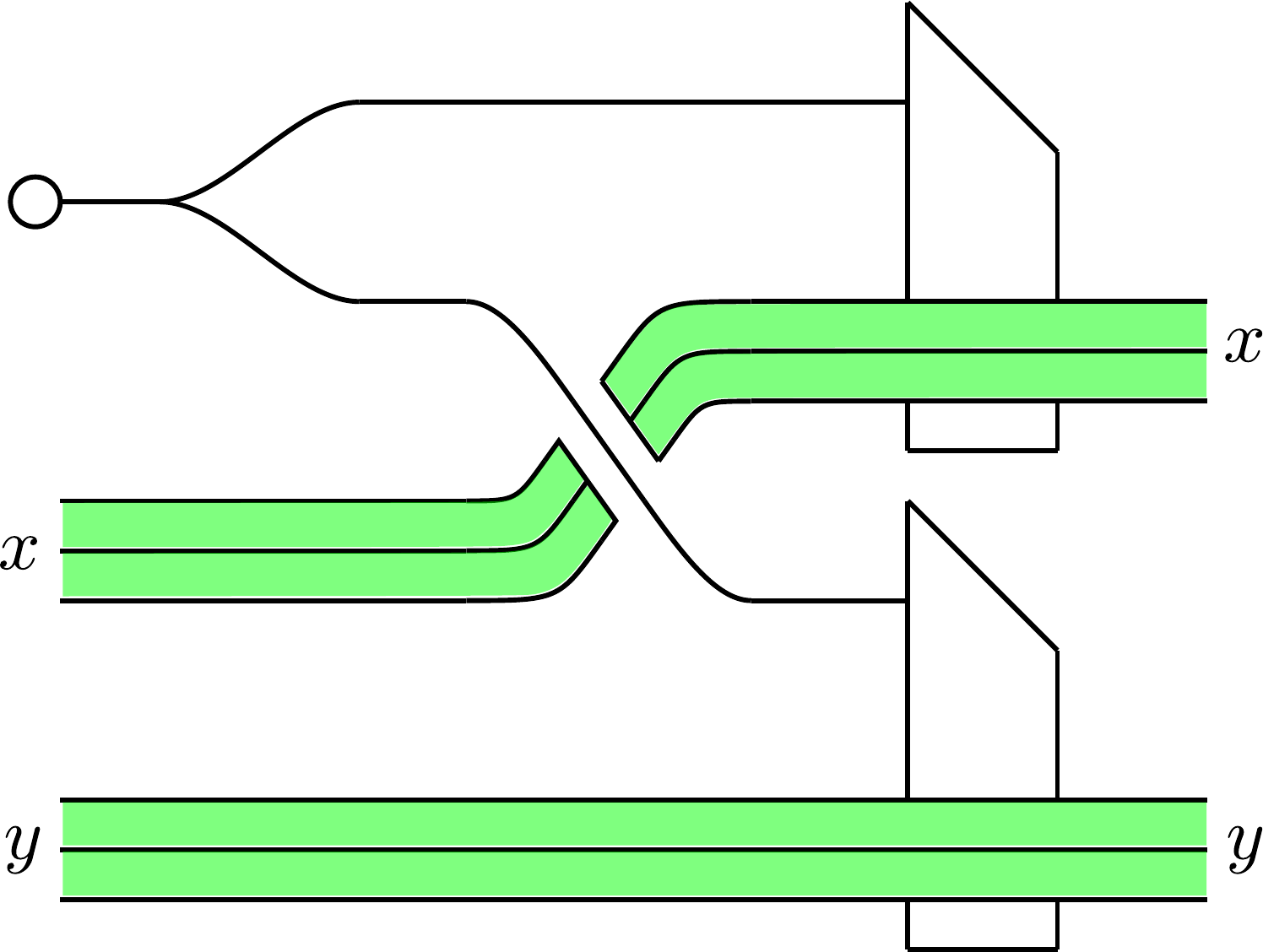} \right) \\
                &= \left( Fx \tens Fy, \pic{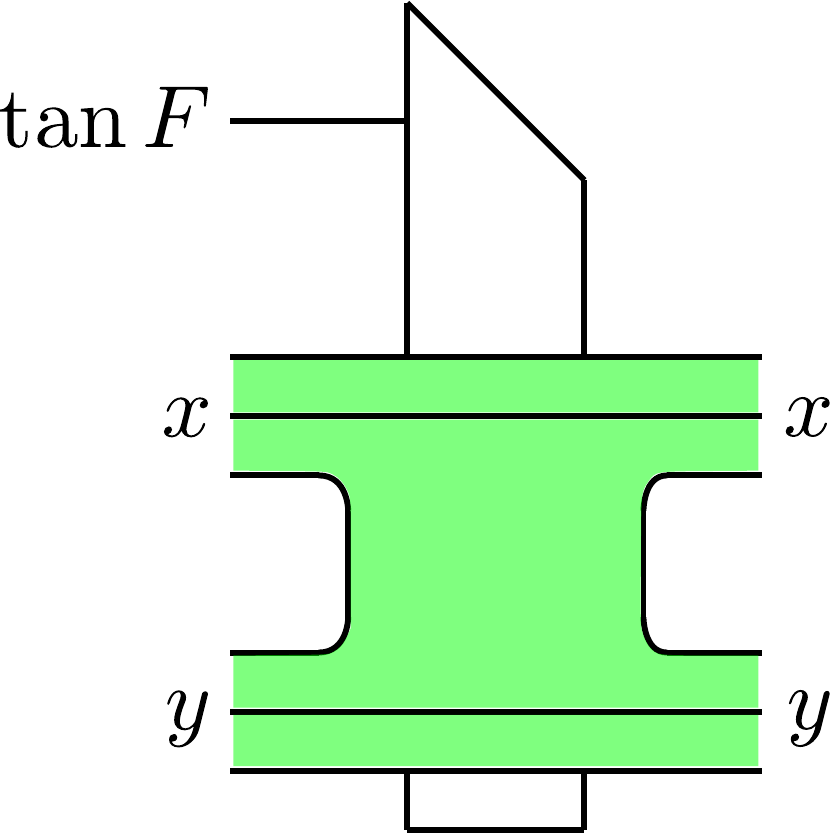}, \pic{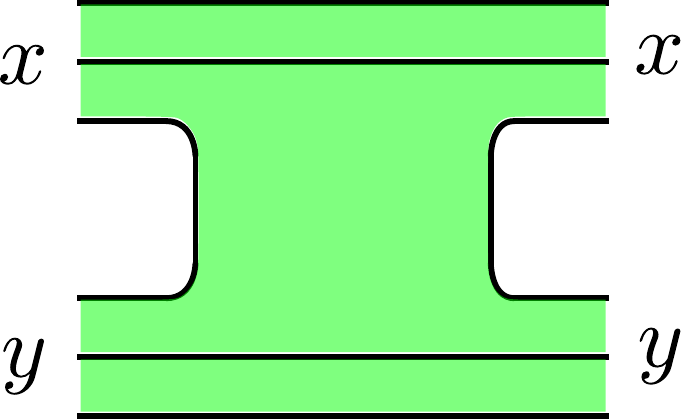} \right)
\end{align*} We therefore take the binary monoidal and comonoidal structures on $\epsilon$ to be those of $F$, this is well-defined as a map
of actions and a map of idempotents precisely because $F$ is separable. 

As for the nullary monoidal and comonoidal structures on $\epsilon$, we compute: \begin{align*}
        \epsilon \boot_A &= \left( F\boot, \tan F \tens F\boot \oto{\alpha} F\boot, F\boot \right) \\
        \boot_{(\tan F)\text-\mod} &= \left( \tan F, \tan F \tens \tan F \oto{\mu} \tan F \oto{t} \tan F, t_{\tan H} \right)
\end{align*} We therefore define the nullary monoidal structure $\phi_0 \colon \boot \too \epsilon \boot$ to be:
\[ \tan F \oto{\rulaw^{-1}} \tan F \tens \boot \oto{\tan F \tens \phi_0} \tan F \tens F\boot \oto{\alpha} F\boot \]
and we define the nullary comonoidal structure $\psi_0 \colon \epsilon \boot \too \boot$
to be the map $F\boot \too \tan F$ corresponding to the action of $F\boot$
on $F$ defined by: \[ F\boot \tens Fx \oto{\phi} F(\boot \tens x) \oto{F\lulaw} Fx \] Graphically, this defines $\psi_0$ as the unique map such that:
\[ \includegraphics[scale=0.2]{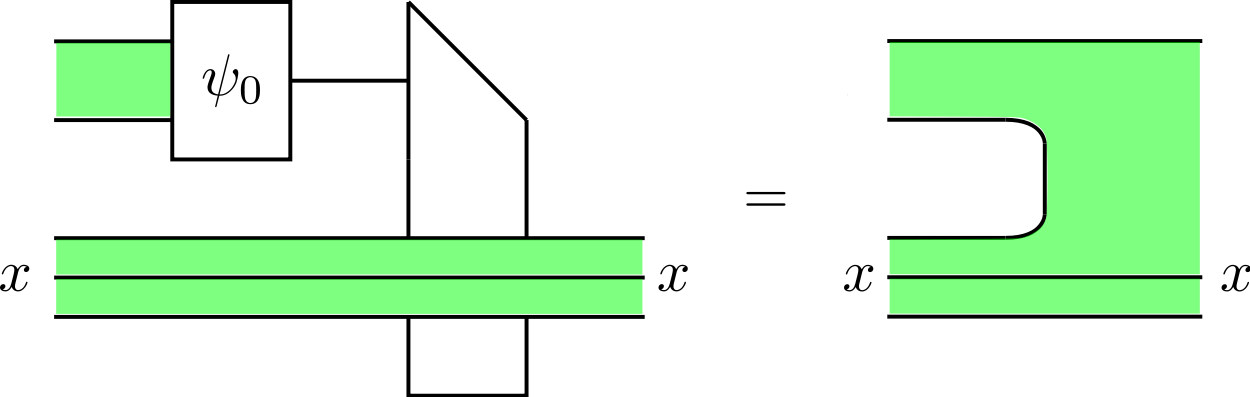} \]		
One checks at some length that $\phi_0$ and $\psi_0$ so defined are maps of idempotents, are maps of actions,
are mutually inverse, form coherent monoidal and comonoidal structures on $\epsilon$, and render $\epsilon U_{\tan F} = F$ as Frobenius functors.
To see that they are mutually inverse, for instance, one first computes: \[ \includegraphics[scale=0.2]{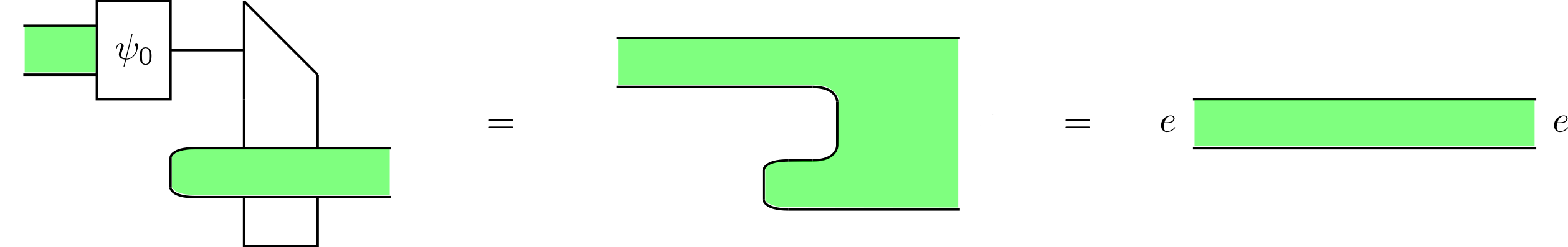} \]
and furthermore, that \[ \includegraphics[scale=0.2]{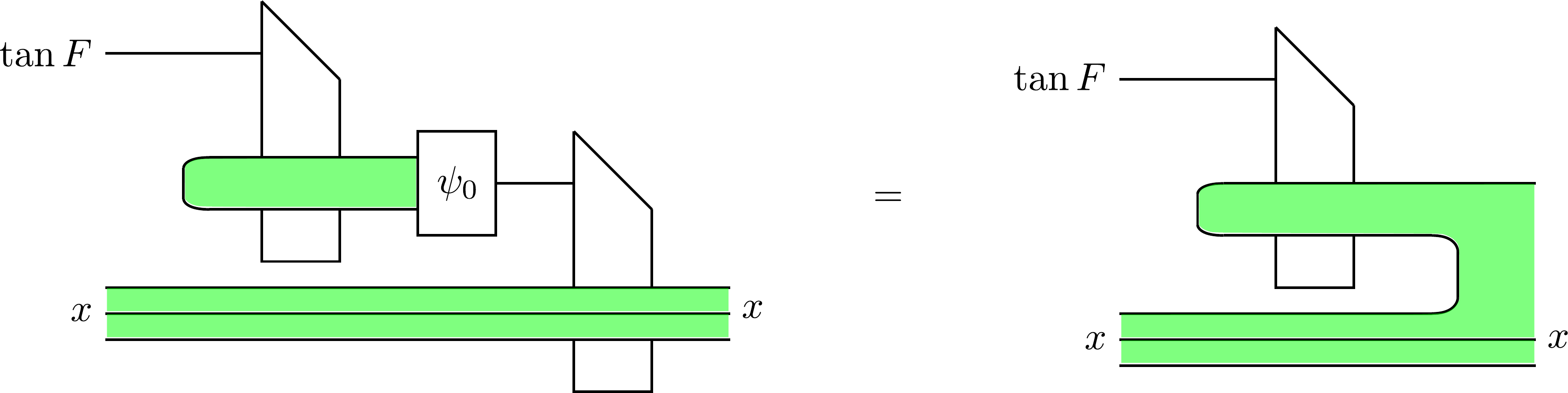} \] which we recognize from the right-hand-side of Figure~\ref{ants}
as the discharged form of the idempotent $t$ on $\tan F$, as required. Furthermore, $\epsilon$ commutes with $F$ and $U_{\tan F}$ as a Frobenius
functor since $F$ is separable. Note in particular that, although $F$ is not strong, $\epsilon$ \emph{is} strong, since the identity on 
$\epsilon x \tens \epsilon y$ is the idempotent given.

Hence, this $\epsilon$ defines a morphism $F \too U_{\tan F}$ in $\recon{\V}$ and is, in fact, \emph{strong} monoidal.
Furthermore, it is easily seen to be natural in $F$.

We must verify the triangle identities for the adjunction $\tan \ladj \mod$. On the one hand, let a weak bialgebra $H$ be given, we must show
that \[ \mod H \oto{\epsilon_{U_H}} \mod\left( \tan U_H \right) \oto{\mod\left( \eta_H\right)} \mod H \] is the identity. Hence, let 
$(a, \gamma \colon H \tens a \too a)$ in $\mod H$ be given. We compute that 
\begin{align*}
\mod\left(\eta_H\right)\epsilon_{U_H}\left(a,H \tens a \oto{\gamma} a \right)
&= \mod\left(\eta_H\right)\left(a,\tan U_H \tens U_H(a,\gamma) \oto{\alpha} U_H(a,\gamma) \right)  \\
&= \left(a,H \tens U_H(a,\gamma) \oto{\eta_N \tens U_H(a,\gamma)} \tan U_H \tens U_H(a,\gamma) \oto{\alpha} U_H(a,\gamma) \right) \\
&= \left(a,H \tens U_H(a,\gamma) \oto{\alpha} U_H(a,\gamma) \right) \\
&= \left(a,H \tens a \oto{\gamma} a \right)
\end{align*} Where the equalities hold: by definition of $\epsilon$, by definition of $\mod$, and by definition of $\eta$.
On the other hand, let $F \colon A \too \V$ be a separable Frobenius monoidal functor of reconstruction type; we must show that 
\[ \tan F \oto{\eta_{\tan F}} \tan U_{\tan F}
\oto{\tan \epsilon_F} \tan F \] is the identity. For this, consider the following diagram:
\[ \bfig

\hscalefactor{1.24}

\node g(-500,0)[\tan F \tens U_{\tan F}\epsilon_F]
\node h(+500,0)[\tan U_{\tan F} \tens U_{\tan F}\epsilon_F]
\node d(0,-500)[U_{\tan F}\epsilon_F]
\node a(-1000,+500)[\tan F \tens F]
\node b(+1000,+500)[\tan U_{\tan F} \tens F]
\node c(-1000,-500)[F]
\node e(+2000,-500)[F]
\node f(+2000,+500)[\tan F \tens F]

\arrow[a`b;\eta_{\tan F} \tens F]
\arrow[b`f;\tan \epsilon_F \tens F]
\arrow|r|[f`e;\alpha]
\arrow[a`c;\alpha]
\arrow/{@{>}@/^2em/}/[g`h;\eta_{\tan F} \tens U_{\tan F}\epsilon_F]
\arrow|m|[g`d;\alpha]
\arrow|m|[h`d;\alpha]
\arrow/=/[a`g;]
\arrow/=/[b`h;]
\arrow/=/[c`d;]
\arrow/=/[d`e;]

\efig \]

The upper cell commutes since $U_{\tan F}\epsilon_F = F$; the left-hand cell commutes by definition of $\epsilon$; the right-hand cell
commutes by definition of $\tan$; and the central cell commutes by definition of $\eta$. Hence, we have shown that:
\[ \alpha\left(\tan \epsilon_F\eta_{\tan F} \tens F\right) = \alpha \] which, by the universal property of $\alpha$, gives
\[ \tan \epsilon_F \eta_{\tan F} = \tan F \] as desired. Hence, we have that $\tan \ladj \mod$, as desired.

Furthermore, we have noted that the components of $\eta$ and $\epsilon$ are actually strong, and that the functors $\tan$ and $\mod$
are well-defined when simultaneously restricted to strong morphisms of weak bialgebras and strong monoidal functors between separable Frobenius
functors. Therefore, this restricted ``$\tan$'' is left adjoint to this restricted ``$\mod$''. This restricted adjunction is well-known; see, 
for instance, Section~16 of Street~\cite{Street07}.

So, we have proved:
\begin{proposition} There is a linked pair of adjunctions:
\[ \bfig

\hscalefactor{1.3}
\vscalefactor{1.6}

\node a(-500,0)[\recon{\V}]
\node b(+500,0)[\left( \wba \V \right)^\op]

\arrow|a|/{@{>}@/^2em/}/[a`b;\tan]
\arrow|b|/{@{>}@/^2em/}/[b`a;\mod]
\place(0,0)[\bot]

\node aa(-500,-500)[\reconstar{\V}]
\node bb(+500,-500)[\left( \wha \V \right)^\op]

\arrow|a|/{@{>}@/^2em/}/[aa`bb;\tan]
\arrow|b|/{@{>}@/^2em/}/[bb`aa;\mod]
\place(0,-500)[\bot]

\arrow[aa`a;]
\arrow[bb`b;]

\efig \]

Where the diagram commutes serially. Furthermore, we can restrict to non-weak bialgebras and strong monoidal functors, and the above adjunctions restrict 
to the well-known adjunctions:
\[ \bfig

\hscalefactor{1.3}
\vscalefactor{1.6}

\node a(-500,0)[\displaystyle\strongrecon{\V}]
\node b(+500,0)[\left( \ba \V \right)^\op]

\arrow|a|/{@{>}@/^2em/}/[a`b;\tan]
\arrow|b|/{@{>}@/^2em/}/[b`a;\mod]
\place(0,0)[\bot]

\node aa(-500,-500)[\displaystyle\strongreconstar{\V}]
\node bb(+500,-500)[\left( \ha \V \right)^\op]

\arrow|a|/{@{>}@/^2em/}/[aa`bb;\tan]
\arrow|b|/{@{>}@/^2em/}/[bb`aa;\mod]
\place(0,-500)[\bot]

\arrow[aa`a;]
\arrow[bb`b;]

\efig \]

There is an evident quadruple of inclusions from the four categories in this last diagram to the four categories in the first diagram, making in
all a commutative square of adjunctions.

\end{proposition}

\subsection{The Internal Separable Frobenius Algebra in $\tan F$}

We have seen above that the nullary monoidal and comonoidal structures of the functor $\epsilon$---namely, 
$\phi_0 \colon F\boot \too \tan F$ and $\psi_0 \colon \tan F \too F\boot$---have the property that $\phi_0\psi_0 = t_{\tan F}$ and $\psi_0\phi_0 = F\boot$; 
that is, we have witnessed $F\boot$ as a splitting of $t_{\tan F}$. 
Recall from Section~\ref{internal-sepfrob} that any such splitting $H \oto{\alpha} h \oto{\beta} H$ inherits a separable Frobenius structure 
from the bialgebra structure of $H$; specifically:
\begin{align*} 
        \mu' &= h \tens h \oto{\beta \tens \beta} H \tens H \oto{\mu} H \oto{\alpha} h \\ 
        \delta' &= h \oto{\beta} H \oto{\delta} H \tens H \oto{\alpha \tens \alpha} h \tens h \\ 
        \epsilon' &= h \oto{\beta} H \oto{\epsilon} \boot \\ 
        \eta' &= \boot \oto{\eta} H \oto{\alpha} h 
\end{align*} 
 
We can calculate the explicit forms of this structure in the case where $(\alpha,\beta) = (\psi_0,\phi_0)$, to find that these four maps are
given by:
\fullwidth{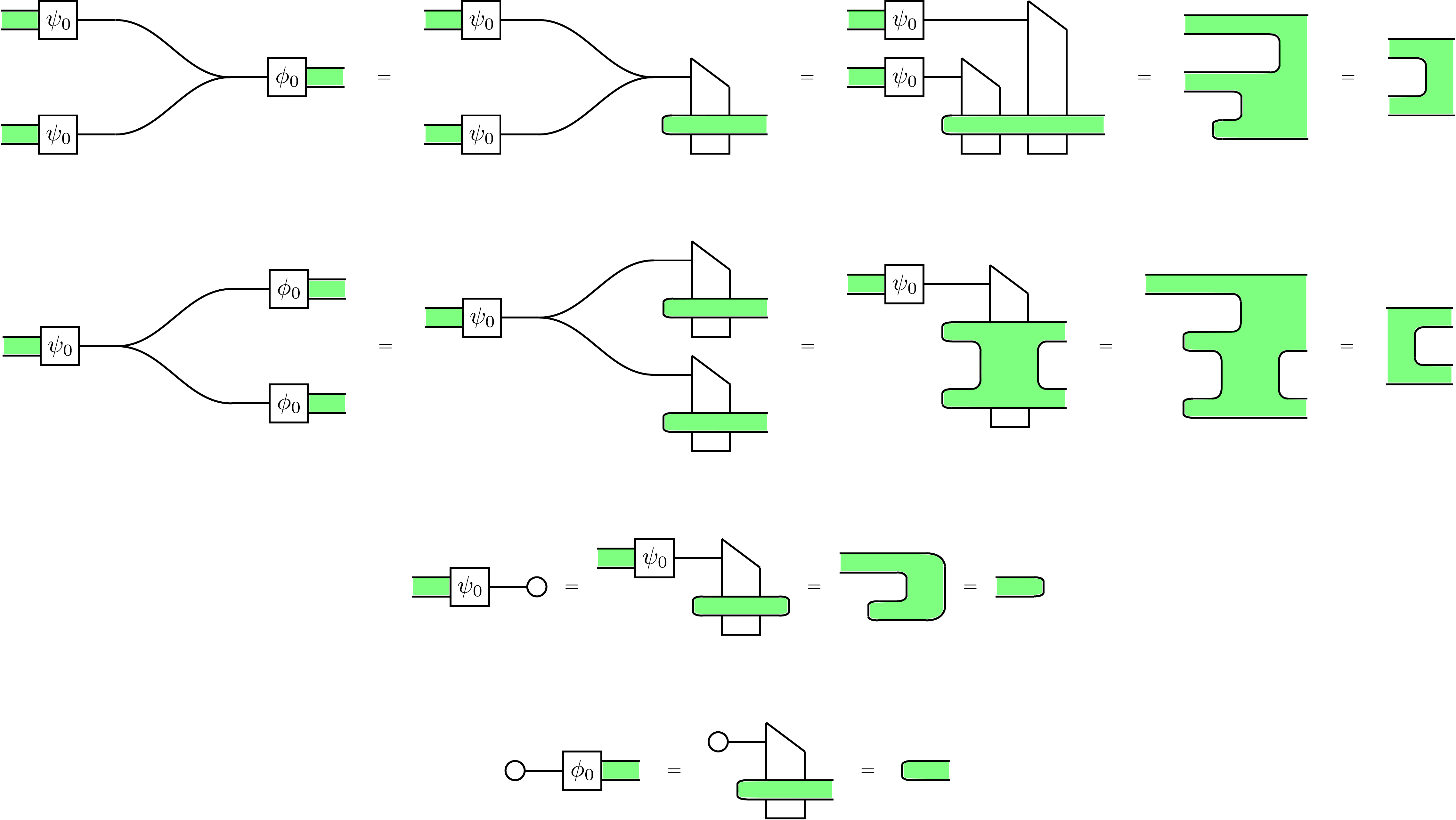}
Trivially, $\boot$ bears a Frobenius algebra structure in $A$, hence, so too does its image $F\boot$ under the separable Frobenius functor $F$.
The above calculation proves a conjecture of Dimitri Chikhladze that these two Frobenius algebra structures on $F\boot$ coincide.

\section{Change of Base for the Tannaka Adjunction}\label{change-of-base}

We have seen that, for fixed $\V$, there is an adjunction: 
\[ \bfig \adjunction{\recon{\V}}{\left(\wba \V\right)^{\op}}{\tan_\V}{\mod_\V} \efig \]

Now let us consider what happens when we vary the base category $\V$. We must define a suitable category through which $\V$ is to vary. 
\begin{definition}
	Denote by $\KK$ the 2-category whose objects are braided monoidally complete categories, whose arrows are separable Frobenius
	monoidal functors which are braided as monoidal functors and braided as comonoidal functors, 
	and whose 2-cells are monoidal and comonoidal natural transformations.
\end{definition}

\begin{proposition} There is a 2-functor $\recon{-} \colon \KK \too \Cat$ whose value at a braided monoidally complete category $\V$ is
$\recon{\V}$ as defined above.
\end{proposition}
\begin{proof} For each object $\V$ in $\KK$, we define $\recon{\V}$ as above, namely, to be the comma category of separable 
Frobenius monoidal functors of reconstruction type into $\V$ with Frobenius monoidal functors between them.
If $\Phi \colon \V \too \W$ is an arrow in $\KK$, then composition with $\Phi$ defines a functor 
$\recon{\Phi} \colon \recon{\V} \too \recon{\W}$, since the composition of a separable Frobenius monoidal functor of 
reconstruction type followed by an arbitrary separable Frobenius monoidal functor is again a separable Frobenius functor of reconstruction type. 
Similarly, given $\Phi,\Psi \colon \V \too \W$ and 
$\alpha \colon \Phi \Too \Psi$ in $\KK$, then $\recon{\alpha} \colon \recon{\Phi} \too \recon{\Psi}$ defines a natural
transformation whose value at an object $F~\colon~A~\too~\V$ of $\recon{\V}$ is $\alpha$ whiskered by $F$. Verification of the 2-functor
axioms is routine.
\end{proof}

We will require the following:
\begin{lemma}[The Bow Lemma] If $F$ is a Frobenius functor which is braided as a monoidal functor or braided as a comonoidal functor, then the following
equation holds: \[ \includegraphics[scale=0.1]{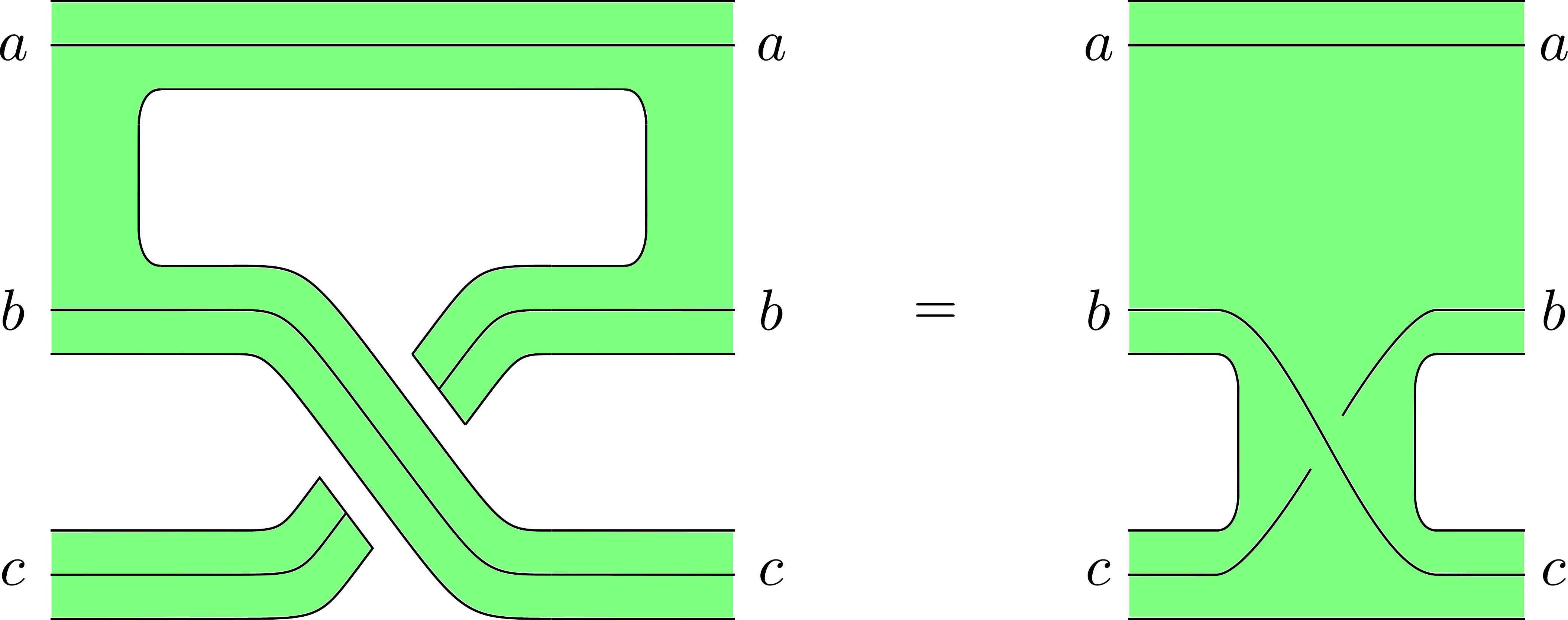} \]
\end{lemma}
\begin{proof} We present the case where $F$ is known to be braided as a comonoidal functor; a dual proof can be obtained by taking horizontal flips
of every step. Consider the following calculation:
\fullwidth{bow-lemma.pdf}
The first equality is simply the insertion of an isomorphism (in the codomain) and its inverse. 
The second equality uses the braidedness of the functor on the left and the naturality of the braid on the right. The third equality uses a 
Frobenius axiom followed by another instance of the braidedness of the functor. Finally, the last equality simply cancels out an isomorphism 
(in the domain) with its inverse. \end{proof}

\begin{proposition} There is a 2-functor $\wba - \colon \KK \too \Cat$ whose value at a braided monoidally complete category $\V$ is
$\wba \V$ as defined above.
\end{proposition}
\begin{proof} Let $\Phi \colon \V \too \W$ be an arrow in $\KK$. Define $\wba \Phi \colon \wba \V \too \wba \W$ as follows: 
Let $(B,\delta,\mu,\eta,\epsilon)$ be a weak bialgebra in $\wba \V$. Define $(\wba \Phi)B$ to be $\Phi B$ equipped with suitably conjugated versions 
of the structural maps of $\V$, this is again a weak bialgebra. 
To see that $(\wba \Phi)B$ satisfies the weak counit axioms, consider the following calculation:
\fullwidth{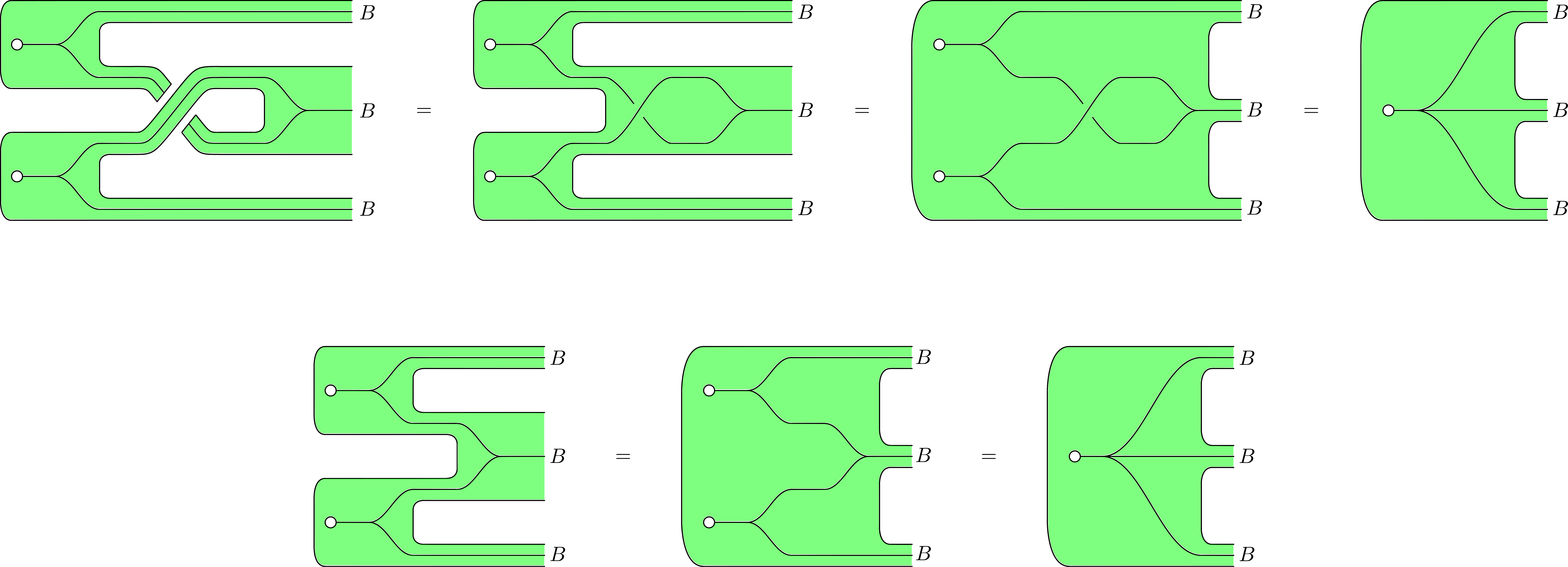}
The first equality in the first line uses the fact that $\Phi$ is braided as a monoidal functor; after that, the equalities in both lines follow 
from the Frobenius axioms, followed by the weak bialgebra counit axioms in the domain. The weak unit axioms are satisfied by the horizontally
flipped versions of the same calculations; this will use the fact that $\Phi$ is braided as a comonoidal functor.

Finally, we must verify the bialgebra axiom. To this end, consider the following: \fullwidth{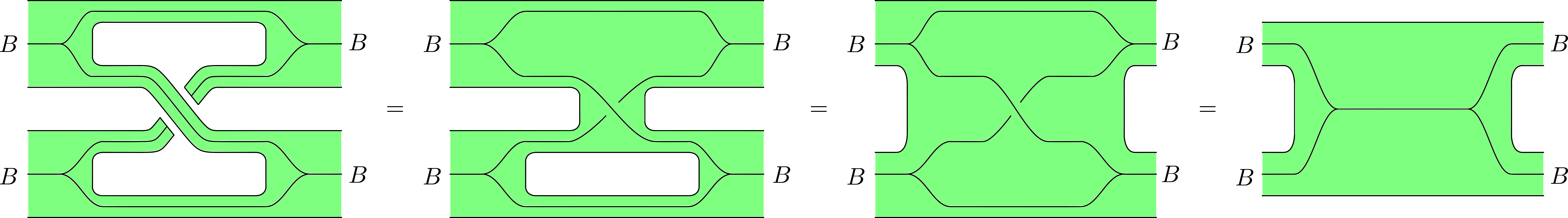}
The first equality holds by the Bow Lemma, the second by both Frobenius axioms and separability of $\Phi$, and the last by the bialgebra 
axiom in $\V$. Thus, $(\wba \Phi)B$ is a weak bialgebra as defined.

Let arrows $\Phi,\Psi \colon \V \too \W$ and 2-cell $\alpha \colon \Phi \Too \Psi$ in $\KK$ be given. Then define $\wba \alpha \colon \wba \Phi
\Too \wba \Psi$ to be $\alpha B \colon \Phi B \too \Psi B$. Since $\alpha$ is monoidal and comonoidal, this defines a \emph{strict} morphism of
weak bialgebras, although we will not need this fact.

Verifying that $\wba-$ so defined satisfies the 2-functor axioms is straightforward.
\end{proof}

With these definitions in hand, we can discuss the naturality of the Tannaka construction:
\begin{proposition} There is a lax natural transformation $\tan-$ from 
$\recon{-}$ to $\left(\wba-\right)^\op$, whose value at a braided monoidally complete $\V$ is the functor 
$\tan_\V~\colon~\recon{\V}~\too~\left(\wba \V\right)^\op$ discussed above.
\end{proposition}
\begin{proof}

As promised, we define the 1-cells of the lax natural transformation $\tan -$ to be $\tan_\V$ for each object $\V$ of $\KK$. Given an arrow
$\Phi \colon \V \too \W$ in $\KK$, define the 2-cells of the lax natural transformation $\tan-$ to be $\rho \Phi$:

\[ \bfig \square(-500,-500)<1000,1000>[
\displaystyle \recon{\V}`\left( \wba \V\right)^\op`\displaystyle \recon{\W}`\left(\wba \W\right)^\op;
\tan_\V`\displaystyle \recon{\Phi}`\left(\wba \Phi\right)^\op`\tan_\W] \place(0,0)[\twoar(1,1) \rho \Phi] \efig \]
where $\rho \Phi$ is defined at an object $F \in \recon{\V}$ as the morphism \[ \rho\Phi F \colon \Phi \tan F \too \tan \Phi F \] in 
$\left(\wba \W \right)^\op$ corresponding to \[ \Phi \tan F \tens \Phi F \oto{\monoidal} \Phi\left( \tan F \tens F \right) \oto{\Phi\alpha} \Phi F \]

Verifying that this is natural in $F$ is a routine unravelling of the definitions of $\rho$ and $\tan-$ on arrows.

We must show that $\rho \Phi$ so defined is a weak morphism of weak bialgebras. In fact, it is a \emph{strong} morphism of weak bialgebras.

First, to see that $\rho \Phi$ preserves the unit, consider: \fullwidth{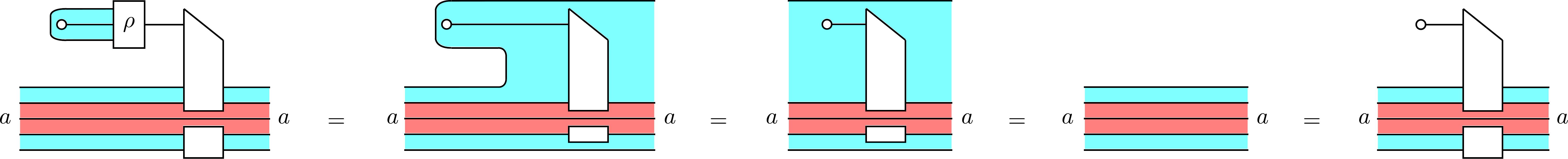}
The equalities hold by: definition of $\rho$; naturality and monoidality of the monoidal structure of $\Phi$; the definition of the unit of 
$\tan F$; and the definition of the unit of $\tan \Phi F$.

Second, to see that $\rho \Phi$ preserves the counit, consider: \fullwidth{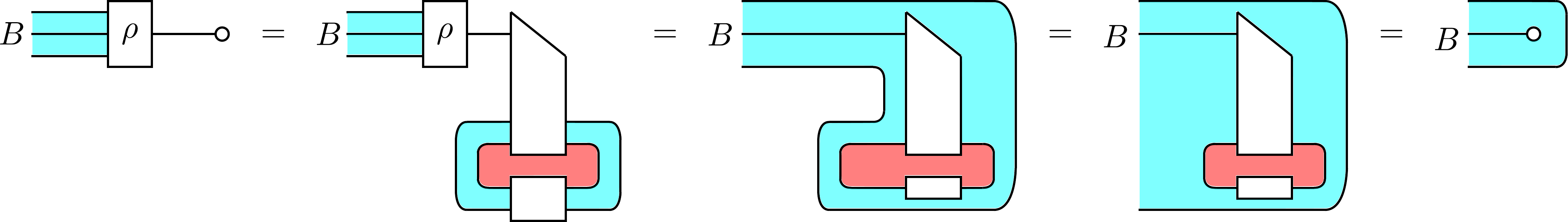}
The equalities hold by: definition of the counit of $\tan \Phi F$; definition of $\rho$; naturality and monoidality of the monoidal structure of 
$\Phi$; and the definition of the counit of $\tan F$.

Third, to see that $\rho \Phi$ preserves the multiplication, consider: \fullwidth{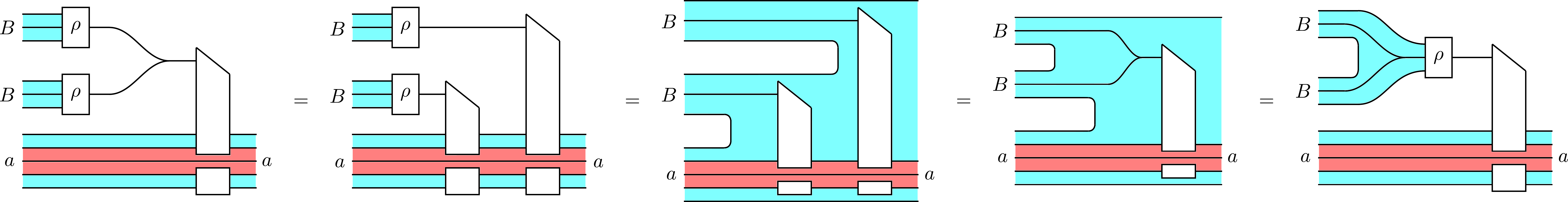}
The equalities hold by: definition of the multiplication of $\tan \Phi F$; definition of $\rho$; naturality and associativity of the monoidal 
structure of $\Phi$; and the definition of $\rho$ once more.

Fourthly and finally, to see that $\rho \Phi$ preserves the comultiplication, see Figure~\ref{rho-comultiplication}
\fullwidthfigure{rho/coherence-with-comultiplication.pdf}{Preservation of comultiplication by $\rho \Phi$.
Counterclockwise from top-left, the equalities hold by: definition of $\rho$; the bow lemma for $\Phi$; Frobenius and associativity axioms
for $\Phi$; the definition of the comultiplication of $\tan F$; the definition of $\rho$ again; and, finally, the definition of the comultiplication
of $\tan \Phi F$.}{rho-comultiplication}

Verifying the lax natural transformation axioms is routine.\end{proof}

Since, for each $\V$, the functor $\tan_\V$ has a right adjoint,
an application of ``Australian mates'' to this lax natural transformation $\rho$ yields an oplax natural transformation
\[ \bfig \square(-500,-500)<1000,1000>[
\left( \wba \V\right)^\op`\displaystyle\recon{\V}`\left(\wba \W\right)^\op`\displaystyle\recon{\W};
\mod_\V`\left(\wba \Phi\right)^\op`\displaystyle\recon{\Phi}`\mod_\W] \place(0,0)[\twoar(-1,-1) \gamma \Phi] \efig \]

Given a weak bialgebra $B$ in $\V$, the behaviour of $\gamma \colon \mod_\V B \too \mod_\W \Phi B$ can be calculated as
\[ \gamma\left( a, B \tens a \oto{\beta} a, \nabla_a \colon a \too a  \right) = 
\left( \Phi a, \Phi B \tens \Phi a \oto{\monoidal} \Phi(B \tens a) \oto{\Phi \beta} \Phi a, \Phi \nabla_a \right) \]

\bibliographystyle{alpha}
\bibliography{weak-recon}

\def\cprime{$'$}
\begin{thebibliography}{CMZ97}

\bibitem[Bic01]{Bichon01}
Julien Bichon.
\newblock Cosovereign {H}opf algebras.
\newblock {\em J. Pure Appl. Algebra}, 157(2-3):121--133, 2001.

\bibitem[BNS99]{BohmNillSzlachanyi99}
Gabriella B{\"o}hm, Florian Nill, and Korn{\'e}l Szlach{\'a}nyi.
\newblock Weak {H}opf algebras. {I}. {I}ntegral theory and {$C^*$}-structure.
\newblock {\em J. Algebra}, 221(2):385--438, 1999.

\bibitem[BS00]{BohmSzlachanyi00}
Gabriella B{\"o}hm and Korn{\'e}l Szlach{\'a}nyi.
\newblock Weak {H}opf algebras. {II}. {R}epresentation theory, dimensions, and
  the {M}arkov trace.
\newblock {\em J. Algebra}, 233(1):156--212, 2000.

\bibitem[CMZ97]{CaenepeelMilitaruZhu97}
S.~Caenepeel, G.~Militaru, and S.~Zhu.
\newblock Doi-{H}opf modules, {Y}etter-{D}rinfel\cprime d modules and
  {F}robenius type properties.
\newblock {\em Trans. Amer. Math. Soc.}, 349(11):4311--4342, 1997.

\bibitem[CS99]{CockettSeely99}
J.~R.~B. Cockett and R.~A.~G. Seely.
\newblock Linearly distributive functors.
\newblock {\em J. Pure Appl. Algebra}, 143(1-3):155--203, 1999.
\newblock Special volume on the occasion of the 60th birthday of Professor
  Michael Barr (Montreal, QC, 1997).

\bibitem[Day96]{Day96}
Brian~J. Day.
\newblock Enriched {T}annaka reconstruction.
\newblock {\em J. Pure Appl. Algebra}, 108(1):17--22, 1996.

\bibitem[DP08]{DayPastro08}
Brian Day and Craig Pastro.
\newblock Note on {F}robenius monoidal functors.
\newblock {\em New York J. Math.}, 14:733--742, 2008.

\bibitem[Dri87]{Drinfeld86}
V.~G. Drinfel{\cprime}d.
\newblock Quantum groups.
\newblock In {\em Proceedings of the {I}nternational {C}ongress of
  {M}athematicians, {V}ol. 1, 2 ({B}erkeley, {C}alif., 1986)}, pages 798--820,
  Providence, RI, 1987. Amer. Math. Soc.

\bibitem[Dri89]{Drinfeld89}
V.~G. Drinfel{\cprime}d.
\newblock Quasi-{H}opf algebras.
\newblock {\em Algebra i Analiz}, 1(6):114--148, 1989.

\bibitem[Egg08]{Egger08}
J.~M. Egger.
\newblock Star-autonomous functor categories.
\newblock {\em Theory Appl. Categ.}, 20:No. 11, 307--333, 2008.

\bibitem[EM10]{EggerMcCurdy10}
J.~M. Egger and Micah~Blake McCurdy.
\newblock On cyclic star-autonomous categories.
\newblock submitted, 2010.

\bibitem[HO97]{Haring-Oldenburg97}
Reinhard H{\"a}ring-Oldenburg.
\newblock Reconstruction of weak quasi {H}opf algebras.
\newblock {\em J. Algebra}, 194(1):14--35, 1997.

\bibitem[JS91]{JoyalStreet90}
Andr{\'e} Joyal and Ross Street.
\newblock An introduction to {T}annaka duality and quantum groups.
\newblock In {\em Category theory ({C}omo, 1990)}, volume 1488 of {\em Lecture
  Notes in Math.}, pages 413--492. Springer, Berlin, 1991.

\bibitem[Maj92]{Majid92}
Shahn Majid.
\newblock Tannaka-{K}re\u\i n theorem for quasi-{H}opf algebras and other
  results.
\newblock In {\em Deformation theory and quantum groups with applications to
  mathematical physics ({A}mherst, {MA}, 1990)}, volume 134 of {\em Contemp.
  Math.}, pages 219--232. Amer. Math. Soc., Providence, RI, 1992.

\bibitem[Maj93]{Majid93}
Shahn Majid.
\newblock Braided groups.
\newblock {\em J. Pure Appl. Algebra}, 86(2):187--221, 1993.

\bibitem[McC]{McCurdy11Tannaka2}
Micah~Blake McCurdy.
\newblock Reconstruction of structured weak bialgebras and weak hopf algebras.
\newblock In preparation.

\bibitem[McC00]{McCrudden00}
Paddy McCrudden.
\newblock Categories of representations of coalgebroids.
\newblock {\em Adv. Math.}, 154(2):299--332, 2000.

\bibitem[McC02]{McCrudden02}
Paddy McCrudden.
\newblock Tannaka duality for {M}aschkean categories.
\newblock {\em J. Pure Appl. Algebra}, 168(2-3):265--307, 2002.
\newblock Category theory 1999 (Coimbra).

\bibitem[Mel06]{Mellies06}
Paul-Andr{\'e} Melli{\`e}s.
\newblock Functorial boxes in string diagrams.
\newblock In {\em Computer science logic}, volume 4207 of {\em Lecture Notes in
  Comput. Sci.}, pages 1--30. Springer, Berlin, 2006.

\bibitem[MS10]{McCurdyStreet10}
Micah~Blake McCurdy and Ross Street.
\newblock What separable {F}robenius monoidal functors preserve.
\newblock {\em Cah. Topol. G\'eom. Diff\'er. Cat\'eg.}, 51(1):29--50, 2010.

\bibitem[Nil99]{Nill99}
Florian Nill.
\newblock Axioms for weak bialgebras.
\newblock arXiv:math/9805104v1, 1999.

\bibitem[Pfe09]{Pfeiffer09}
Hendryk Pfeiffer.
\newblock Tannaka-{K}re\u\i n reconstruction and a characterization of modular
  tensor categories.
\newblock {\em J. Algebra}, 321(12):3714--3763, 2009.

\bibitem[PS09]{PastroStreet09}
Craig Pastro and Ross Street.
\newblock Weak {H}opf monoids in braided monoidal categories.
\newblock {\em Algebra Number Theory}, 3(2):149--207, 2009.

\bibitem[RT90]{ReshetikhinTuraev90}
N.~Yu. Reshetikhin and V.~G. Turaev.
\newblock Ribbon graphs and their invariants derived from quantum groups.
\newblock {\em Comm. Math. Phys.}, 127(1):1--26, 1990.

\bibitem[Sch92]{Schauenburg92}
Peter Schauenburg.
\newblock {\em Tannaka duality for arbitrary {H}opf algebras}, volume~66 of
  {\em Algebra Berichte [Algebra Reports]}.
\newblock Verlag Reinhard Fischer, Munich, 1992.

\bibitem[Sch03]{Schauenburg03}
Peter Schauenburg.
\newblock Weak {H}opf algebras and quantum groupoids.
\newblock In {\em Noncommutative geometry and quantum groups ({W}arsaw, 2001)},
  volume~61 of {\em Banach Center Publ.}, pages 171--188. Polish Acad. Sci.,
  Warsaw, 2003.

\bibitem[Sch09]{Schaeppi09}
Daniel Sch{\"a}ppi.
\newblock Tannaka duality for comonoids in cosmoi.
\newblock arXiv:0911.0977v1, 2009.

\bibitem[Str07]{Street07}
Ross Street.
\newblock {\em Quantum groups}, volume~19 of {\em Australian Mathematical
  Society Lecture Series}.
\newblock Cambridge University Press, Cambridge, 2007.
\newblock A path to current algebra.

\bibitem[Szl03]{Szlachanyi03}
Korn{\'e}l Szlach{\'a}nyi.
\newblock Galois actions by finite quantum groupoids.
\newblock In {\em Locally compact quantum groups and groupoids ({S}trasbourg,
  2002)}, volume~2 of {\em IRMA Lect. Math. Theor. Phys.}, pages 105--125. de
  Gruyter, Berlin, 2003.

\bibitem[Szl05]{Szlachanyi05}
Korn{\'e}l Szlach{\'a}nyi.
\newblock Adjointable monoidal functors and quantum groupoids.
\newblock In {\em Hopf algebras in noncommutative geometry and physics}, volume
  239 of {\em Lecture Notes in Pure and Appl. Math.}, pages 291--307. Dekker,
  New York, 2005.

\bibitem[Tan38]{Tannaka38}
Tadao Tannaka.
\newblock {\"U}ber den dualit\"atssatz der nichtkommutativen topologischen
  gruppen.
\newblock {\em Tohoku Math. J. (1st)}, 45:1--12, 1938.

\bibitem[Ulb90]{Ulbrich90}
K.-H. Ulbrich.
\newblock On {H}opf algebras and rigid monoidal categories.
\newblock {\em Israel J. Math.}, 72(1-2):252--256, 1990.
\newblock Hopf algebras.

\end{thebibliography}

\end{document}